\documentclass[11pt,final]{article}

\usepackage[T1]{fontenc}
\usepackage{amsmath, amscd, amsfonts, amssymb, amsthm, mathtools}
\usepackage{graphicx, tikz}
\usepackage{dsfont, bbm}
\usepackage{enumerate}
\usepackage[title]{appendix}
\usepackage{hyperref}
\usepackage{fullpage}
\usepackage{kotex}
\usepackage{showkeys}
\usepackage{authblk}
\usepackage{accents}

\theoremstyle{definition}
\newtheorem{definition}{Definition}[section]
\newtheorem{assumption}[definition]{Assumption}

\theoremstyle{plain}
\newtheorem{theorem}[definition]{Theorem}
\newtheorem{lemma}[definition]{Lemma}
\newtheorem{corollary}[definition]{Corollary}
\newtheorem{proposition}[definition]{Proposition}

\newtheorem{claim}[definition]{Claim}

\newtheorem{example}[definition]{Example}

\theoremstyle{remark}
\newtheorem{remark}[definition]{Remark}

\numberwithin{equation}{section}

\newcommand{\prob}{\mathbb{P}}

\newcommand{\R}{\mathbb{R}}

\newcommand{\E}{\mathbb{E}}
\newcommand{\eps}{\epsilon}

\newcommand{\nn}{\nonumber}

\renewcommand{\Im}{\mathrm{Im}\,}
\renewcommand{\Re}{\mathrm{Re}\,}

\DeclareMathOperator{\Tr}{Tr}

\newcommand{\T}{\mathsf{T}}

\newcommand{\triplenorm}[1]{{\left|\!\left|\!\left| #1 \right|\!\right|\!\right|}}
\newcommand{\triplenorms}[1]{{|\!|\!| #1 |\!|\!|}}

\newcommand{\bdot}[1]{\accentset{\scriptscriptstyle\bullet}{#1}}

\newcommand{\email}[1]{\href{mailto:#1}{\texttt{#1}}}

\newcommand{\subjclass}[1]{
  \medskip
  \noindent{\scriptsize 2020 Mathematics Subject Classification:} {\scriptsize #1}\par
}

\newcommand{\keywords}[1]{
  \noindent{\scriptsize Keywords and phrases:} {\scriptsize #1}\par
}

\let\oldthebibliography\thebibliography
\let\endoldthebibliography\endthebibliography

\renewenvironment{thebibliography}[1]
  {
    \oldthebibliography{#1}
    \footnotesize
    \setlength{\itemsep}{1pt}
    \setlength{\parskip}{0pt}
    \setlength{\parsep}{0pt}
  }
  {
    \endoldthebibliography
  }

\begin{document}

\title{Mesoscopic eigenvalue statistics for correlated random matrices} 

\author{
L\'aszl\'o~Erd\H{o}s\textsuperscript{*}
\qquad
Jaehun~Lee\textsuperscript{*}
}

\date{\today}

\maketitle

\begingroup
\renewcommand{\thefootnote}{\fnsymbol{footnote}}
\footnotetext[1]{
ISTA, Am Campus 1, 3400 Klosterneuburg, Austria.
Email: \email{laszlo.erdoes@ist.ac.at}, \email{jaehun.lee@ist.ac.at}.
Supported by ERC Advanced Grant ``RMTBeyond'' No.\,101020331. 
}
\endgroup

\begin{abstract}
	We prove a mesoscopic central limit theorem for linear eigenvalue statistics of correlated Hermitian random matrices. The class considered here includes Wigner and Wigner-type matrices, as well as models whose entry correlations decay polynomially in the distance between index pairs.
The proof combines a multivariate cumulant expansion with multi-resolvent local laws and a detailed analysis of the resulting variance kernel on the operator-level.
\end{abstract}

\subjclass{60B20, 15B52.}

\keywords{Central limit theorem, universality, matrix Dyson equation, multi-resolvent local law.}

\section{Introduction}

Let $H=(H_{ij})$ be an $N\times N$ Hermitian random matrix, and let
$\lambda_1,\ldots,\lambda_N$ denote its eigenvalues. For a compactly supported
and sufficiently regular test function $g:\mathbb{R}\to\mathbb{R}$, we consider
the linear eigenvalue statistics
\begin{equation}\label{eq: 93f}
        \Tr g \bigg(\frac{H-E}{\eta}\bigg)
        =
        \sum_{i=1}^N g\bigg(\frac{\lambda_i-E}{\eta}\bigg)
\end{equation}
around  a fixed reference energy $E\in\mathbb{R}$ and
on an $N$-dependent spectral scale $\eta=\eta(N)>0$.  In this paper we focus on
the case where $E$ lies in the bulk of the limiting spectrum\footnote{Roughly speaking, the bulk means the interior of the support of the limiting spectral distribution away from the edges or other singularities. For the precise definition, see Theorem \ref{thm: 96}.} and $\eta$ is a
mesoscopic scale inside the bulk, that is,
\begin{equation*}
        N^{-1}\ll \eta \ll 1.
\end{equation*}
We prove that, for $H$ belonging to a large class of \emph{correlated} random
matrices, the centered linear eigenvalue statistics converges in distribution
to a centered Gaussian random variable with variance
$(2\beta\pi^2)^{-1}\lVert g\rVert_{\dot{H}^{1/2}}^2$, where
\begin{equation}\label{eq: 186c}
        \lVert g\rVert_{\dot{H}^{1/2}}
        \coloneqq
        \bigg(
        \iint_{\mathbb{R}^2}
        \frac{\bigl(g(x)-g(y)\bigr)^2}{(x-y)^2}
        \,\mathrm{d}x\,\mathrm{d}y
        \bigg)^{1/2}.
\end{equation}
Here the parameter $\beta$ encodes the symmetry class: $\beta=1$ for real
symmetric matrices and $\beta=2$ for complex Hermitian matrices. 
Note that the variance depends only on the $\dot{H}^{1/2}$-norm of the test function (modulo $\beta$) and is otherwise independent of the other details of the model.
Our correlated random matrix model allows for dependence among the entries,
subject to certain summability conditions on their correlations.
These assumptions are designed as a natural extension of the Wigner (and even more generally the Wigner-type) setting,
where the entries are independent up to the symmetry constraint.

For orientation, let us present a prototypical example of our model, previously considered in
\cite{EHR25,MR3941370,MR4480831}. This example is characterized by polynomial
decay of correlations in the sense that
\begin{equation*}
        \big| N \big(\E[H_{i_1j_1}H_{i_2j_2}]-\E[H_{i_1j_1}]\E[H_{i_2j_2}]\big)\big|
        \lesssim
        \frac{1}{
        1+\mathrm{d}\bigl(i_1 j_1, i_2 j_2\bigr)^s
        },
        \qquad s>2,
\end{equation*}
with appropriate matching assumptions for higher multivariate cumulants.
Here the distance between index pairs is defined by
\begin{equation}\label{eq: 179ft}
        \mathrm{d}\bigl(i_1 j_1, i_2 j_2\bigr)
        \coloneqq
        \min
        \left\{
        |i_1-i_2|+|j_1-j_2|,
        |i_1-j_2|+|j_1-i_2|
        \right\}.
\end{equation}
See Example \ref{eg: 275} below for the details of this prototype.
We remark that Wigner and Wigner-type matrices are
included as special cases. In these models, the covariance structure is
supported only on the direct pairing $(ij,ij)$ and the cross pairing
$(ij,ji)$; all other covariances between distinct index pairs vanish.

The problem considered in this paper lies at the intersection of two lines of
research. The first  one is the mesoscopic central limit theorem (CLT) for linear
eigenvalue statistics, a topic with a substantial literature. For random matrix
models with independent entries, such results are by now well understood; see,
for instance,
\cite{MR3678478,MR3914908,LS20,MR4255226,LLS24,Vova25a}. Once dependence
among the entries is allowed, however, new difficulties arise.
CLTs for eigenvalue statistics of random matrices with
dependent entries are comparatively scarce and
existing results in this direction mostly concern the macroscopic scale $\eta\sim 1$.
In \cite{MR4893421}, Kr\"uger and Nemish established a mesoscopic CLT for block correlated random matrices, in which each entry is potentially dependent on $O(1)$ other entries.
The model considered in the present paper allows each entry to depend on a number of other entries that grows with the matrix dimension $N$.
This more general dependence structure is a key distinction from \cite{MR4893421} and introduces substantial difficulties.
We defer a discussion of related results to Section \ref{sec: 449a}.

The second line of research concerns the spectral statistics of correlated
random matrix models. There is an extensive body of work on global spectral
distributions for matrices with correlated entries; see, for example,
\cite{MR2417889,MR3332852,MR1431189,MR4705168,MR1887675,MR2191967,NEURIPS2021_29d74915,MR2444540,
MR2155229}. Local eigenvalue statistics were studied on the microscopic scale in \cite{MR3916109,MR3941370}, where bulk universality was established under certain correlation decay assumptions; fast (exponential) decay for \cite{MR3916109} and slow (polynomial) decay for \cite{MR3941370}. This was later refined and generalized in
\cite{EHR25}, where local laws near regular edges and cusps were proved, and so were corresponding universalities on the microscopic scale.
From this perspective, it is natural to ask how eigenvalues behave on the mesoscopic scales, and whether the universal Gaussian limit persists in this case.
The present work settles this natural question at the intersection of these two research directions.

We focus on the mesoscopic scale rather than the macroscopic one because, for general correlated random matrix models, macroscopic fluctuations can exhibit model-specific features that are non-universal.
Indeed, on the macroscopic scale the limiting fluctuations may be non-Gaussian,
or they may remain Gaussian but with a variance depending on the specific model;
see, for example, \cite[Theorems 35 and 39]{MR3078290}.
By contrast, the mesoscopic regime is closer in spirit to universality of local eigenvalue statistics.
Following Wigner-Dyson-Mehta vision in an extended sense, one expects such non-global eigenvalue behavior to depend only on the symmetry class of the ensemble,
and not on the finer details of the underlying distribution or even its correlation structure.
This motivates the present paper:
despite the more complicated global structure and the lack of entrywise independence, the mesoscopic CLT with its $\dot{H}^{1/2}$-norm in the variance remains universal under suitable assumptions on the correlations.

We now describe the main ideas and new technical ingredients of the proof.
Through the Helffer--Sj\"ostrand representation, the study of the mesoscopic CLT reduces to the analysis of traces of operators associated with the \emph{stability operator} \eqref{eq: 535}.
In the independent-entry setting, such as Wigner-type matrices, the fluctuation analysis can be organized through \emph{entrywise} cumulant expansions and deterministic approximations coming from the \emph{vector Dyson equation} \eqref{eq: 528ft}.
In the correlated setting, this entrywise approach is no longer sufficient, because the second and higher cumulants couple many different matrix entries.

One of the key contributions of this paper is to formulate and carry out the corresponding \emph{multivariate} cumulant expansion at the operator level; see Lemmas \ref{lem: 393}--\ref{lem: 445}.
The covariance and higher cumulants are encoded as operators acting on spaces of matrices, and the multivariate cumulant expansion, together with the \emph{matrix Dyson equation} \eqref{eq: 197}, yields self-consistent relations in this operator language.
This is the first point at which the correlated model differs substantially from the Wigner-type setting: the cumulant expansion is no longer local in the entry indices, and the resulting terms must be analyzed using the full structure of the correlation encoded in these operators.

To derive the leading term through the multivariate cumulant expansion, we need a key input, namely the \emph{two-resolvent local law}, which is the simplest case of multi-resolvent local laws; see Theorem \ref{thm: 766}.
The multi-resolvent local law is an estimate controlling products of several resolvents and deterministic matrices, uniformly down to all mesoscopic scales.
These estimates are supplied by the work \cite{EHR26+}.
Even with these local laws at hand, the derivation of the mesoscopic CLT from them is far from straightforward.
The variance computation requires not only a careful analysis of the self-consistent relations produced by the cumulant expansion, but also new estimates of the trace of an operator involving the stability operator; see Section \ref{sec: 2383}.

The main difficulty in the variance computation is that all leading terms must be treated at the operator-level (by considering operator traces) rather than the matrix-level.
In particular, the limiting variance is governed by a two-body stability operator \eqref{eq: 535}.
The required stability estimates are borrowed from \cite{EHR26+}, and the present work identifies how this stability enters the variance calculation and how the universal result emerges from the traces of operators.
The central analytic task is to extract, in the bulk regime, the singular part of the variance kernel. After a subtle
cancellation, this singular part produces precisely the factor $(x-y)^{-2}$ in \eqref{eq: 186c} giving the $\dot{H}^{1/2}$-norm, while the remaining model-dependent terms become negligible and do not contribute to the limiting variance.

\subsection{Results} \label{sec: 445a}

Before introducing the model, we mention that the notational conventions used throughout the paper are collected in Section \ref{sec: 364}.
Let $H=(H_{ij})$ be an $N\times N$ matrix such that
\begin{equation*}
	H = A + W,
\end{equation*}
where $A=(A_{ij})\in\mathbb{C}^{N\times N}$ is a Hermitian deterministic matrix and $W=(W_{ij})\in\mathbb{C}^{N\times N}$ is a centered Hermitian random matrix.
We consider $H$ satisfying Assumption \ref{assump: 118}--\ref{assump: 176} below.

\begin{assumption}[Bounded expectation]\label{assump: 118}
There exists a constant $C_{A}>0$ independent of $N$ satisfying $\lVert A \rVert\le C_{A}$.
\end{assumption}

\begin{assumption}[Finite moments]\label{assump: 180}
For every positive integer $p\ge 2$, there exists a constant $\nu_{p}$ such that $\E|\sqrt{N}W_{\alpha}|^{p}\le \nu_{p}$ for all $\alpha\in[N]^{2}$ and $N\in\mathbb{N}$.
If $W$ has complex-valued entries, then we assume that
\begin{equation*}
	W = X + \mathrm{i}Y,
\end{equation*}
where $X$ is real symmetric, $\{Y_{ij}\}_{i< j}$ are independent copies of $\{X_{ij}\}_{i< j}$, $Y_{ii}=0$ for all $i$, and $Y^{\mathsf{T}}=-Y$.
\end{assumption}

As in \cite{EHR25}, we introduce the normalized cumulant\,\footnote{For a random vector $\mathbf{w}=(w_{1},\dots,w_{k})$, its joint cumulants, $\kappa_{\mathbf{m}}^{(\mathbf{w})}$ with $\mathbf{m}\in\mathbb{N}_{0}^{k}$, are given as the coefficients of the log-characteristic function
\begin{equation*}
	\log\E\big[e^{\mathrm{i}\mathbf{w}\cdot\mathbf{t}}\big]
	 = \sum_{\mathbf{m}} \kappa_{\mathbf{m}}^{(\mathbf{w})} \frac{(i\mathbf{t})^{\mathbf{m}}}{\mathbf{m}!}.
\end{equation*}
We set $\kappa(\alpha_{1},\dots,\alpha_{k}) \equiv \kappa(\sqrt{N}W_{\alpha_{1}},\dots,\sqrt{N}W_{\alpha_{k}})\coloneqq \kappa_{(1,\cdots,1)}^{(\mathbf{w})}$ with $\mathbf{w}=(\sqrt{N}W_{\alpha_{1}},\dots,\sqrt{N}W_{\alpha_{k}})$.
}
\begin{equation*}
	\kappa(\alpha_{1},\dots,\alpha_{k}) \equiv \kappa(\sqrt{N}W_{\alpha_{1}},\dots,\sqrt{N}W_{\alpha_{k}}),
	\quad \alpha_{i}\in [N]^{2}, \quad i\in \{1,\ldots, k\}.
\end{equation*}
For example, $\kappa(\alpha_{1},\alpha_{2})=N\,\E[W_{\alpha_{1}}W_{\alpha_{2}}]$.
Note that $\kappa(\alpha_{1},\dots,\alpha_{k})$ is invariant under permutations of its arguments.
Now we formulate the assumptions on $\kappa$ used in this paper. They express certain summability properties
of the correlations among the matrix elements of $W$; most of them have already appeared in 
previous works on random matrices with general correlations in \cite{EHR25, MR3941370}.
We formulate them in the most general form needed for our proof, but the reader may skip them and jump directly to 
 Example~\ref{eg: 275} illustrating Assumption~\ref{assump: 140} and serving as our motivation.
For the sake of conciseness, Assumption~\ref{assump: 140} is formulated for the case where $W$ is real symmetric. In the complex Hermitian case, we further assume that the cumulant conditions in Assumption~\ref{assump: 140} hold for all choices of real and imaginary parts in each argument of the cumulant; that is,
\begin{equation*}
        \kappa(\alpha_{1}^{\mathfrak{X}_{1}},\dots,\alpha_{k}^{\mathfrak{X}_{k}})
        \equiv
        \kappa(\sqrt{N}\mathfrak{X}_{1} W_{\alpha_{1}},\dots,\sqrt{N}\mathfrak{X}_{k} W_{\alpha_{k}}),
        \qquad
        \mathfrak{X}_{i}\in \{\mathrm{Re},\mathrm{Im}\},
\end{equation*}
as in \cite{EHR25}; see also \cite[Appendix C]{MR3941370} for details.

\begin{assumption}[Correlation structure]\label{assump: 140}
\text{ }
\begin{enumerate}
	\item[(i)] For every $k\ge 2$, there exists a constant $C_{k}$ such that\,\footnote{
	Due to the invariance of $\kappa(\alpha_{1},\dots,\alpha_{k})$ for permutations of its arguments, the norm $\triplenorm{\kappa}_{k}$ can be defined with respect to any two index pairs, for example, $\triplenorm{\kappa}_{k} = \Big\lVert \sum_{\alpha_{3},\dots,\alpha_{k}} |\kappa(\ast,\ast,\alpha_{3},\dots,\alpha_{k})| \Big\lVert.$
	}
	\begin{equation*}
		\triplenorm{\kappa}_{k} \coloneqq \Bigg\lVert \sum_{\alpha_{1},\dots,\alpha_{k-2}} |\kappa(\alpha_{1},\dots,\alpha_{k-2},\ast,\ast)|\, \Bigg\lVert \le C_{k},
	\end{equation*}
	where we use the notation $\lVert X(\ast,\ast) \rVert$ to denote the operator norm of the $N^{2}\times N^{2}$ matrix with entries $X(\alpha,\alpha')$ for each $\alpha,\alpha'\in[N]^{2}$.
	
	\item[(ii)] The second order cumulant $\kappa(\cdot,\cdot)$ admits a decomposition\,\footnote{The subscripts in the symbols, $\kappa_{c}$ and $\kappa_{d}$, stands for ``cross'' and ``direct'', inspired by the correlation structure of Wigner matrices; see \cite[Remark 2.8]{MR3941370} for detail in that context.}
	\begin{equation*}
		\kappa(\alpha_{1},\alpha_{2})=\kappa_{c}(\alpha_{1},\alpha_{2})+\kappa_{d}(\alpha_{1},\alpha_{2}),
	\end{equation*}
	satisfying the following conditions.
    There exist constants $\{C_{r,s}>0\}_{|r|,|s|\le N/2}$ and a constant $C_{2}'$ such that
    
    \begin{equation}\label{eq: 241b}
	    \max_{a,x} |\kappa_{c}(xa,(a+s)(x+r))| \le C_{r,s}, \qquad \sum_{r,s} C_{r,s} \le C_{2}'.
    \end{equation}
    
    Similarly, there exist constants $\{D_{r,s}>0\}_{|r|,|s|\le N/2}$ and a constant $C_{2}''$ such that
    
    \begin{equation}\label{eq: 245b}
    	\max_{a,x} |\kappa_{d}(xa,(x+r)(a+s))| \le D_{r,s}, \qquad \sum_{r,s} D_{r,s} \le C_{2}''.
    \end{equation}
    
    In the above formulas, all shifted indices, $(x+r)$ and $(a+s)$, are understood modulo $N$, with representatives chosen in $[N]$. 
    
    We further impose that
    \begin{equation}\label{eq: 237}
	    \kappa_{d}(aj,bk) = \kappa_{c}(aj,kb).
    \end{equation}
    If $W$ is complex Hermitian, the condition \eqref{eq: 237} is not required.
	
	\item[(iii)] There exists a constant $C_{3}'$ such that
	\begin{equation*}
		N^{-3/2}\sup_{\substack{X,Y,Z\in\mathbb{C}^{N\times N}\\
	\lVert X\rVert,\lVert Y\rVert\le 1, \lVert Z\rVert_{\text{hs}}\le 1}}
	\sum_{a_{1},b_{1}}\sum_{a_{2},b_{2}}\sum_{a_{3},b_{3}}|\kappa(a_{1}b_{1},a_{2}b_{2},a_{3}b_{3})|
	|X_{b_{1}a_{2}}| |Y_{b_{2}a_{3}}| |Z_{b_{3}a_{1}}| \le C_{3}'.
	\end{equation*}
	
	\item[(iv)] There exists a (small) constant $\mu>0$, such that for every $\alpha\in[N]^{2}$, there exists an index set $\mathcal{N}(\alpha)$ of cardinality $|\mathcal{N}(\alpha)|\le N^{\frac{1}{2}-\mu}$ with the property that $W_{\alpha}$ and $W_{\beta}$ are independent for all $\beta\notin\mathcal{N}(\alpha)$.
\end{enumerate}
\end{assumption}

In Assumption \ref{assump: 140}, the condition (i) naturally arise as a mild summability condition of multivariate cumulants; see \cite[Assumption C]{MR3941370}.
The other conditions (ii)--(iv) are introduced for technical reasons.
The condition (iv) roughly means that each entry $W_{\alpha}$ is possibly dependent on at most $N^{\frac{1}{2}-\mu}$ other matrix entries $W_{\beta}$ $(\beta\in[N]^{2})$, and independent of the rest. This condition is needed 
only for truncating the cumulant expansion; see \eqref{eq: 1796} for more details.

Compared with \cite{EHR25}, the conditions \eqref{eq: 241b}--\eqref{eq: 245b} and \eqref{eq: 237} are newly added in this paper.
We will use \eqref{eq: 241b} and \eqref{eq: 245b} to control the trace error terms that do not appear in \cite{EHR25}; for example, see \eqref{eq: 3257a} and \eqref{eq: 3271a} for their usage in our argument.
Note that \eqref{eq: 241b} and \eqref{eq: 245b} implies the weaker condition \cite[Eq.\,(2.7)]{EHR25}.
The other new condition \eqref{eq: 237} may be viewed as a natural extension of the covariance structure of 
real symmetric Wigner matrices. They have the natural decomposition $\kappa(\cdot,\cdot)=\kappa_d (\cdot,\cdot) + \kappa_c (\cdot,\cdot)$ where $\kappa_{c}(aj,bk) = \delta_{ak}\delta_{jb}$ and $\kappa_{d}(aj,bk) = \delta_{ab}\delta_{jk}$,
and hence \eqref{eq: 237} holds for them.

\begin{assumption}[Fullness]\label{assump: 176}
There exists a constant $c_{\text{full}}>0$ such that
\begin{equation*}
	N\E[|\Tr(WX)|^{2}] \ge c_{\text{full}}\Tr(X^{2}),
\end{equation*}
for any deterministic Hermitian matrix $X\in\mathbb{C}^{N\times N}$. 
\end{assumption}

Assumptions \ref{assump: 118}--\ref{assump: 176} also imply the following useful condition called \emph{flatness}:
there exist constants $C_{\text{flat}},c_{\text{flat}}>0$ such that for any positive semi-definite matrix $X$,
\begin{equation}\label{eq: 294a}
	c_{\text{flat}}\langle X\rangle \le \mathcal{S}[X] \le C_{\text{flat}}\langle X\rangle.
\end{equation}
We shall refer to the constants in Assumptions \ref{assump: 118}--\ref{assump: 176} as \emph{model parameters}.

For the reader's convenience we now demonstrate a prototype example satisfying Assumption
 \ref{assump: 140} (i)--(iii).

\begin{example}[Polynomially decaying correlation structure, {\cite{EHR25,MR3941370,MR4480831}}]\label{eg: 275}
For the second order cumulant, there exists a constant $C_{\textnormal{eg}}>0$ such that for a fixed exponent $s>2$,
\begin{equation*}
	|\kappa(a_{1}b_{1},a_{2}b_{2})| \le \frac{C_{\textnormal{eg}}}{1+\mathrm{d}(a_{1}b_{1},a_{2}b_{2})^{s}},
\end{equation*}
where the distance $\mathrm{d}$ is defined in \eqref{eq: 179ft}.
For cumulants of order $k\ge 3$, we further assume
\begin{equation*}
	|\kappa(\alpha_{1},\dots,\alpha_{k})|\le C_{\textnormal{eg}}\prod_{e\in\mathfrak{T}_{\min}(\alpha_{1},\dots,\alpha_{k})} \frac{1}{1+\mathrm{d}(e)^{s}},
\end{equation*}
where $\mathfrak{T}_{\min}(\alpha_{1},\dots,\alpha_{k})$ is a minimal spanning tree, i.e., a spanning tree for which the sum of the edge distances $\mathrm{d}(e)$ is minimal in a complete graph with the vertex set $\{\alpha_{1},\dots ,\alpha_{k}\}$.
According to \cite[Example 2.10]{MR3941370}, this example satisfies Assumption \ref{assump: 140} (i).
In addition, the validity of Assumption \ref{assump: 140} (iii) is shown in \cite[Appendix B]{EHR25}.
One can easily check Assumption \ref{assump: 140} (ii) by setting
\begin{equation*}
	\kappa_{c}(a_{1}b_{1},a_{2}b_{2}) =
	\begin{cases}
		\kappa(a_{1}b_{1},a_{2}b_{2}), & |a_{1}-b_{2}|+|b_{1}-a_{2}| < |a_{1}-a_{2}|+|b_{1}-b_{2}|, \\
		0, & |a_{1}-b_{2}|+|b_{1}-a_{2}| \ge |a_{1}-a_{2}|+|b_{1}-b_{2}|,
	\end{cases}
\end{equation*}
and similarly,
\begin{equation*}
	\kappa_{d}(a_{1}b_{1},a_{2}b_{2}) =
	\begin{cases}
		0, & |a_{1}-b_{2}|+|b_{1}-a_{2}| < |a_{1}-a_{2}|+|b_{1}-b_{2}|, \\
		\kappa(a_{1}b_{1},a_{2}b_{2}), & |a_{1}-b_{2}|+|b_{1}-a_{2}| \ge |a_{1}-a_{2}|+|b_{1}-b_{2}|.
	\end{cases}
\end{equation*}
\end{example}

Our main result holds in the bulk of the spectrum of $H$. First, in order to define this concept properly, we
need to introduce the \emph{self-consistent density of states}. 
The \emph{self-energy operator}
$\mathcal{S}$ is defined by, for every deterministic matrix $X\in\mathbb{C}^{N\times N}$,
\begin{equation}\label{eq: 258}
	\mathcal{S}[X] \coloneqq \E[WXW].
\end{equation}
For $z\in\mathbb{C}\backslash\mathbb{R}$, let $M\equiv M(z)$ be the solution of the following \emph{matrix Dyson equation} (MDE):
\begin{equation}\label{eq: 197}
	-M^{-1}(z) = z - A + \mathcal{S}[M(z)], \quad \Im z \cdot \Im M(z) > 0, 
\end{equation}
where $\Im M=(M-M^{*})/(2\mathrm{i})$.
It is well-known that \eqref{eq: 197} has a unique solution
and that the following limit exists
\begin{equation}\label{eq: 180}
	\rho(E) \coloneqq \lim_{\eta\downarrow 0}\frac{1}{\pi} \langle \Im M(E+\mathrm{i}\eta)\rangle ,
\end{equation}
defining the \emph{self-consistent density of states} $\rho(E)$; we refer to \cite{MR4031100}.
Now we can state our main result.

\begin{theorem}[Mesoscopic CLT in the bulk]\label{thm: 96}
Consider $H=A+W\in\mathbb{C}^{N\times N}$ satisfying Assumptions \ref{assump: 118}--\ref{assump: 176}. 
Let $g\in C_{\textnormal{c}}^{2}(\mathbb{R})$. Fix a small constant $\eps_{0}\in(0,1/2)$, independent of $N$. Consider $\eta_{0}=\eta_{0}(N)\in(N^{-1+\eps_{0}},N^{-\eps_{0}})$.
Take a fixed reference energy $E_{0}>0$ in the bulk, i.e.\,$\rho(E_{0})\ge\eps_{0}$ where $\rho$ is defined in \eqref{eq: 180}.
We define the mesoscopically scaled test function $f$ by setting
\begin{equation}\label{eq: 192}
	f(x) \coloneqq g\Big(\frac{x-E_{0}}{\eta_{0}}\Big).
\end{equation}
Let $\mathsf{N}(\mu,\sigma^{2})$ denote a Gaussian random variable of mean $\mu$ and variance $\sigma^{2}$.
Then we have the following convergence in distribution:
\begin{equation}\label{eq: 104}
	\Tr f(H) - \E[\Tr f(H)] \overset{\text{d}}{\longrightarrow} \mathsf{N}\Big(0,\frac{1}{2\beta\pi^{2}}\lVert g\rVert_{\dot{H}^{1/2}}^{2}\Big),
\end{equation}
where $\beta=1$ and $\beta=2$ corresponds to real symmetric and complex Hermitian $H$, respectively, and
the norm $\lVert g\rVert_{\dot{H}^{1/2}}$ is  defined in \eqref{eq: 186c}.
\end{theorem}

\begin{remark}
If one is interested in the macroscopic CLT, corresponding to the spectral scale $\eta_{0}\sim 1$, it may be possible to approach it by studying the joint moments of resolvent traces (together with Wick's theorem
type argument), as in \cite{CES23}. 
In this paper, however, we do not pursue this direction. For general correlated random matrices, the macroscopic CLT may exhibit model-dependent features, as does the global limiting spectral distribution. Our focus is instead on universal phenomena, such as local eigenvalue behavior, which are insensitive to the finer details of the model.
\end{remark}

\subsection{Related works} \label{sec: 449a}

For linear eigenvalue statistics of the form \eqref{eq: 93f}, first the macroscopic scale $\eta\sim 1$ was studied.
Arharov \cite{MR309171} considered the Wishart--Laguerre ensemble and addressed the joint asymptotic normality of normalized traces of its powers; this result was later reformulated and proved by Jonsson \cite{MR650926}.
However, the variance of the limiting Gaussian distribution was not explicitly identified in either \cite{MR309171} or \cite{MR650926}.
For traces of resolvents, Girko provided macroscopic CLTs for Wigner and sample covariance matrices, including formulas for the variance of the Gaussian limit; see \cite{MR1887675, MR1887676}.
This line of results was later developed in a more refined and generalized form by Lytova and Pastur \cite{MR2561434}, and M.\,Shcherbina \cite{MR2829615}.
We remark that the Fourier transform (characteristic function) method was introduced in \cite{MR2561434} as an alternative to the Stieltjes transform and moment methods;
it was also employed in \cite{MR2829615} for further improvements.
Macroscopic CLTs have also been studied in more general settings, for instance under weaker moment assumptions or in the presence of deformations; see, e.g., \cite{MR3568772,MR4119597,MR4923079}.

We now review the literature on the mesoscopic CLT, focusing on the Wigner matrix model and its variants.
In \cite{MR1678012}, Boutet de Monvel and Khorunzhy proved that, for the Gaussian orthogonal ensemble (GOE), the centered statistics $\eta \Tr (H-z)^{-1}$ converges in distribution to a Gaussian random variable throughout the mesoscopic range $\eta \in [N^{-1+\eps},N^{-\eps}]$.
Here $\eta \Tr (H-z)^{-1}$ corresponds to $\Tr g(\frac{H-E}{\eta})$ with $g(x)=(x-\mathrm{i})^{-1}$ and $z=E+\mathrm{i}\eta$.
This result was later extended to Wigner matrices in \cite{MR1689027}, though under several additional restrictions, including the condition $\eta \in [N^{-\frac{1}{8}+\eps},N^{-\eps}]$ on the mesoscopic scale.
Subsequently, for general test functions $g$ satisfying suitable regularity and decay assumptions, the mesoscopic CLT was established throughout the full mesoscopic regime by He--Knowles \cite{MR3678478}; the same result was also shown for $\eta \in [N^{-\frac{1}{3}+\eps},N^{-\eps}]$ by Lodhia--Simm \cite{LS15+}.
Later, the result was further generalized to deformed Wigner matrices and sample covariance matrices in \cite{MR3914908,MR4255183}, to generalized Wigner matrices in \cite{MR4255226}, and to Wigner matrices with more general test functions $g$, without decay assumptions, in \cite{LS20}.

Departing from the Wigner model, one can consider generalized Wigner matrices, as introduced in \cite{MR2871147}.
In this model, $\E H_{ij}=0$, but the variances of the entries, $S_{ij}=\E |H_{ij}|^2$, are not necessarily identical to $1/N$, but the variance profile $S=(S_{ij})$ is assumed to be \emph{doubly stochastic}, i.e.\,$\sum_{i=1} S_{ij}=1$ and $\sum_{j=1} S_{ij}=1$
for all $i,j$, and \emph{flat}, that is, $S_{ij}\sim N^{-1}$ for all $i,j$.
For generalized Wigner matrices, the mesoscopic CLT was obtained by Li and Xu in \cite{MR4255226}.
By dropping the doubly stochastic condition, one arrives at the class of Wigner-type matrices, studied first in \cite{MR3719056}.
In this setting, the variance profile is still assumed to be flat, but it is not necessarily doubly stochastic. Consequently, the limiting spectral distribution remains semicircular for generalized Wigner matrices, but not for Wigner-type matrices in general.
More precisely, the Stieltjes transformation of the limiting spectral distribution of Wigner matrices, and also of generalized Wigner matrices, is described by the scalar self-consistent equation
\begin{equation*}
        -m^{-1}(z) = z+m(z), \quad \Im z \cdot \Im m(z)>0,
\end{equation*}
whereas the corresponding deterministic description for Wigner-type matrices is given by the vector Dyson equation
\begin{equation}\label{eq: 528ft}
        -m_i^{-1}(z) = z+\sum_{j} S_{ij}m_j(z), \quad
        \Im z \cdot \Im m_i(z)>0, \quad i\in[N].
\end{equation}
For Wigner-type matrices, the mesoscopic CLT was recently obtained by Landon, Lopatto and Sosoe \cite{LLS24}, and Riabov \cite{Vova25a}.
More precisely, \cite{LLS24} covers the case where $g$ is a smooth compactly supported function, $g\in C^{\infty}_{\mathrm{c}}$, as well as the case of a smooth step function specifically satisfying $g'(x)=0$ for $|x|>1$, with $g(1)=1$ and $g(-1)=0$.
Later, Riabov \cite{Vova25a} relaxed the regularity assumption and established the result for test functions $g\in C^{2}_{\mathrm{c}}$.

All the mesoscopic-scale results mentioned above concern bulk reference energies $E$, where the limiting density satisfies $\rho(E)\gtrsim 1$.
Near spectral singularities,\footnote{For a general class of random matrices satisfying a flatness condition, including Wigner-type matrices and the correlated models considered in the present paper, such singularities are classified into two types: \emph{regular edge} and \emph{cusp}; we refer to \cite{MR3916109} for detail.}
where $\rho(E)$ becomes small or vanishes,
one can obtain CLTs with different universal variance formulas; we refer to \cite{MR1706781,MR4038239,MR4255183,Vova25b} for related results in simpler models, such as Gaussian ensembles, Wigner matrices, and Wigner-type matrices, all of which have independent entries.

Beyond the independent-entry setting, there are several works on specific random matrix models with dependent entries.
Here we leave aside invariant ensembles, for which powerful computational tools are available;
see \cite{MR1487983,MR1454702,MR1797877,MR2404179,MR2817648,MR3063494} for CLT results for invariant ensembles and related models.
Examples of dependent-entry models include random permutation matrices \cite{MR1813834, MR3335019}, sums of i.i.d. permutation matrices \cite{MR3078290}, and random geometric graphs \cite{HNO26+}.
These works concern macroscopic CLTs, and the corresponding limiting distributions may be non-Gaussian, or Gaussian with model-dependent variances.
Moreover, the methods employed there are elegant but necessarily rather model-specific.

A closely related result is the universal bulk mesoscopic CLT established by Kr\"uger and Nemish in \cite{MR4893421} for matrices with a fixed block correlation structure.
Their model is constructed from a fixed number of independent matrices with i.i.d. entries and, after reordering the coordinates, can be viewed as an array of independent matrix-valued entries of fixed size, up to the Hermitian symmetry.
Consequently, the associated matrix Dyson equation and stability analysis reduce to a matrix space whose dimension is fixed independently of $N$.
Our model does not require such a block representation and allows each entry to depend on a number of other entries that grows with $N$, subject to the cumulant summability and dependence assumptions stated below.
The resulting fluctuation analysis involves operators acting on the full $N\times N$ matrix space.
A principal additional difficulty is to control their operator traces uniformly in $N$, since bounds on operator norms alone may introduce unacceptable dimension factors.
Although the two works share conceptually similar proof scheme, the present paper develops the operator formulation and trace estimates needed for this more general dependence structure.

Finally, we remark that, although our focus is on generalizing the random matrix model itself where the mesoscopic CLT holds in the universal form,
another important direction in the study of CLTs for linear eigenvalue statistics is to weaken or generalize the regularity assumptions on the test functions; see \cite{LS22+} and the references therein.

\subsection{Organization}
The proof of the main result, Theorem \ref{thm: 96}, is given in Section \ref{sec: 444}, where we reduce it to two main
 inputs, Propositions \ref{prop: 147}--\ref{prop: 169}. 
In Section \ref{sec: 552}, we collect several preliminary results that serve as the toolbox for the rest of the paper. 
The remaining sections, Sections \ref{sec: 535}--\ref{sec: 1776}, are devoted to the proofs of
Propositions \ref{prop: 147}--\ref{prop: 169}, respectively.

\subsection{Notation} \label{sec: 364}

For $N \in \mathbb{N}$, we write $[N]\coloneqq\{1,2,\dots,N\}$.
We use Latin letters, such as $i,j,k,\ell$, to denote integer indices. 
Greek letters, such as $\alpha$, denote pairs of integer indices, possibly with subscripts. 
Thus $\alpha$ may be written either as $(i,j)$ or, when no confusion can arise, simply as $ij$; in both cases $\alpha$ is an element of $[N]^2 = [N]\times [N].$

For a vector $\mathbf{v}=(v_{1},\dots,v_{N})\in\mathbb{C}^{N}$, we denote by
$\lVert \mathbf{v} \rVert$ its Euclidean norm.
For each $i,j\in[N]$, let $E^{ij}$ be a standard basis matrix defined by $E^{ij}_{ab}\coloneqq\delta_{ai}\delta_{jb}$.
Let $A=(A_{ij})$ be an $N \times N$ complex matrix.
We write $A^{\mathsf{T}}$ and $A^{*}$ for the transpose and conjugate transpose of $A$, respectively.
We denote by $\lVert A \rVert$
the operator norm of $A$, induced by the Euclidean norm.
We write $\langle A \rangle \coloneqq N^{-1}\Tr(A)$ for the normalized trace of $A$. 
For matrices $A,B\in\mathbb{C}^{N\times N}$, we define the normalized Hilbert--Schmidt
inner product by
\begin{equation}\label{eq: 117}
	\langle A,B \rangle \coloneqq \frac{1}{N} \Tr(A^{*}B),
\end{equation}
and set the normalized Hilbert--Schmidt norm as $\lVert A \rVert_{\mathrm{hs}} = \langle A,A \rangle^{1/2}$.
We introduce the sandwich operator $\mathcal{C}_{A,B}:\mathbb{C}^{N\times N}\to\mathbb{C}^{N\times N}$ by
$\mathcal{C}_{A,B}[X] \coloneqq AXB.$
When $A=B$, we write simply $\mathcal{C}_{A} \equiv \mathcal{C}_{A,A}$.
With respect to the inner product \eqref{eq: 117}, we have the adjoint $\mathcal{C}_{A,B}^{*} = \mathcal{C}_{A^{*},B^{*}}$.
Moreover, if $A$ and $B$ are invertible, then $\mathcal{C}_{A,B}^{-1} = \mathcal{C}_{A^{-1},B^{-1}}$.
For a linear operator $\mathcal{F}:\mathbb{C}^{N\times N}\to\mathbb{C}^{N\times N}$, we denote its induced norms by
\begin{equation*}
	\lVert \mathcal{F}\rVert_{\lVert\cdot\rVert\to\lVert\cdot\rVert} \quad\text{and}\quad
	\lVert \mathcal{F}\rVert_{\lVert\cdot\rVert_{\text{hs}}\to\lVert\cdot\rVert_{\text{hs}}},
\end{equation*}
which are induced by the operator norm and the normalized Hilbert--Schmidt norm on the domain and codomain.
Let $\mathrm{Id} : \mathbb{C}^{N \times N} \to \mathbb{C}^{N \times N}$
denote the identity operator.
To simplify notation, we identify each scalar $\zeta \in \mathbb{C}$ with the
operator $\zeta\,\mathrm{Id}$ whenever no confusion arises.

For two positive quantities $X$ and $Y$,
we write $X\lesssim Y$
if there exists a constant $C>0$ only depending on the model parameters in Assumptions \ref{assump: 118}--\ref{assump: 176} (unless explicitly stated otherwise), such that $X\le CY$.
We also write $X\sim Y$ if $X\lesssim Y$ and $Y\lesssim X$. 
For a (possibly complex) quantity $X$ and a positive quantity $Y$, we use the notation $X=O(Y)$ to indicate $|X|\lesssim Y$.

For two positive random variables $X_N$ and $Y_N$, we write $X_N \prec Y_N$
if, for every (small) $\eps > 0$ and every (large) $D > 0$, there exists $N_0(\eps,D) \in \mathbb{N}$
such that, for all $N \geq N_0(\eps,D)$,
\begin{equation*}
	\prob\{ X_N > N^{\eps} Y_N \}
    \leq
    N^{-D}.
\end{equation*}
If $X_{N}$ and $Y_{N}$ are deterministic, then $X_{N}\prec Y_{N}$ if and only if for any $\eps>0$, $X_{N}\le N^{\eps} Y_{N}$ for sufficiently large $N$.
We write $X_N = O_{\prec}(Y_N)$ to denote $|X_{N}|\prec Y_{N}$ for a (possibly complex) variable $X_{N}$ and a positive variable $Y_{N}$.

\subsection*{Acknowledgment}
We are grateful to Volodymyr Riabov for many helpful discussions.

\section{Proof of Theorem \ref{thm: 96}} \label{sec: 444}

In this section, we will provide the proof of the main result by reducing it to two technical ingredients (Proposition \ref{prop: 147}--\ref{prop: 169}) stated below.
We will use the characteristic function approach to prove the central limit theorem, which goes back to \cite{MR2561434, MR2829615},
but the recent works \cite{MR3914908, LS20, LLS24, Vova25a} are more relevant to our proof details.

Consider the characteristic function $\phi$ of the centered trace:
\begin{equation*}
	\phi(\lambda) \coloneqq \E e^{\mathrm{i}\lambda(\Tr f(H)-\E\Tr f(H))}, \quad \lambda\in\mathbb{R}.
\end{equation*}
In order to get \eqref{eq: 104}, by L\'{e}vy's continuity theorem, it is enough to show that for any fixed $\lambda\in\mathbb{R}$
\begin{equation}\label{eq: 137}
	\phi'(\lambda) = - \lambda\phi(\lambda) V(f) + o(1),
\end{equation}
where
\begin{equation}\label{eq: 141}
	V(f) \coloneqq \frac{1}{2\beta\pi^{2}}\lVert g\rVert_{\dot{H}^{1/2}}^{2},
\end{equation}
and, here and throughout this paper, $o(1)$ denotes a quantity, possibly depending on fixed parameters such as $\lambda$, $g$, $E_{0}$, and $\eps_{0}$, that tends to zero as $N \to \infty$.
Thus, from now on, we focus on showing \eqref{eq: 137} and \eqref{eq: 141},
which are immediate consequences of the following two propositions.

\begin{proposition}\label{prop: 147}
Let $g,\eps_{0}$ and $\eta_{0}$ be as in Theorem \ref{thm: 96}.
Recall the scaled test function $f$ as in \eqref{eq: 192}.
Define the quasi-analytic extension of $f$ by
\begin{equation}\label{eq: 691e}
	\tilde{f}(x+\mathrm{i}\eta) \coloneqq \chi(\eta)(f(x)+\mathrm{i}\eta f'(x)),
\end{equation}
where $\chi:\mathbb{R}\to[0,1]$ is an even $C_{\text{c}}^{\infty}(\mathbb{R})$ function supported on $[-1,1]$ with $\chi(\eta)=1$ for $|\eta|<1/2$.
Let $M(z)$ be the solution of \eqref{eq: 197}.
Introduce the kernel $\mathcal{K}(z,w)$ by setting, for any $z,w\in\mathbb{C}$, 
\begin{equation}\label{eq: 723ft}
	\mathcal{K}(z,w) \coloneqq \frac{2}{\beta}\Tr_{\textnormal{op}}(\mathcal{C}_{M'M^{-1},I}(1-\mathcal{C}_{M,\widetilde{M}}\mathcal{S})^{-1}\mathcal{C}_{M,\widetilde{M}'}\mathcal{S}(1-\mathcal{C}_{M,\widetilde{M}}\mathcal{S})^{-1}),
\end{equation}
where $M\equiv M(z)$ and $\widetilde{M}\equiv M(w)$.
Then, for any fixed $\lambda\in\mathbb{R}$, we have
\begin{equation*}
	\phi'(\lambda) = - \frac{\lambda\phi(\lambda)}{\pi^{2}}
	\int_{\Omega_{0}}\int_{\Omega_{0}'}
	\frac{\partial\tilde{f}(z)}{\partial\bar{z}}\frac{\partial\tilde{f}(w)}{\partial\bar{w}}\mathcal{K}(z,w)
	\mathrm{d}^{2}w \mathrm{d}^{2}z
	+ O_{\prec}\Big( (1+|\lambda|)^{2}N^{-\eps_{0}/2} \Big),
\end{equation*}
where
\begin{equation}\label{eq: 273}
	\Omega_{0} \coloneqq \{z\in\mathbb{C}: |\Im z|> N^{-\eps_{0}/100}\eta_{0}\}, \quad
	\Omega_{0}' \coloneqq \{z\in\mathbb{C}: |\Im z|> 2N^{-\eps_{0}/100}\eta_{0}\}.
\end{equation}
\end{proposition}

\begin{proposition}\label{prop: 169}
Under the same setting as in Proposition \ref{prop: 147}, we have
\begin{equation*}
	\int_{\Omega_{0}}\int_{\Omega_{0}'}\frac{\partial\tilde{f}(z)}{\partial\bar{z}}\frac{\partial\tilde{f}(w)}{\partial\bar{w}}\mathcal{K}(z,w) \mathrm{d}^{2}w\mathrm{d}^{2}z
	= \frac{1}{2\beta}\iint_{\mathbb{R}^{2}}\frac{(f(x)-f(y))^{2}}{(x-y)^{2}}\mathrm{d}x\mathrm{d}y
	+ o(1).
\end{equation*}
\end{proposition}

We  remark that
\begin{equation}\label{eq: 513a}
    \iint_{\mathbb{R}^{2}}\frac{(f(x)-f(y))^{2}}{(x-y)^{2}}\mathrm{d}x\mathrm{d}y
    = \iint_{\mathbb{R}^{2}}\frac{(g(s)-g(t))^{2}}{(s-t)^{2}}\mathrm{d}s\mathrm{d}t
	= \lVert g\rVert_{\dot{H}^{1/2}}^{2}.
\end{equation}
Combining Propositions \ref{prop: 147}--\ref{prop: 169} and \eqref{eq: 513a},
we deduce \eqref{eq: 137} and \eqref{eq: 141}, which completes the proof of Theorem \ref{thm: 96}.
We present the proofs of Proposition \ref{prop: 147}--\ref{prop: 169} in Section \ref{sec: 535} and Section \ref{sec: 1776}, respectively.
Following the convention of \cite{LS20, LLS24, Vova25a}, we restrict our proof to the real symmetric case ($\beta=1$) to make a concise presentation.
The same argument works to the complex Hermitian case ($\beta=2$) with a minor modification, which is omitted for brevity.

\begin{remark}\label{rmk: 760b} 
Besides approaches based on the characteristic function of the linear statistics,
Gaussian fluctuations may also be obtained through moment computations for resolvent traces.
This alternative path was followed, for example, 
 in \cite{MR3678478,CES23,CES24}.
In particular, \cite{CES23,CES24} established CLTs for linear eigenvalue statistics of non-Hermitian random matrices.
Starting with Girko's Hermitization formula, they computed the joint moments of the resolvent traces of the Hermitized matrices and showed that, asymptotically, only Wick pairings contribute.
\end{remark}

\section{Preliminaries} \label{sec: 552}

We collect the preliminary results  from earlier works that form the toolkit for us.

\begin{lemma}[Multivariate cumulant expansion, {\cite[Lemma 3.1]{MR3678478}; \cite[Proposition 3.2]{MR3941370}; \cite[Proposition 5.2]{EHR25}}]\label{lem: 378}
For a positive integer $L$, let $f:\mathbb{R}^{N\times N}\to\mathbb{C}$ be an $L$ times differentiable function with bounded derivatives. Let $W=(W_{\alpha})_{\alpha\in[N]^{2}}$ be an $N\times N$ random matrix, whose normalized cumulants satisfy Assumption \ref{assump: 140}.
We use the notation $\partial_{\alpha}\coloneqq\frac{\partial}{\partial W_{\alpha}}$ for every $\alpha\in [N]^{2}$.
Let $\mathcal{N}(\alpha)$ be as in Assumption \ref{assump: 140} (iv).
Then, for any index $\alpha_{0}\in[N]^{2}$,
\begin{equation*}
	\E[W_{\alpha_{0}}f(W)] = \sum_{k=0}^{L-1} \sum_{\boldsymbol{\alpha}\in\mathcal{N}^{k}(\alpha_{0})}
	\frac{\kappa(\alpha_{0},\boldsymbol{\alpha})}{N^{(k+1)/2}k!}\E[\partial_{\boldsymbol{\alpha}}f(W)] + \mathcal{E}_{L}(f,\alpha_{0}),
\end{equation*}
where $\boldsymbol{\alpha}\coloneqq(\alpha_{1},\dots,\alpha_{k})$ and
$\partial_{\boldsymbol{\alpha}}\coloneqq\partial_{\alpha_{1}}\dots\partial_{\alpha_{k}}$.
Furthermore, the term $\mathcal{E}_{L}(f,\alpha_{0})$ satisfies, for some constant $C_{L}>0$ depending only on $L$,
\begin{equation}\label{eq: 391}
	|\mathcal{E}_{L}(f,\alpha_{0})|
	\lesssim \frac{C_{L}}{N^{(L+1)/2}}\sum_{\boldsymbol{\alpha}\in\mathcal{N}^{L}(\alpha_{0})}
	\sup_{\lambda\in[0,1]} \bigg(\E\bigg[\Big|
	\partial_{\boldsymbol{\alpha}}
	f\big(\lambda W|_{\mathcal{N}(\alpha_{0})}+W|_{[N]^{2}\backslash\mathcal{N}(\alpha_{0})}\big)\Big|^{2}\bigg]\bigg)^{1/2},
\end{equation}
where the notation $W|_{\mathcal{N}}$ for $\mathcal{N}\subset [N]^{2}$ refers to the matrix defined by
\begin{equation*}
	W|_{\mathcal{N}}(\alpha) \coloneqq
	\begin{cases}
		W_{\alpha} & \alpha\in\mathcal{N},\\
		0 & \alpha\notin\mathcal{N}.
	\end{cases}
\end{equation*} 
\end{lemma}

\begin{theorem}[Single resolvent local law in the bulk, {\cite[Theorem 2.8]{EHR25} and \cite[Theorem 2.2]{MR3941370}}]\label{thm: 369}
Consider an $N\times N$ matrix $H=A+W$ satisfying Assumptions \ref{assump: 118}--\ref{assump: 176}.
Let $\eps_{0}$ as in Theorem \ref{thm: 96}.
Let $\rho$ be as in \eqref{eq: 180}.
Define $$\mathcal{I}(\eps)\coloneqq\{E\in\mathbb{R}: \rho(E)\ge\eps\}, \quad \eps>0.$$
Let $M(z)$ be the solution of \eqref{eq: 197}.
Then, uniformly for all $z=E+\mathrm{i}\eta$ with $E\in\mathcal{I}(\eps_{0}/2)$ and $\eta\in[N^{-1+(\eps_{0}/100)},100]$ (where the number $100$ can be replaced with any large positive real), the resolvent $G(z)\coloneqq (H-z)^{-1}$ satisfies the isotropic local law,
\begin{equation}\label{eq: 869e}
	|\mathbf{u}^{*}(G(z)-M(z))\mathbf{v}| \prec \lVert\mathbf{u}\rVert \lVert\mathbf{v}\rVert (N\eta)^{-1/2},
\end{equation}
for any deterministic vectors $\mathbf{u},\mathbf{v}\in\mathbb{C}^{N}$, and the average local law,
\begin{equation}\label{eq: 873e}
	|\langle (G(z)-M(z)) B \rangle| \prec \frac{\lVert B\rVert_{\textnormal{hs}}}{N\eta} ,
\end{equation}
for any deterministic matrix $B\in\mathbb{C}^{N\times N}$.
\end{theorem}

\begin{theorem}[Two-resolvent local law in the bulk, {\cite[Theorem 3.7--3.8]{EHR26+}}]\label{thm: 766}
Suppose the same setting as in Theorem \ref{thm: 369}.
Then, for any deterministic $N\times N$ matrix $A$ and any deterministic vectors $\mathbf{x},\mathbf{y}\in\mathbb{C}^{N}$, we have
\begin{equation}\label{eq: 882d}
	\langle \mathbf{x}, (G(z)AG(w) - (1-\mathcal{C}_{M(z),M(w)}\mathcal{S})^{-1}\mathcal{C}_{M(z),M(w)}[A]) \mathbf{y} \rangle
	\prec \lVert A\rVert\lVert\mathbf{x}\rVert\lVert\mathbf{y}\rVert \Psi(z,w)
\end{equation}
uniformly in all $z=E_{1}+\mathrm{i}\eta_{1}$ and $w=E_{2}+\mathrm{i}\eta_{2}$ satisfying 
$E_{i}\in \mathcal{I}(\eps_{0}/2)$ and $\eta_{i}\in[N^{-1+(\eps_{0}/100)},100]$ ($i=1,2$),
and the term $\Psi$ is defined by
\begin{equation}\label{eq: 928ty}
	\Psi(z,w) = \frac{1}{\min\{\eta_{1},\eta_{2}\}^{1/2}} + \frac{1}{(N\min\{\eta_{1},\eta_{2}\}^{3})^{1/2}}
	+ \frac{1}{N\min\{\eta_{1},\eta_{2}\}^{2}}.
\end{equation}
\end{theorem}
\begin{proof} 
The proof of \eqref{eq: 882d} is based upon the analogous two-resolvent local law with $\Im G$’s
instead of $G$’s  from \cite[Theorem 3.7--3.8]{EHR26+} and an integral representation formula from~\cite{MR4819934}
that expresses $G$ in terms of $\Im G$.
Since the more delicate part is `perturbative spectral regime' 
(following the terminology of \cite{EHR26+}), we focus on how to use \cite[Theorem 3.7]{EHR26+}.
The other regime, based upon \cite[Theorem 3.8]{EHR26+} is much easier.

To apply the local law in terms of $\Im G\coloneqq (G-G^{*})/2\mathrm{i}$, we recall the cone integral representation \cite[Lemma 4.11]{MR4819934}:
\begin{equation*}
	G(z) = \frac{1}{\pi}\int_{\mathbb{R}} \frac{\Im G(\psi(x))}{x-\psi^{-1}(z)}\mathrm{d}x,
\end{equation*}
where $\psi(x)$ is a conformal mapping of the closed upper half plane to a cone, see \cite[Eq.\,(4.75)]{MR4819934}.
Let $\xi=|\Im z| /2$ and $\gamma=1/8$, then $\psi( \mathrm{i}\xi^{1/\gamma}) = z$ and $\psi(0)$ is mapped into the tip of the cone, lying between $z$ and the real line (see~\cite[Fig.2]{MR4819934}). 

Using this representation twice, we write
\begin{equation}\label{eq: 967ty}
	\langle G(z)B_{1}G(w)B_{2}\rangle = \frac{1}{\pi^{2}}\iint_{\R^{2}}\frac{\langle \Im G(\psi(x))B_{1}\Im G(\widetilde\psi(y))B_{2}\rangle}{(x-\psi^{-1}(z))(y-\widetilde\psi^{-1}(w))} \mathrm{d}x\mathrm{d}y,
\end{equation}
where $\widetilde\psi$ is defined analogously for a cone representation with vertex $w$.
Let 
$\tilde{z}\coloneqq\psi(x)$ and $\tilde{w}\coloneqq\widetilde\psi(y)$ for brevity.
Notice that the same representation holds for the corresponding deterministic terms, i.e. 
\begin{equation}\label{eq: 1001ty}
	\langle M(z,B_{1},w)B_{2}\rangle = \frac{1}{\pi^{2}}\iint_{\R^{2}}\frac{\langle 
	\widehat{M}(\tilde{z},B_{1},\tilde{w})B_{2}\rangle}{(x-\psi^{-1}(z))(y-\widetilde\psi^{-1}(w))}\mathrm{d}x\mathrm{d}y,
\end{equation}
where 
$$ M(z,B_{1},w):=
(1-\mathcal{C}_{M(z),M(w)}\mathcal{S})^{-1}\mathcal{C}_{M(z),M(w)}[B_1]) 
$$
is the deterministic approximation to $G(z) B_1 G(w)$ and 
 $\widehat{M}(\tilde{z},B_{1},\tilde{w})$ is defined as a four-fold linear combination of $M(\tilde{z},B_{1},\tilde{w})$’s  with the spectral parameters $\tilde{z}, \tilde{w}$ and their complex conjugates~\cite[Eq.\,(3.16)]{EHR26+}. 
The proof of~\eqref{eq: 1001ty} follows by applying
a standard tensorization procedure (also known as \emph{meta-argument}, e.g.\,\cite[Section 2.6]{MR3513608} and \cite[Appendix D]{MR4748758}): tensorize the original matrix $H\in\mathbb{C}^{N\times N}$ into $\widetilde{H}\in\mathbb{C}^{N\times N}\otimes \mathbb{C}^{K\times K}$ in the canonical way and then take the limit $K\to\infty$.
Applying the global law~\cite[Proposition 5.3]{EHR26+} for fixed $N$ and $K\to \infty$, 
 we deduce \eqref{eq: 1001ty} from \eqref{eq: 967ty}. Combining these two formulas, we have
\begin{equation}\label{eq: 994ty}
	\langle [G(z)B_{1}G(w)-M(z,B_{1},w)]B_{2}\rangle
	= \frac{1}{\pi^{2}}\iint_{\R^{2}}\frac{\langle [\Im G(\tilde{z})B_{1}\Im G(\tilde{w})-\widehat{M}(\tilde{z},B_{1},\tilde{w})]B_{2}\rangle}{(x-\psi^{-1}(z))(y-\widetilde\psi^{-1}(w))} \mathrm{d}x\mathrm{d}y.
\end{equation}

We  use this formula by setting $B_{1}=A$ and $B_{2}=N\mathbf{y}\mathbf{x}^{*}$. 
For this choice and in the regime where $\Re \tilde z, \Re\tilde w$ are in the bulk,  using \cite[Theorem 3.7]{EHR26+}, the numerator can be estimated as follows
\begin{equation}\label{eq: 973ty}
	\langle [\Im G(\tilde{z})B_{1}\Im G(\tilde{w}) - \widehat{M}(\tilde{z},B_{1},\tilde{w})] B_{2}\rangle
	\prec \lVert A\rVert\lVert\mathbf{x}\rVert\lVert\mathbf{y}\rVert \Psi(\tilde{z},\tilde{w}).
\end{equation}
 
Once~\eqref{eq: 973ty} is established, the double integral in \eqref{eq: 994ty} can be estimated easily. Indeed,
observe that for the $x$ integration we have
\begin{equation}\label{eq: 960ty}
	\int_{\R} \frac{1}{|x-\psi^{-1}(z)| |\Im\psi(x)|^{p}} \mathrm{d}x \lesssim 
	\frac{1}{|\Im z|^{p}} + \int_{|x|>|\Im z|^{1/\gamma}} \frac{1}{|x|^{\gamma p+1}}\mathrm{d}x
	\lesssim \frac{1}{|\Im z|^{p}}, \quad p>0,
\end{equation} 
using $\Im\psi(x)\gtrsim \max\{|x|^{\gamma},|\Im z|\}$ and
$|x-\psi^{-1}(z)|\gtrsim \max\{|x|,|\Im z|^{1/\gamma}\}$ and analogous formula holds for the $y$-integration. 
The $|\Im\psi(x)|^{p}\equiv |\Im\tilde z|^p$ factor comes from the $\Psi(\tilde{z},\tilde{w})$ factor in~\eqref{eq: 973ty}. 
Strictly speaking~\eqref{eq: 973ty} is available only if $\Re \tilde z, \Re\tilde w$ are in the bulk but 
since $\Re z$ is in the bulk, note 
 that $\Re\tilde z\equiv \Re\psi(x)$ also
  lies in the bulk if $x$ is in a neighborhood of the origin. When $x$ is away from the origin, then 
 $\Im \tilde z\gtrsim 1$, the denominators in~\eqref{eq: 994ty} are bounded away from zero and
 we can use the global law for the numerator.

Now we explain the proof of~\eqref{eq: 973ty}.
Let $C>0$ be a constant sufficiently large. 
We restrict our proof to the case that $\Im\tilde z, \Im\tilde w\le C$, the nontrivial regime.
Without loss of generality, we suppose $\Re\tilde z$ and $\Re\tilde w$ are in the bulk for all $\Im\tilde z, \Im\tilde w\le C$ by choosing the cones appropriately.
Next we regularize the observables (with respect to the spectral parameters $\tilde{z}$ and $\tilde{w}$), that is, $B_{i}=\bdot{B}_{i}+\Upsilon_{i}$ for each $i\in\{1,2\}$;
see \cite[Definition 3.3]{EHR26+} for the notion of regular observables.
Note that each $\Upsilon_{i}$ is a scalar.
After regularization, we split into four terms and apply the local law \cite[Theorem 3.7]{EHR26+} for each term.
Using \cite[Eq.\,(3.36a)]{EHR26+} to the term $\langle \Im G(\tilde z) \bdot{B}_{1} \Im G(\tilde w) \bdot{B}_{2} \rangle$ containing two regular observables, we get the bound $\lVert A \rVert\lVert\mathbf{x}\rVert\lVert\mathbf{y}\rVert/\sqrt{\min(\eta_{1},\eta_{2})}$, which contributes the first term in \eqref{eq: 928ty}.
Here we used that 
 $\Im\tilde z \gtrsim \Im z\equiv\eta_1$ and $\Im\tilde w \gtrsim \Im w\equiv \eta_2$ based upon the cone constructrion and we also used
  $\triplenorms{\bdot{B}_{1}}_{\tilde{z},\tilde{w}}\lesssim \lVert A\rVert$ and $\triplenorms{\bdot{B}_{2}}_{\tilde{z},\tilde{w}}\lesssim \lVert\mathbf{x}\rVert\lVert\mathbf{y}\rVert\sqrt{N}$ in the bulk (see \cite[Remark 3.12]{EHR26+}).
There are two terms containing a single regular observables, e.g.\,$\Upsilon_{2}\langle \Im G(\tilde z) \bdot{B}_{1} \Im G(\tilde w) \rangle$. For this term, we apply \cite[Eq.\,(3.36b)]{EHR26+}.
The other term (with a single regular observable) is estimated analogously. Together with $|\Upsilon_{1}|\lesssim\lVert A\rVert$, we obtain the bound $\lVert A\rVert\lVert\mathbf{x}\rVert\lVert\mathbf{y}\rVert/\sqrt{N\min\{\eta_{1},\eta_{2}\}^{3}}$ for these single regular observable case, which contributes the middle term in \eqref{eq: 928ty}.
The last one $\Upsilon_{1}\Upsilon_{2}\langle \Im G(\tilde z) \Im G(\tilde w) \rangle$ is estimated by \cite[Eq.\,(3.36c)]{EHR26+} with the fact $|\Upsilon_{2}|\lesssim \lVert\mathbf{x}\rVert\lVert\mathbf{y}\rVert$ (recalling our setting $B_{2}=N\mathbf{y}\mathbf{x}^{*}$) in the bulk, which gives the bound $\lVert A\rVert\lVert\mathbf{x}\rVert\lVert\mathbf{y}\rVert/N \min\{\eta_{1},\eta_{2}\}^{2}$.
Putting together all these contributions, we get \eqref{eq: 973ty}.

Combining \eqref{eq: 973ty}--\eqref{eq: 960ty} and \eqref{eq: 1001ty}, we obtain \eqref{eq: 882d}.
The non-perturbative regime is easier since the regularization is not needed and one can use \cite[Theorem 3.8]{EHR26+} directly. This concludes the proof. 
\end{proof}

Define the stability operators $\mathcal{B}_{z,w}$ and $\mathcal{B}_{z}\coloneqq\mathcal{B}_{z,z}$ by setting
\begin{equation}\label{eq: 535}
	\mathcal{B}_{z,w} \coloneqq 1 - \mathcal{C}_{M(z),M(w)}\mathcal{S}.
\end{equation} 

\begin{lemma}[Two-body stability, {\cite[Lemma 3.2]{EHR26+}}]\label{lem: 792}
Suppose the same setting as in Theorem \ref{thm: 369}.
Consider $z=E_{1}+\mathrm{i}\eta_{1}$ and $w=E_{2}+\mathrm{i}\eta_{2}$ with $E_{1},E_{2}\in\mathcal{I}(\eps_{0}/2)$ and $|\eta_{1}|,|\eta_{2}|\in(0,100)$.
Let $\mathcal{B}_{z,w}$ be as in \eqref{eq: 535}.
Then,
\begin{equation}\label{eq: 913d}
	\lVert \mathcal{B}_{z,w}^{-1} \rVert_{\lVert\cdot\rVert\to\lVert\cdot\rVert} + \lVert \mathcal{B}_{z,w}^{-1} \rVert_{\lVert\cdot\rVert_{\textnormal{hs}}\to\lVert\cdot\rVert_{\textnormal{hs}}}
	\lesssim
	\begin{cases}
		1, & \eta_{1}\eta_{2} > 0, \\
		\max\{1,|z-w|^{-1}\}, & \eta_{1}\eta_{2} < 0.
	\end{cases}
\end{equation}
Specifically, for the case $z=w$,
\begin{equation}\label{eq: 922d}
	\lVert \mathcal{B}_{z}^{-1} \rVert_{\lVert\cdot\rVert\to\lVert\cdot\rVert} + \lVert \mathcal{B}_{z}^{-1} \rVert_{\lVert\cdot\rVert_{\textnormal{hs}}\to\lVert\cdot\rVert_{\textnormal{hs}}}
	\lesssim 1.
\end{equation}
Moreover, there exist (small) positive thresholds $\beta_{*},\eps_{*}\sim 1$, $\eps_{*}<1/3$, such that, if $|z-w|\le\beta_{*}$ and $\eta_{1}\eta_{2}<0$, then
\begin{equation}\label{eq: 572a}
	\lVert (\zeta-\mathcal{B}_{z,w})^{-1} \rVert_{\lVert\cdot\rVert\to\lVert\cdot\rVert} + \lVert (\zeta-\mathcal{B}_{z,w})^{-1} \rVert_{\lVert\cdot\rVert_{\textnormal{hs}}\to\lVert\cdot\rVert_{\textnormal{hs}}}
	\lesssim 1,
\end{equation}
for all $\zeta\in\mathbb{C}$ satisfying $|\zeta|\ge\eps_{*}$ and $|1-\zeta|\ge 1-2\eps_{*}$,
and the disk of radius $\eps_{*}$ around the origin contains a single eigenvalue $\beta_{z,w}$ of $\mathcal{B}_{z,w}$ with algebraic multiplicity one, that is,
\begin{equation}\label{eq: 482}
	\textnormal{rank}\,\Pi_{z,w} = 1, \quad\;\; \Pi_{z,w}\coloneqq \frac{1}{2\pi\mathrm{i}}\oint_{|\zeta|=\eps_{*}} (\zeta-\mathcal{B}_{z,w})^{-1} \mathrm{d}\zeta,
\end{equation}
where $\Pi_{z,w}$ is the spectral projection corresponding to $\beta_{z,w}$.
Furthermore, the eigenvalue $\beta_{z,w}$ of $\mathcal{B}_{z,w}$ satisfies
\begin{equation}\label{eq: 937d}
	|\beta_{z,w}| \sim 
	\begin{cases}
		1, & \eta_{1}\eta_{2} > 0, \\
		\min\{1,|z-w|\}, & \eta_{1}\eta_{2} < 0.
	\end{cases}
\end{equation}
\end{lemma}

We obtain the following corollary from Lemma \ref{lem: 792}.

\begin{corollary}\label{cor: 953}
Suppose the same setting as in Lemma \ref{lem: 792}.
For $|z-w|\le\beta_{*}$ and $\eta_{1}\eta_{2}<0$, we have
\begin{equation*}
	\lVert \mathcal{B}_{z,w}^{-1}(1-\Pi_{z,w})\rVert_{\lVert\cdot\rVert\to\lVert\cdot\rVert}\lesssim 1, \quad
	\lVert \mathcal{B}_{z,w}^{-1}(1-\Pi_{z,w})\rVert_{\lVert\cdot\rVert_{\textnormal{hs}}\to\lVert\cdot\rVert_{\textnormal{hs}}}\lesssim 1.
\end{equation*}
\end{corollary}
\begin{proof}
Note that
\begin{equation*}
	\mathcal{B}_{z,w}^{-1}(1-\Pi_{z,w}) = -\frac{1}{2\pi\mathrm{i}}\oint_{|\zeta|=\eps_{*}}\zeta^{-1}(\zeta-\mathcal{B}_{z,w})^{-1}\mathrm{d}\zeta.
\end{equation*}
Then the desired result follows from \eqref{eq: 572a}.
\end{proof}

\begin{lemma}[Bounds for $M(z)$ and $\mathcal{S}$, {\cite[Proposition 3.5; Lemma 4.8]{AEK20}}]\label{lem: 354}
Suppose the same setting as in Theorem \ref{thm: 369}.
We have
\begin{equation}\label{eq: 949d}
	\lVert M(z) \rVert \lesssim 1 \quad\text{and}\quad \lVert M^{-1}(z) \rVert \lesssim 1+|z|,
\end{equation}
uniformly for every $z=E+\mathrm{i}\eta$ with $E$ and $\eta$ as in Theorem \ref{thm: 369}.
In addition,
\begin{equation}\label{eq: 954d}
	\lVert \mathcal{S} \rVert_{\lVert\cdot\rVert_{\textnormal{hs}}\to\lVert\cdot\rVert}
	\lesssim 1.
\end{equation}
\end{lemma}

\begin{lemma}[{\cite[Lemma 4.4]{LS20}} and {\cite[Lemma 5.6]{Vova25a}}]\label{lem: 358}
Let $g$, $f$, $\tilde{f}$, $\eta_{0}$, and $\chi$ be as in Theorem \ref{thm: 96} and Proposition \ref{prop: 147}.
Consider a domain $\Omega$ of the form
\begin{equation*}
	\Omega = \{z=x+\mathrm{i}\eta\in\mathbb{C}: x \in (a,b),\, \eta\in(cN^{-\tau}\eta_{0},1) \},
\end{equation*}
such that $\textnormal{supp}(f)\subset (a,b)$ and $\tau,c$ are positive constants.
Let $K(z)$ be a holomorphic function on $\Omega$ satisfying
\begin{equation*}
	|K(z)| \le C|\eta|^{-s}, \quad z\in\Omega,
\end{equation*}
for some constants $C>0$ and $s\ge0$. Then there exists a constant $C'>0$ depending only on $g$, $\chi$, and $s$, such that
\begin{equation*}
	\bigg| \int_{\Omega}\frac{\partial\tilde{f}}{\partial\bar{z}} K(x+\mathrm{i}\eta) \mathrm{d}x\mathrm{d}\eta \bigg|
	\le
	\begin{cases}
		CC'\eta_{0}^{1-s}\log{N}, & s\in[0,2] \\
		c^{2-s}CC'N^{-\tau(2-s)}\eta_{0}^{1-s}\log{N}, & s > 2
	\end{cases}
	.
\end{equation*}
\end{lemma}
\begin{proof} The claim for $s\in[0,2]$ was shown in \cite[Lemma 4.4]{LS20} and \cite[Lemma 5.6]{Vova25a}.
For $s>2$, the only difference is the integral estimate for $|\eta|\in (cN^{-\tau}\eta_{0},\eta_{0})$;
specifially,
\begin{equation*}
	\int_{cN^{-\tau}\eta_{0}}^{\eta_{0}}|\eta|^{1-s}\mathrm{d}\eta \lesssim (cN^{-\tau}\eta_{0})^{2-s}.
\end{equation*}
The remaining parts of the proof is the same as in \cite[Lemma 4.4]{LS20} and \cite[Lemma 5.6]{Vova25a}.
\end{proof}

\begin{lemma}[Non-Hermitian perturbation, {\cite[Lemma C.1]{AEK20}}]\label{lem: 523}
Let $\beta_{*}$ and $\eps_{*}$ be as in Lemma \ref{lem: 792}. Consider $(z_{0},w_{0}),(z_{1},w_{1})\in\mathbb{C}$ satisfying $|z_{i}-w_{i}|\le\beta_{*}$, $\Re z_{i}, \Re w_{i}\in\mathcal{I}(\eps_{0}/2)$, and $\Im z_{i}\cdot\Im w_{i}<0$ for each $i=0,1$.
Denote the stability operator $\mathcal{B}_{i}\coloneqq \mathcal{B}_{z_{i},w_{i}}$ and
the single eigenvalue $\beta_{i}\coloneqq\beta_{z_{i},w_{i}}$ of $\mathcal{B}_{i}$ for each $i=0,1$.
Let $R_{0}$ be the normalized right eigenvector of $\mathcal{B}_{0}$ corresponding $\beta_{0}$;
i.e.\,$\mathcal{B}_{0}[R_{0}]=\beta_{0}R_{0}$ and $\lVert R_{0}\rVert_{\textnormal{hs}}=1$.
Similarly, let $L_{0}$ be the corresponding left eigenvector, that is, $\mathcal{B}_{0}^{*}[L_{0}]=\bar{\beta}_{0}L_{0}$  and $\lVert L_{0}\rVert_{\textnormal{hs}}=1$.
Let $\Pi_{i}$ be the rank-one spectral projection corresponding to $\beta_{i}$ (as in Lemma \ref{lem: 792}) for each $i=0,1$.
Define $R_{1}\coloneqq \Pi_{1}[R_{0}]$ and $L_{1}\coloneqq \Pi_{1}^{*}[L_{0}]$.

Then $R_{1}$ and $L_{1}$ are (unnormalized) eigenvectors of $\mathcal{B}_{1}$ and $\mathcal{B}_{1}^{*}$ corresponding to $\beta_{1}$ and $\bar{\beta}_{1}$, respectively.
Write $\Delta\mathcal{B}\coloneqq \mathcal{B}_{1}-\mathcal{B}_{0}$. Then, we have
\begin{equation*}
	R_{1} = R_{0} + O(\lVert \Delta\mathcal{B}\rVert_{\lVert\cdot\rVert_{\textnormal{hs}}\to\lVert\cdot\rVert_{\textnormal{hs}}}), \quad
	L_{1} = L_{0} + O(\lVert \Delta\mathcal{B}\rVert_{\lVert\cdot\rVert_{\textnormal{hs}}\to\lVert\cdot\rVert_{\textnormal{hs}}}),
\end{equation*}
\begin{equation*}
	\beta_{1}\langle L_{1}, R_{1} \rangle 
	= \beta_{0} \langle L_{0}, R_{0}\rangle + \langle L_{0}, \Delta\mathcal{B}[R_{0}]\rangle + O(\lVert  \Delta\mathcal{B}\rVert_{\lVert\cdot\rVert_{\textnormal{hs}}\to\lVert\cdot\rVert_{\textnormal{hs}}}^{2}),
\end{equation*}
where the implicit constants in the error terms $O(\lVert \Delta\mathcal{B}\rVert_{\lVert\cdot\rVert_{\textnormal{hs}}\to\lVert\cdot\rVert_{\textnormal{hs}}})$ depend only on $\eps_{*}$ and $\beta_{*}$.
\end{lemma}
\begin{proof}
This is a direct application of \cite[Lemma C.1]{AEK20} for $\mathcal{B}_{z,w}$ together with Lemma \ref{lem: 792}.
The only difference is replacing $\lVert \Delta\mathcal{B}\rVert_{\lVert\cdot\rVert\to\lVert\cdot\rVert}$ with $\lVert \Delta\mathcal{B}\rVert_{\lVert\cdot\rVert_{\textnormal{hs}}\to\lVert\cdot\rVert_{\textnormal{hs}}}$ in the error term, which is possible due to \eqref{eq: 572a} of Lemma \ref{lem: 792}.
All the other parts of the proof are identical with those of \cite[Lemma C.1]{AEK20} so we omit them.
\end{proof}

\section{Proof of Proposition \ref{prop: 147}} \label{sec: 535}

Following \cite{LS20, LLS24, Vova25a},
define $e(\lambda) \coloneqq \exp\big( \mathrm{i}\lambda \{\Tr f(H)-\E\Tr f(H)\} \big).$
Then, $\phi(\lambda) = \E[e(\lambda)].$
Expressing $\Tr f(H)$ by the Helffer-Sj\"{o}strand formula \cite[Section 1.13]{MR3791802},
\footnote{One can also refer to the lecture note at \href{https://arxiv.org/pdf/1601.04055}{\texttt{https://arxiv.org/pdf/1601.04055}} for the same content. See Appendix C therein.}
\begin{equation*}
	\phi'(\lambda) = \frac{\mathrm{i}}{\pi}\int_{\mathbb{C}}\frac{\partial\tilde{f}}{\partial\bar{z}}
	\E\big[e(\lambda)\Tr (G(z) - \E G(z))\big] \mathrm{d}^{2}z.
\end{equation*}

\begin{lemma}[E.g.\,{\cite[(5.9)]{Vova25a} and \cite[(4.22)]{LS20}}]
Let $\Omega_{0}$ and $\Omega_{0}'$ as in \eqref{eq: 273}.
Define
\begin{equation*}
	\tilde{e}(\lambda)
	\coloneqq
	\exp\bigg\{ \frac{\mathrm{i}\lambda}{\pi}\int_{\Omega_{0}'}\frac{\partial\tilde{f}}{\partial\bar{z}}\Tr (G(z) - \E G(z))\mathrm{d}^{2}z\bigg\},
\end{equation*}
and
\begin{equation}\label{eq: 678}
	\mathfrak{g} \equiv \mathfrak{g}(z, \lambda) \coloneqq \E[\tilde{e} (\lambda)  (G(z)-\E[G(z)])].
\end{equation}
Then it follows that
\begin{equation}\label{eq: 667a}
	\phi'(\lambda) = \frac{\mathrm{i}}{\pi}\int_{\Omega_{0}}\frac{\partial\tilde{f}}{\partial\bar{z}}
	\Tr(\mathfrak{g}) \mathrm{d}^{2}z
	+ O_{\prec}(|\lambda|N^{-\eps_{0}/2}).
\end{equation}
\end{lemma}
\begin{proof}
As in \cite[Appendix A]{Vova25a}, we can use the argument around \cite[Eq.\,(4.21)--(4.22)]{LS20}.
The same reasoning works for our correlated model 
since it relies only on the monotonicity of the map $\eta\mapsto\eta\,\Im G(E+\mathrm{i}\eta)$
that remains valid.
Hence, applying the average local law in Theorem \ref{thm: 369} for $\eta\sim N^{-\eps_{0}/2}\eta_{0}$, we obtain \eqref{eq: 667a}. The other minor details are omitted.
\end{proof}

The following lemma computes $\mathfrak{g}$. The proof relies on
deriving an approximate self-consistent equation for $\mathfrak{g}$ by applying the multivariate cumulant expansion (Lemma \ref{lem: 378}).
The proof will be given in Section \ref{sec: 783}.

\begin{lemma}\label{lem: 393} 
Let $\mathcal{B}_{z}$ and $\mathfrak{g}$ be as in \eqref{eq: 535} and \eqref{eq: 678}.
Then,
\begin{equation}\label{eq: 694a}
	\mathfrak{g} = -\mathcal{B}_{z}^{-1}\Big[\Big(R + \sum_{k=1}^{L-1} Q^{(k)} + \mathcal{E}^{(L)}\Big)M\Big],
\end{equation}
where $R=(R_{ij})$ and $Q^{(k)}=(Q^{(k)}_{ij})$ are defined by
\begin{equation}\label{eq: 698a}
	R_{ij} \coloneqq \sum_{a,b,k} \frac{\kappa(aj,bk)}{N} \E[(\partial_{bk}\tilde{e})G_{ia}],
\end{equation}
\begin{equation}\label{eq: 1104e}
	Q^{(1)}_{ij} \coloneqq Q^{(1,1)}_{ij} + Q^{(1,2)}_{ij} + Q^{(1,3)}_{ij}, \quad
	Q^{(1,1)}_{ij} \coloneqq -\sum_{a,b,k}\frac{\kappa_{d}(aj,bk)}{N}\E[\tilde{e}(G_{ib}G_{ka}-\E[G_{ib}G_{ka}])],
\end{equation}
\begin{multline}\label{eq: 1108e}
	Q^{(1,2)}_{ij} \coloneqq - \sum_{a,b,k}\frac{\kappa_{c}(aj,bk)}{N} \E[\tilde{e}(G_{ib}-\E[G_{ib}])]\E[G_{ka}-M_{ka}]
	- \sum_{a,b,k}\frac{\kappa_{c}(aj,bk)}{N} \E[\tilde{e}(G_{ka}-\E[G_{ka}])]\E[G_{ib}-M_{ib}] \\
	-\sum_{a,b,k}\frac{\kappa_{c}(aj,bk)}{N} \E[(\tilde{e}-\E[\tilde{e}])(G_{ib}-\E[G_{ib}])(G_{ka}-\E[G_{ka}])],
\end{multline}
\begin{equation}\label{eq: 1113e}
	Q^{(1,3)}_{ij}
	\coloneqq \sum_{a,b,k}\frac{\kappa_{d}(aj,bk)}{N} \E[\tilde{e}(G_{ib}-\E[G_{ib}])]M_{ka}
	+ \sum_{a,b,k}\frac{\kappa_{d}(aj,bk)}{N} \E[\tilde{e}(G_{ka}-\E[G_{ka}])]M_{ib},
\end{equation}
\begin{equation}\label{eq: 709a}
	Q^{(k)}_{ij} \coloneqq \sum_{a} \sum_{\boldsymbol{\alpha}\in\mathcal{N}^{k}(\alpha_{0})}
	\frac{\kappa(\alpha_{0},\boldsymbol{\alpha})}{N^{(k+1)/2}k!}\E[\partial_{\boldsymbol{\alpha}}(\tilde{e}G_{ia}) - \tilde{e}\E[\partial_{\boldsymbol{\alpha}}G_{ia}]], \quad k\in\{2,\dots,L-1\}.
\end{equation}
Finally, each entry of $\mathcal{E}^{(L)}=(\mathcal{E}^{(L)}_{ij})$ corresponds to a linear combination of error terms satisfying \eqref{eq: 391} where $f$ is replaced with
\begin{equation*}
	\sum_{a,b,k}\frac{\kappa(aj,bk)}{N}G_{ia} \quad\text{or}\quad \sum_{a,b,k}\frac{\kappa(aj,bk)}{N}\tilde{e}G_{ia}.
\end{equation*}
\end{lemma}
Using Lemma \ref{lem: 393}, we get
\begin{equation}\label{eq: 722a}
	\frac{\mathrm{i}}{\pi}\int_{\Omega_{0}}\frac{\partial\tilde{f}}{\partial\bar{z}}
	\Tr(\mathfrak{g}) \mathrm{d}^{2}z
	= \frac{\mathrm{i}}{\pi}\int_{\Omega_{0}}\frac{\partial\tilde{f}}{\partial\bar{z}}
	\Big(-\Tr\mathcal{B}_{z}^{-1}\Big[\Big(R + \sum_{k=1}^{L-1} Q^{(k)} + \mathcal{E}^{(L)}\Big)M\Big]\Big) \mathrm{d}^{2}z.
\end{equation}
We remark that $R$ is the main term and the other terms are negligible as shown in the next lemma.

\begin{lemma}\label{lem: 445}
We define the matrix $V=(V_{ij})$ by setting
\begin{equation}\label{eq: 531}
	V_{ij}\equiv V_{ij}(z,w) \coloneqq \sum_{b} \big( \mathcal{B}_{z,w}^{-1}\mathcal{C}_{M,\widetilde{M}}\mathcal{S}[E^{jb}] \big)_{ib}.
\end{equation}
Then,
\begin{multline}\label{eq: 459}
	\frac{\mathrm{i}}{\pi}\int_{\Omega_{0}}\frac{\partial\tilde{f}}{\partial\bar{z}}
	\Tr\mathcal{B}_{z}^{-1}[RM] \mathrm{d}^{2}z \\
	=
	\frac{2\lambda\phi(\lambda)}{\pi^{2}}
	\int_{\Omega_{0}}\int_{\Omega_{0}'}\frac{\partial\tilde{f}(z)}{\partial\bar{z}}\frac{\partial\tilde{f}(w)}{\partial\bar{w}}
	\partial_{w}\Tr (M'M^{-1}V) \mathrm{d}^{2}w \mathrm{d}^{2}z
	+ O_{\prec}\big( (1+|\lambda|)N^{-\eps_{0}/4} \big).
\end{multline}
Moreover,
\begin{equation}\label{eq: 468}
	\int_{\Omega_{0}}\frac{\partial\tilde{f}}{\partial\bar{z}}
	\Tr\mathcal{B}_{z}^{-1}[Q^{(k)}M] \mathrm{d}^{2}z
	= O_{\prec}\Big( (1+|\lambda|)^{2}\eta_{0}^{1/2} + (N\eta_{0})^{-1/2} \Big), \quad k\in\{1,\dots,L-1\},
\end{equation}
and
\begin{equation}\label{eq: 474}
	\int_{\Omega_{0}}\frac{\partial\tilde{f}}{\partial\bar{z}}
	\Tr\mathcal{B}_{z}^{-1}[\mathcal{E}^{(L)}M] \mathrm{d}^{2}z = O_{\prec}(\eta_{0}).
\end{equation}
\end{lemma}
 
Combining \eqref{eq: 667a} and \eqref{eq: 722a}, Lemma \ref{lem: 445} completes the proof of Proposition \ref{prop: 147}, together with the following identities;
\begin{align*}
	\Tr(M'M^{-1}V)
	&= \Tr_{\text{op}}(\mathcal{C}_{M'M^{-1},I}\mathcal{B}_{z,w}^{-1}\mathcal{C}_{M,\widetilde{M}}\mathcal{S}), \\
	\partial_{w}\Tr_{\text{op}}(\mathcal{C}_{M'M^{-1},I}\mathcal{B}_{z,w}^{-1}\mathcal{C}_{M,\widetilde{M}}\mathcal{S})
	&= \Tr_{\textnormal{op}}(\mathcal{C}_{M'M^{-1},I}\mathcal{B}_{z,w}^{-1}\mathcal{C}_{M,\widetilde{M}'}\mathcal{S}\mathcal{B}_{z,w}^{-1}).
\end{align*}

\subsection{Proof of Lemma \ref{lem: 393}} \label{sec: 783}

Consider 
\begin{equation}\label{eq: 784a}
	z\E[\tilde{e}(G_{ij}-\E[G_{ij}])]
	= \sum_{a}\E[\tilde{e} (G_{ia}H_{aj}-\E[G_{ia}H_{aj}])],
\end{equation}
where we used $G(H-z)=I$.
Next,
\begin{equation}\label{eq: 220}
	\sum_{a}\E[\tilde{e} (G_{ia}H_{aj}-\E[G_{ia}H_{aj}])] = \sum_{a}\E[\tilde{e} (G_{ia}W_{aj}-\E[G_{ia}W_{aj}])]
	+ \sum_{a}\E[\tilde{e} (G_{ia}-\E[G_{ia}])A_{aj}].
\end{equation}
We expand the first term in the right-hand side of \eqref{eq: 220} via the multivariate cumulant expansion, Lemma \ref{lem: 378}.
Recalling \eqref{eq: 709a},
\begin{equation}\label{eq: 1259ft}
	\sum_{a}\E[\tilde{e} (G_{ia}W_{aj}-\E[G_{ia}W_{aj}])]
	= \sum_{a,b,k}\frac{\kappa(aj,bk)}{N}\E\big[\partial_{bk}(\tilde{e}G_{ia}) - \tilde{e}\E[\partial_{bk}G_{ia}]\big]
	+ \sum_{k=2}^{L-1} Q^{(k)}_{ij}
	+ \mathcal{E}^{(L)}_{ij},
\end{equation}
where the error term $\mathcal{E}^{(L)}_{ij}$ satisfies the desired property due to \eqref{eq: 391}.
Since
\begin{equation*}
	\E\big[\partial_{bk}(\tilde{e}G_{ia}) - \tilde{e}\E[\partial_{bk}G_{ia}]\big]
	= \E[\tilde{e}(\partial_{bk}G_{ia}-\E[\partial_{bk}G_{ia}])] + \E[(\partial_{bk}\tilde{e})G_{ia}],
\end{equation*}
together with \eqref{eq: 698a}, we have
\begin{equation}\label{eq: 819a}
	\eqref{eq: 1259ft}
	= \sum_{a,b,k}\frac{\kappa(aj,bk)}{N}\E[\tilde{e}(\partial_{bk}G_{ia}-\E[\partial_{bk}G_{ia}])]
	+ R_{ij}
	+ \sum_{k=2}^{L-1} Q^{(k)}_{ij}
	+ \mathcal{E}^{(L)}_{ij}.
\end{equation}
Now we claim that
\begin{equation}\label{eq: 827a}
	\sum_{a,b,k}\frac{\kappa(aj,bk)}{N}\E[\tilde{e}(\partial_{bk}G_{ia}-\E[\partial_{bk}G_{ia}])]
	= - (\mathfrak{g}\mathcal{S}[M])_{ij} - (M\mathcal{S}[\mathfrak{g}])_{ij} + Q^{(1)}_{ij}.
\end{equation}
Combining \eqref{eq: 784a}--\eqref{eq: 827a}, we obtain
\begin{equation*}
	z\mathfrak{g} = \mathfrak{g}A - \mathfrak{g}\mathcal{S}[M] - M\mathcal{S}[\mathfrak{g}]
	+ R
	+ \sum_{k=1}^{L-1} Q^{(k)} + \mathcal{E}^{(L)}.
\end{equation*}
Using the MDE \eqref{eq: 197}, we have
\begin{equation*}
	\mathcal{B}_{z}[\mathfrak{g}] = -\Big(R + \sum_{k=1}^{L-1} Q^{(k)} + \mathcal{E}^{(L)}\Big)M,
\end{equation*}
which implies \eqref{eq: 694a}.

What remains is to show the claim \eqref{eq: 827a}.
Using $\partial_{bk}G_{ia} = -G_{ib}G_{ka}$, and
the decomposition $\kappa(\cdot,\cdot)=\kappa_{c}(\cdot,\cdot)+\kappa_{d}(\cdot,\cdot)$ from Assumption \ref{assump: 140} (ii),
\begin{multline}\label{eq: 850a}
    \sum_{a,b,k}\frac{\kappa(aj,bk)}{N}\E[\tilde{e}(\partial_{bk}G_{ia}-\E[\partial_{bk}G_{ia}])]
	= -\sum_{a,b,k}\frac{\kappa_{c}(aj,bk)}{N}\E[\tilde{e}(G_{ib}G_{ka}-\E[G_{ib}G_{ka}])] + Q^{(1,1)}_{ij}.
\end{multline}
For arbitrary random variables $X$, $Y$ and $Z$, we have the identity
\begin{equation*}
	\E[X(YZ-\E[YZ])] = \E[(X-\E[X])(Y-\E[Y])(Z-\E[Z])]
	+ \E[X(Y-\E[Y])]\E[Z] + \E[X(Z-\E[Z])]\E[Y].
\end{equation*}
Hence, for the first term of the right-hand side in \eqref{eq: 850a},
\begin{multline*}
	-\sum_{a,b,k}\frac{\kappa_{c}(aj,bk)}{N}\E[\tilde{e}(G_{ib}G_{ka}-\E[G_{ib}G_{ka}])]
	= -\sum_{a,b,k}\frac{\kappa_{c}(aj,bk)}{N} \E[(\tilde{e}-\E[\tilde{e}])(G_{ib}-\E[G_{ib}])(G_{ka}-\E[G_{ka}])] \\
	- \sum_{a,b,k}\frac{\kappa_{c}(aj,bk)}{N} \E[\tilde{e}(G_{ib}-\E[G_{ib}])]\E[G_{ka}]
	- \sum_{a,b,k}\frac{\kappa_{c}(aj,bk)}{N} \E[\tilde{e}(G_{ka}-\E[G_{ka}])]\E[G_{ib}].
\end{multline*}
Using $G_{ij} = M_{ij} + (G_{ij}-M_{ij})$, it follows that
\begin{align}\label{eq: 881a} 
	-&\sum_{a,b,k}\frac{\kappa_{c}(aj,bk)}{N}\E[\tilde{e}(G_{ib}G_{ka}-\E[G_{ib}G_{ka}])] \nn\\
	&= - \sum_{a,b,k}\frac{\kappa(aj,bk)}{N} \E[\tilde{e}(G_{ib}-\E[G_{ib}])]M_{ka}
	- \sum_{a,b,k}\frac{\kappa(aj,bk)}{N} \E[\tilde{e}(G_{ka}-\E[G_{ka}])]M_{ib}
	+ Q^{(1,2)}_{ij} + Q^{(1,3)}_{ij}.
\end{align}
Combining \eqref{eq: 850a}--\eqref{eq: 881a}, we obtain
\begin{multline}\label{eq: 907a}
	\sum_{a,b,k}\frac{\kappa(aj,bk)}{N}\E[\tilde{e}(\partial_{bk}G_{ia}-\E[\partial_{bk}G_{ia}])] \\
	= - \sum_{a,b,k}\frac{\kappa(aj,bk)}{N} \E[\tilde{e}(G_{ib}-\E[G_{ib}])]M_{ka}
	- \sum_{a,b,k}\frac{\kappa(aj,bk)}{N} \E[\tilde{e}(G_{ka}-\E[G_{ka}])]M_{ib}
	+ Q^{(1)}_{ij}.
\end{multline}
Recalling $\kappa(aj,bk) = N\E[W_{aj}W_{bk}]$ and using the fact that $\mathcal{S}[M]=\E[WMW]$,
we observe
\begin{multline*}
    \sum_{a,b,k}\frac{\kappa(aj,bk)}{N} \E[\tilde{e}(G_{ib}-\E[G_{ib}])]M_{ka}
    =
	\sum_{a,b,k}\E[W_{aj}W_{bk}]\E[\tilde{e}(G_{ib}-\E[G_{ib}])]M_{ka} \\
	= -\sum_{b}\E[\tilde{e}(G_{ib}-\E[G_{ib}])](\mathcal{S}[M])_{bj}
	= (\mathfrak{g}S[M])_{ij},
\end{multline*}
which identifies the first term in the right-hand side of \eqref{eq: 907a}.
The identification of the second term with $(MS[\mathfrak{g}])_{ij}$ is analogous by writing $S[\mathfrak{g}]=\E[\widetilde{W}\mathfrak{g}\widetilde{W}]$ with an independent copy $\widetilde{W}$ of $W$.
This completes the proof of claim~\eqref{eq: 827a} and thus Lemma~\ref{lem: 393}.
\qed

\subsection{Proof of Lemma \ref{lem: 445}} \label{sec: 947}

This lemma consists of the estimates for certain quantities of the following form: 
\begin{equation*}
	\int_{\Omega_{0}}\frac{\partial\tilde{f}}{\partial\bar{z}}
	\Tr\mathcal{B}_{z}^{-1}[JM] \mathrm{d}^{2}z, \quad J\in\mathbb{C}^{N\times N}.
\end{equation*}
We observe that
\begin{equation}\label{eq: 953a}
	\Tr\mathcal{B}_{z}^{-1}[JM] = \Tr (M'M^{-1}J).
\end{equation}
To verify \eqref{eq: 953a}, we start with the following identity:
\begin{equation*}
	\Tr\mathcal{B}_{z}^{-1}[J M] = N\langle I,\mathcal{B}_{z}^{-1}[JM]\rangle
	= N\langle (\mathcal{B}_{z}^{*})^{-1}[I],JM\rangle.
\end{equation*}
Note that $\mathcal{B}_{z}^{*} = 1 - \mathcal{S}\mathcal{C}_{M^{*}(z)} = 1 - \mathcal{S}\mathcal{C}_{M(\bar{z})}.$
Combining the identity
\begin{equation*}
	(1-\mathcal{C}_{M}\mathcal{S})^{-1}\mathcal{C}_{M} = \mathcal{C}_{M}(1-\mathcal{S}\mathcal{C}_{M})^{-1},
\end{equation*}
and the fact
\begin{equation}\label{eq: 671}
	M' = (1-\mathcal{C}_{M}\mathcal{S})^{-1}\mathcal{C}_{M}[I], \quad M'(w)=\partial_{w}M(w),
\end{equation}
we get
\begin{equation*}
	(1 - \mathcal{S}\mathcal{C}_{M})^{-1}[I] = \mathcal{C}_{M}^{-1}(1-\mathcal{C}_{M}\mathcal{S})^{-1}\mathcal{C}_{M}[I] = M^{-1}M'M^{-1}.
\end{equation*}
Applying this to $M(\bar{z})=M^{*}(z)$, it implies $N\langle (\mathcal{B}_{z}^{*})^{-1}[I],JM\rangle = \Tr (M'M^{-1}J),$
which gives \eqref{eq: 953a}.
Thus, in order to show Lemma \ref{lem: 445}, we will consider the following quantities:
\begin{equation}\label{eq: 1000a}
	\frac{\mathrm{i}}{\pi}\int_{\Omega_{0}}\frac{\partial\tilde{f}}{\partial\bar{z}}
	\Tr (M'M^{-1}R) \mathrm{d}^{2}z, \,
	\int_{\Omega_{0}}\frac{\partial\tilde{f}}{\partial\bar{z}}
	\Tr(M'M^{-1}Q^{(k)}) \mathrm{d}^{2}z, \,
	\int_{\Omega_{0}}\frac{\partial\tilde{f}}{\partial\bar{z}}
	\Tr(M'M^{-1}\mathcal{E}^{(L)}) \mathrm{d}^{2}z.
\end{equation}
We now present their analysis one by one.

\begin{proof}[Proof of \eqref{eq: 459}]

Write
\begin{equation*}
	\Tr (M'M^{-1}R)
	= \sum_{i,j,a,b,k} (M'M^{-1})_{ji}\frac{\kappa(aj,bk)}{N} \E[(\partial_{bk}\tilde{e})G_{ia}].
\end{equation*}
Note that
\begin{equation}\label{eq: 2003e}
	\partial_{bk}\tilde{e}(\lambda)
	= - \frac{\mathrm{i}\lambda\tilde{e}(\lambda)}{\pi}
	\int_{\Omega_{0}'}\frac{\partial\tilde{f}}{\partial\bar{w}} \partial_{w}G_{kb}(w) \mathrm{d}^{2}w.
\end{equation}
Then we observe
\begin{multline}\label{eq: 1038a}
	\frac{\mathrm{i}}{\pi}\int_{\Omega_{0}}\frac{\partial\tilde{f}}{\partial\bar{z}}
	\Tr (M'M^{-1}R) \mathrm{d}^{2}z \\
	= \frac{\lambda}{\pi^{2}} \int_{\Omega_{0}} \int_{\Omega_{0}'}
	\frac{\partial\tilde{f}}{\partial\bar{z}} \frac{\partial\tilde{f}}{\partial\bar{w}}
	\E\bigg[\tilde{e}\, \partial_{w}\bigg( \sum_{i,j,a,b,k}(M'M^{-1})_{ji}\frac{\kappa(aj,bk)}{N}G_{ia}\widetilde{G}_{kb} \bigg) \bigg]
	\mathrm{d}^{2}w \mathrm{d}^{2}z,
\end{multline}
where $G\equiv G(z)$ and $\widetilde{G}\equiv G(w)$.
We claim that
\begin{multline}\label{eq: 1048a}
	\sum_{i,j,a,b,k}(M'M^{-1})_{ji}\frac{\kappa(aj,bk)}{N}G_{ia}\widetilde{G}_{kb} \\
	= \frac{2}{N}\sum_{j,b}( \mathcal{C}_{M'M^{-1},I} (1-\mathcal{C}_{M,\widetilde{M}}\mathcal{S})^{-1}\mathcal{C}_{M,\widetilde{M}}[T_{c}^{(jb)}] )_{jb}
	+ O_{\prec}\big(\Psi(z,w)\big),
\end{multline}
where the matrix $T_{c}^{(jb)}\in\mathbb{R}^{N\times N}$ is defined (for each $(j,b)\in[N]^{2}$) by $T_{c}^{(jb)}(a,k) \coloneqq \kappa_{c}(aj,bk)$, and the term $\Psi(z,w)$ is from \eqref{eq: 928ty}.
The proof of \eqref{eq: 1048a}, an application of two-resolvent local laws, will be given later after we completed the proof of \eqref{eq: 459} assuming \eqref{eq: 1048a}.

We define the operator $\mathcal{S}_{c}$ by
\begin{equation}\label{eq: 1061a}
	(\mathcal{S}_{c}[E^{jb}])_{ak} \coloneqq \frac{\kappa_{c}(aj,bk)}{N}
	= \frac{T_{c}^{(jb)}(a,k)}{N}
	, \quad (j,b),(a,k)\in [N]^{2}.
\end{equation}
By Cauchy's integral formula on a circular contour centred at the origin with radius $\Im w/2$, the claim \eqref{eq: 1048a} implies, together with \eqref{eq: 1061a},
\begin{multline*}
	\partial_{w}\sum_{i,j,a,b,k}(M'M^{-1})_{ji}\frac{\kappa(aj,bk)}{N}G_{ia}\widetilde{G}_{kb} \\
	= 2\partial_{w}\sum_{j,b} ( \mathcal{C}_{M'M^{-1},I} (1-\mathcal{C}_{M,\widetilde{M}}\mathcal{S})^{-1}\mathcal{C}_{M,\widetilde{M}}\mathcal{S}_{c}[E^{jb}] )_{jb}
	+ O_{\prec}\big((\Im w)^{-1}\Psi(z,w)\big).
\end{multline*}
Applying Lemma \ref{lem: 358} twice,
\begin{multline}\label{eq: 1060a}
	\frac{\lambda}{\pi^{2}} \int_{\Omega_{0}} \int_{\Omega_{0}'}
	\frac{\partial\tilde{f}}{\partial\bar{z}} \frac{\partial\tilde{f}}{\partial\bar{w}}
	\E\bigg[\tilde{e} \, \partial_{w}\bigg( \sum_{i,j,a,b,k}(M'M^{-1})_{ji}\frac{\kappa(aj,bk)}{N}G_{ia}\widetilde{G}_{kb} \bigg) \bigg]
	\mathrm{d}^{2}w \mathrm{d}^{2}z \\
	=  \frac{2\lambda\phi(\lambda)}{\pi^{2}}
	\int_{\Omega_{0}}\int_{\Omega_{0}'}\frac{\partial\tilde{f}(z)}{\partial\bar{z}}\frac{\partial\tilde{f}(w)}{\partial\bar{w}}
	\partial_{w}\Tr (M'M^{-1}\tilde{V}) \mathrm{d}^{2}w \mathrm{d}^{2}z
	+ O_{\prec}\big( |\lambda|N^{-\eps_{0}/4} \big),
\end{multline}
where 
\begin{equation*}
	\widetilde{V}_{ij} \coloneqq \sum_{b} \big( (1-\mathcal{C}_{M,\widetilde{M}}\mathcal{S})^{-1}\mathcal{C}_{M,\widetilde{M}}\mathcal{S}_{c}[E^{jb}] \big)_{ib}.
\end{equation*}
Note that $V$ from \eqref{eq: 531} is defined similarly but with $\mathcal{S}$ instead of $\mathcal{S}_{c}$.
Define the operator $\mathcal{S}_{d}$ by
\begin{equation*}
	(\mathcal{S}_{d}[E^{jb}])_{ak} \coloneqq \frac{\kappa_{d}(aj,bk)}{N}, \quad (j,b), (a,k)\in [N]^{2}.
\end{equation*}
Then, $\mathcal{S} = \mathcal{S}_{c} + \mathcal{S}_{d}$
and it follows that
\begin{equation}\label{eq: 1094}
	\partial_{w}\Tr (M'M^{-1}\widetilde{V}) = \partial_{w}\Tr (M'M^{-1}V) - \partial_{w}\Tr_{\textnormal{op}}\big( \mathcal{C}_{M'M^{-1},I} (1-\mathcal{C}_{M,\widetilde{M}}\mathcal{S})^{-1} \mathcal{C}_{M,\widetilde{M}}\mathcal{S}_{d} \big).
\end{equation}
Based on \eqref{eq: 1038a}, \eqref{eq: 1060a} and \eqref{eq: 1094},
the desired conclusion \eqref{eq: 459} follows from the technical lemma below.

\begin{lemma}\label{lem: 1009}
We have
\begin{equation*}
	\int_{\Omega_{0}}\int_{\Omega_{0}'}\frac{\partial\tilde{f}(z)}{\partial\bar{z}}\frac{\partial\tilde{f}(w)}{\partial\bar{w}}
	\partial_{w}\Tr_{\textnormal{op}}\big( \mathcal{C}_{M'M^{-1},I} (1-\mathcal{C}_{M,\widetilde{M}}\mathcal{S})^{-1} \mathcal{C}_{M,\widetilde{M}}\mathcal{S}_{d} \big) \mathrm{d}^{2}w \mathrm{d}^{2}z
	= O\big(N^{-\eps_{0}}\big).
\end{equation*}
\end{lemma}
We defer the proof of Lemma \ref{lem: 1009} to Section \ref{sec: 3491} because the argument of the proof is related to that of Proposition \ref{prop: 169} to be discussed in Section \ref{sec: 1776}.
Given the claim \eqref{eq: 1048a} and Lemma \ref{lem: 1009}, the proof of \eqref{eq: 459} is completed.

From now on we prove \eqref{eq: 1048a}.
Recalling the decomposition $\kappa=\kappa_{c}+\kappa_{d}$,
it is enough to show that
\begin{multline}\label{eq: 1134a}
	\sum_{i,j,a,b,k} (M'M^{-1})_{ji} \frac{\kappa_{c}(aj,bk)}{N} G_{ia} \widetilde{G}_{kb} \\
	=  \frac{1}{N}\sum_{j,b}( \mathcal{C}_{M'M^{-1},I} (1-\mathcal{C}_{M,\widetilde{M}}\mathcal{S})^{-1}\mathcal{C}_{M,\widetilde{M}}[T_{c}^{(jb)}] )_{jb}
	+ O_{\prec}\big(\Psi(z,w)\big),
\end{multline}
and the same estimate holds if $\kappa_{c}$ is replaced with $\kappa_{d}$,
which yields the factor of $2$ in the right-hand side of \eqref{eq: 1048a}.
To obtain \eqref{eq: 1134a}, we write
\begin{equation*}
	\sum_{i,j,a,b,k}(M'M^{-1})_{ji}\frac{\kappa_{c}(aj,bk)}{N}G_{ia}\widetilde{G}_{kb}
	= \frac{1}{N}\sum_{j,b}(M'M^{-1}GT_{c}^{(jb)}\widetilde{G})_{jb}.
\end{equation*}
Appying the two-resolvent local law \eqref{eq: 882d} for each summand, we get
\begin{multline}\label{eq: 1596tg}
	\frac{1}{N}\sum_{j,b}(M'M^{-1}GT_{c}^{(jb)}\widetilde{G})_{jb}
	= \frac{1}{N}\sum_{j,b}( \mathcal{C}_{M'M^{-1},I} (1-\mathcal{C}_{M,\widetilde{M}}\mathcal{S})^{-1}\mathcal{C}_{M,\widetilde{M}}[T_{c}^{(jb)}] )_{jb} \\
	+ O_{\prec}\bigg(\frac{\Psi(z,w)\lVert M'M^{-1}\rVert\sum_{j,b} \lVert T_{c}^{(jb)}\rVert}{N} \bigg).
\end{multline}
By \eqref{eq: 671}, \eqref{eq: 922d} and \eqref{eq: 949d},
\begin{equation}\label{eq: 796} 
	\lVert M'\rVert \lesssim 1, \quad \lVert M^{-1}\rVert \lesssim 1.
\end{equation}
Since the assumption \eqref{eq: 241b} implies
\begin{equation}\label{eq: 814}
	\sum_{j,b} \lVert T_{c}^{(jb)}\rVert \lesssim N,
\end{equation}
the proof of \eqref{eq: 1134a} is completed.

The proof of \eqref{eq: 1134a} with $\kappa_{d}$ instead of $\kappa_{c}$ is very similar.
We write
\begin{equation*}
	\sum_{a,b,k} \frac{\kappa_{d}(aj,bk)}{N} G_{ia} \widetilde{G}_{kb}
	= \frac{1}{N} \sum_{k} (GT_{d}^{(jk)}\widetilde{G})_{ik},
\end{equation*}
where $T_{d}^{(jk)}(a,b) \coloneqq \kappa_{d}(aj,bk)$ and recalled that $G^{\T} = G$ for the real symmetric case.
By the two-resolvent local law \eqref{eq: 882d},
we get the same estimate as in \eqref{eq: 1596tg} with $T_{d}^{(\cdot)}$ instead of $T_{c}^{(\cdot)}$.

As in \eqref{eq: 814}, the assumption \eqref{eq: 245b} implies
$\sum_{j,k} \lVert T_{d}^{(jk)}\rVert \lesssim N.$
Then the rest of the proof is identical except for using \eqref{eq: 237} additionally to have $T_{d}^{(jk)} = T_{c}^{(jk)}$.
This completes the proof of \eqref{eq: 1048a}.

\end{proof}

\begin{proof}[Proof of \eqref{eq: 468} for $k=1$]

Recalling \eqref{eq: 1000a}, we consider
\begin{equation*}
	\int_{\Omega_{0}}\frac{\partial\tilde{f}}{\partial\bar{z}}
	\Tr(M'M^{-1}Q^{(1)}) \mathrm{d}^{2}z,
\end{equation*}
where $Q^{(1)} = Q^{(1,1)} + Q^{(1,2)} + Q^{(1,3)}$, thus
it is enough to show the following estimates:
\begin{equation} \label{eq: 1379a}
	\bigg| \int_{\Omega_{0}}\frac{\partial\tilde{f}}{\partial\bar{z}}
	\Tr(M'M^{-1}Q^{(1,p)}) \mathrm{d}^{2}z \bigg| \prec
	(N\eta_{0})^{-1/2}, \quad p=1,2,3.
\end{equation}
We remark that the other part $(1+|\lambda|)^{2}\eta_{0}^{1/2}$ of the bound \eqref{eq: 468} is not needed here.

To get the estimate \eqref{eq: 1379a} for $p=1$, we recall the definition of $Q^{(1,1)}$ from \eqref{eq: 1104e}.
Consider
\begin{equation*}
	\sum_{i,j,a,b,k} (M'M^{-1})_{ji} \kappa_{d}(aj,bk) G_{ib} G_{ka} 
	= \sum_{j,k} (M'M^{-1}G\widetilde{T}_{d}^{(jk)}G)_{jk},
\end{equation*}
where $\widetilde{T}_{d}^{(jk)}(b,a) = \kappa_{d}(aj,bk),$ and $G^{\T} = G$ for the real symmetric case.
We apply the two-resolvent local law \eqref{eq: 882d} to the term associated with $G_{ib}G_{ka}$ and the term associated with $(-\E[G_{ib}G_{ka}])$ separately, together with \eqref{eq: 796}.
Then, by the cancellation between the deterministic leading terms (from the local law),
we get
\begin{equation*}
	|\Tr(M'M^{-1}Q^{(1,1)})|
	\prec \frac{\Psi(z,z)}{N} \sum_{j,k} \lVert \widetilde{T}_{d}^{(jk)}\rVert
	\lesssim \Psi(z,z),
\end{equation*}
where we used $\sum_{j,k} \lVert \widetilde{T}_{d}^{(jk)}\rVert \lesssim N$ for the last step, which is deduced from the assumption \eqref{eq: 245b}.
Then, by Lemma \ref{lem: 358}, we obtain \eqref{eq: 1379a} for $p=1$.

Next we shall prove \eqref{eq: 1379a} for $p=2$. Recall that $Q^{(1,2)}$ is defined in \eqref{eq: 1108e}.
Although $Q^{(1,2)}$ splits into several parts by definition, using $G_{ij}-\E[G_{ij}]=(G_{ij}-M_{ij})-\E[G_{ij}-M_{ij}]$),
it is enough to consider the following form:
\begin{equation*}
	\frac{1}{N} \sum_{i,j,a,b,k} (M'M^{-1})_{ji} \kappa_{c}(aj,bk) (G-M)_{ib} (G-M)_{ka}.
\end{equation*}
We ignore the fact that some $(G-M)_{\alpha}$ factor may come with expectation, since we will have high probability bounds on $(G-M)_{\alpha}$ from the local laws. So $(G-M)_{\alpha}$ and $\E[(G-M)_{\alpha}]$ terms are controled in the same way. No additional gain from expectation is used.

Recall $T^{(jb)}_{c}(a,k) = \kappa_{c}(aj,bk)$.
By the average local law \eqref{eq: 873e},
\begin{equation*}
	\bigg|\frac{1}{N} \sum_{a,k} T^{(jb)}_{c}(a,k) (G-M)_{ka} \bigg|
	\prec \frac{1}{N\eta} \lVert T^{(jb)}_{c} \rVert_{\text{hs}}.
\end{equation*}
By the isotropic local law \eqref{eq: 869e},
\begin{equation*}
	|\sum_{i} (M'M^{-1})_{ji}(G-M)_{ib}| = |(M'M^{-1}(G-M))_{jb}| \prec \frac{1}{\sqrt{N\eta}}.
\end{equation*}
From \eqref{eq: 814}, we deduce $\sum_{j,b} \lVert T^{(jb)}_{c} \rVert_{\text{hs}} \lesssim N$.
Hence the application of Lemma \ref{lem: 358} finishes the proof of \eqref{eq: 1379a} for $p=2$.

What remains is to prove the case $p=3$ of \eqref{eq: 1379a}.
Starting from the first term of \eqref{eq: 1113e}, together with $G_{ib}-\E[G_{ib}]=(G_{ib}-M_{ib})-\E[G_{ib}-M_{ib}]$, we consider
\begin{equation*}
	\frac{1}{N} \sum_{i,j,a,b,k} (M'M^{-1})_{ji} \kappa_{d}(aj,bk) (G-M)_{ib} M_{ka}.
\end{equation*}
Setting $T_{bi} \coloneqq \sum_{j,a,k} (M'M^{-1})_{ji} \kappa_{d}(aj,bk) M_{ka}$, then by \eqref{eq: 873e},
\begin{equation*}
	\Big|\frac{1}{N} \sum_{i,j,a,b,k} (M'M^{-1})_{ji} \kappa_{d}(aj,bk) (G-M)_{ib} M_{ka}\Big|
	\prec \frac{1}{N\eta} \lVert T\rVert_{\text{hs}}.
\end{equation*}
We claim
\begin{equation}\label{eq: 5484}
	\lVert T\rVert_{\text{hs}} \lesssim 1,
\end{equation}
and then
\begin{equation}\label{eq: 1148}
	\Big|\frac{1}{N} \sum_{i,j,a,b,k} (M'M^{-1})_{ji} \kappa_{d}(aj,bk) (G-M)_{ib} M_{ka}\Big|
	\prec \frac{1}{N\eta},
\end{equation}
which provides the contribution from the first term of \eqref{eq: 1113e} to the estimate \eqref{eq: 1379a} for $p=3$ via Lemma \ref{lem: 358}.
We will prove \eqref{eq: 5484} after finishing the proof of \eqref{eq: 1379a} for $p=3$.

Similarly, we find for the contribution from the second term of \eqref{eq: 1113e},
\begin{equation}\label{eq: 1210}
	\Big|\frac{1}{N} \sum_{i,j,a,b,k} (M'M^{-1})_{ji} \kappa_{d}(aj,bk) (G-M)_{ka} M_{ib}\Big|
	\prec \frac{1}{N\eta},
\end{equation}
by setting $\widetilde{T}_{ak} \coloneqq \sum_{i,j,b} (M'M^{-1})_{ji} \kappa_{d}(aj,bk) M_{ib}$ and applying \eqref{eq: 873e} with the following claim
\begin{equation}\label{eq: 1224}
	\lVert \widetilde{T}\rVert_{\text{hs}} \lesssim 1.
\end{equation}
Combining \eqref{eq: 1148} and \eqref{eq: 1210}, we indeed obtain \eqref{eq: 1379a} for $p=3$ by Lemma \ref{lem: 358}, thus have completed the proof of \eqref{eq: 468} for $k=1$.

To show \eqref{eq: 5484}, we define (locally for this proof)
\begin{equation*}
	S_{bj} \coloneqq \sum_{a,k} \kappa_{d}(aj,bk) M_{ka}.
\end{equation*}
Then, $T = S M'M^{-1}$ and hence $\lVert T\rVert_{\text{hs}} \le \lVert M'M^{-1} \rVert \lVert S\rVert_{\text{hs}}.$
Since $\lVert M'M^{-1} \rVert=O(1)$ due to \eqref{eq: 796}, it is enough to bound $\lVert S\rVert_{\text{hs}}$ appropriately.
We observe that for each $b,j\in [N]$, by the assumption \eqref{eq: 245b} and the first estimate of \eqref{eq: 949d},
\begin{equation*}
	\sum_{b,j}|S_{bj}|^2 \lesssim \sum_{b,j}\sum_{a,k} D_{b-a,k-j}|M_{ka}|^{2}
	\lesssim \sum_{a,k}|M_{ka}|^{2} = O(N).
\end{equation*}
Consequently the claim \eqref{eq: 5484} follows.
The proof of \eqref{eq: 1224} is analogous hence omitted.

\end{proof}

\begin{proof}[Proof of \eqref{eq: 468} for $k=2$]

Following \eqref{eq: 1000a}, it is enough to consider
\begin{equation*}
	\int_{\Omega_{0}}\frac{\partial\tilde{f}}{\partial\bar{z}}
	\Tr(M'M^{-1}Q^{(2)}) \mathrm{d}^{2}z,
\end{equation*}
where we have from \eqref{eq: 709a} that
\begin{equation}\label{eq: 1857tg}
	Q^{(2)}_{ij} = \sum_{a,b,k,c,\ell}\frac{\kappa(aj,bk,c\ell)}{2N^{3/2}}
	(\E[\partial_{bk}\partial_{c\ell}(\tilde{e}G_{ia})-\tilde{e}\E[\partial_{bk}\partial_{c\ell}G_{ia}]]).
\end{equation}
Our goal is to prove
\begin{equation}\label{eq: 1645a}
	\bigg| \int_{\Omega_{0}}\frac{\partial\tilde{f}}{\partial\bar{z}}
	\Tr(M'M^{-1}Q^{(2)}) \mathrm{d}^{2}z \bigg| \prec (1+|\lambda|)^{2}\eta_{0}^{1/2},
\end{equation}
where we do not need the other part $(N\eta_{0})^{-1/2}$ of the bound \eqref{eq: 468}.
In the term $\partial_{bk}\partial_{c\ell}(\tilde{e}G_{ia})$, each differential can hit the function $\tilde{e}$ or the resolvent entry $G_{ia}$.
We distinguish several cases.

~

\emph{Case 1.}
We first look at the case that both differentials hit $G_{ia}$, namely,
$$\partial_{bk}\partial_{c\ell}G_{ia}=G_{ic}G_{\ell b}G_{ka}+G_{ib}G_{kc}G_{\ell a}.$$
The contribution of the terms, $G_{ic}G_{\ell b}G_{ka}$ and $G_{ib}G_{kc}G_{\ell a}$,
are estimated very similarly
so we focus on the first one, more precisely,
\begin{equation}\label{eq: 1406}
	\sum_{i,j,a,b,k,c,\ell} 
	(M'M^{-1})_{ji}\frac{\kappa(aj,bk,c\ell)}{N^{3/2}}
	\E[\tilde{e}(G_{ic}G_{\ell b}G_{ka}-\E[G_{ic}G_{\ell b}G_{ka}])].
\end{equation}
Using the decomposition $G=M+(G-M)$ for each factor, we further split \eqref{eq: 1406} into finitely many sub-terms, for example, omitting the expectation $\E[\cdot]$ and the function $\tilde{e}$,
\begin{equation}\label{eq: 1888e}
	\sum_{i,j,a,b,k,c,\ell} (M'M^{-1})_{ji}\frac{\kappa(aj,bk,c\ell)}{N^{3/2}}
	(G-M)_{ic}M_{\ell b}M_{ka},
\end{equation}
and
\begin{equation}\label{eq: 1893e}
	\sum_{i,j,a,b,k,c,\ell}(M'M^{-1})_{ji}\frac{\kappa(aj,bk,c\ell)}{N^{3/2}}
	(G-M)_{ic}(G-M)_{\ell b}M_{ka},
\end{equation}
and there is a third type of terms with three $(G-M)$\,factors.
Note that there is a single obvious cancellation, that is, $M_{ic}M_{\ell b}M_{ka}-\E[M_{ic}M_{\ell b}M_{ka}]=0$.

For \eqref{eq: 1888e}, our calculation proceeds as follows.
Summing over $i$ and swapping some arguments of $\kappa$ appropriately, we have
\begin{equation*}
	\eqref{eq: 1888e} = \frac{1}{N^{3/2}}\sum_{j,a,b,k,c,\ell} \kappa(c\ell,bk,aj) M_{\ell b}M_{ka} (M'M^{-1}(G-M))_{jc}.
\end{equation*}
According to Assumption \ref{assump: 140} (iii) and \eqref{eq: 869e}, we get the upper bound
\begin{equation}\label{eq: 1921tg}
    \bigg| \frac{1}{N^{3/2}}\sum_{j,a,b,k,c,\ell} \kappa(c\ell,bk,aj) M_{\ell b}M_{ka} (M'M^{-1}(G-M))_{jc} \bigg| \\
	\lesssim
	\lVert M\rVert^{2} \lVert M'M^{-1}\rVert \lVert G-M\rVert_{\text{hs}} \prec \eta^{-1/2},
\end{equation}
where we used the isotropic local law \eqref{eq: 869e} to get $\lVert G-M\rVert_{\text{hs}}\prec\eta^{-1/2}$.
Then Lemma \ref{lem: 358} implies
\begin{equation*}
	\bigg| \int_{\Omega_{0}}\frac{\partial\tilde{f}}{\partial\bar{z}}
	\bigg( \frac{1}{N^{3/2}}\sum_{j,a,b,k,c,\ell} \kappa(c\ell,bk,aj) M_{\ell b}M_{ka} \E[(M'M^{-1}(G-M))_{jc}] \bigg)
	\mathrm{d}^{2}z \bigg| \prec \eta_{0}^{1/2}.
\end{equation*}
The same estimate holds for the other similar cases, $M_{ic}(G-M)_{\ell b}M_{ka}$ and $M_{ic}M_{\ell b}(G-M)_{ka}$.

For \eqref{eq: 1893e}, summing over $i$, we get
\begin{equation*}
	\eqref{eq: 1893e} = \frac{1}{N^{3/2}} \sum_{a,j,b,k,c,\ell} \kappa(aj,bk,c\ell)
	(M'M^{-1}(G-M))_{jc}(G-M)_{\ell b}M_{ka}.
\end{equation*}
Note that $|(M'M^{-1}(G-M))_{jc}|\prec (N\eta)^{-1/2}$, $(G-M)_{\ell b}\prec (N\eta)^{-1/2}$ and $|M_{ka}|=O(1)$.
By Assumption \ref{assump: 140} (i),
\begin{equation*}
	\sum_{a,j,b,k,c,\ell}|\kappa(aj,bk,c\ell)|\le N^{2} \triplenorm{\kappa}_{3} \lesssim N^2.
\end{equation*}
Thus we find that
\begin{equation*}
	\bigg| \frac{1}{N^{3/2}} \sum_{a,j,b,k,c,\ell} \kappa(aj,bk,c\ell)
	(M'M^{-1}(G-M))_{jc}(G-M)_{\ell b}M_{ka} \bigg| \prec
	\frac{1}{N^{1/2}\eta},
\end{equation*}
which implies, by Lemma \ref{lem: 358},
\begin{equation*}
	\bigg| \int_{\Omega_{0}}\frac{\partial\tilde{f}}{\partial\bar{z}}
	\bigg( \frac{1}{N^{3/2}} \sum_{a,j,b,k,c,\ell} \kappa(aj,bk,c\ell)
	\E[(M'M^{-1}(G-M))_{jc}(G-M)_{\ell b}]M_{ka} \bigg)
	\mathrm{d}^{2}z \bigg| \prec N^{-1/2}.
\end{equation*}
The other cases,
\begin{equation*}
	(G-M)_{ic}M_{\ell b}(G-M)_{ka}, \quad 
	M_{ic}(G-M)_{\ell b}(G-M)_{ka}, \quad
	(G-M)_{ic}(G-M)_{\ell b}(G-M)_{ka},
\end{equation*}
are estimated analogously.
In summary,
\begin{equation}\label{eq: 1979e}
	\bigg| \int_{\Omega_{0}}\frac{\partial\tilde{f}}{\partial\bar{z}}
	\bigg( \sum_{i,j,a,b,k,c,\ell} 
	(M'M^{-1})_{ji}\frac{\kappa(aj,bk,c\ell)}{N^{3/2}}
	\E[\tilde{e}(\partial_{bk}\partial_{c\ell}G_{ia}-\E[\partial_{bk}\partial_{c\ell}G_{ia}])] \bigg)
	\mathrm{d}^{2}z \bigg| \prec \eta_{0}^{1/2}.
\end{equation}

~

\emph{Case 2.}
Consider the case that one differential hits $G_{ia}$ and the other hits $\tilde{e}$ in \eqref{eq: 1857tg}, for example,
\begin{equation}\label{eq: 1498}
	\sum_{i,j,a,b,k,c,\ell} (M'M^{-1})_{ji} \frac{\kappa(aj,bk,c\ell)}{N^{3/2}}
	\E[(\partial_{bk}{\tilde{e}})(\partial_{c\ell}G_{ia})].
\end{equation}
The other term associated with $(\partial_{c\ell}{\tilde{e}})(\partial_{bk}G_{ia})$ is estimated in the same way hence omitted.
Without taking the expectation in \eqref{eq: 1498}, we have
\begin{equation*}
	\frac{1}{N^{3/2}}\sum_{a,j,b,k,c,\ell} \kappa(aj,bk,c \ell) (\partial_{bk}\tilde{e}) (M'M^{-1}G)_{jc} G_{\ell a}.
\end{equation*}
Note that $|(M'M^{-1}G)_{jc}G_{\ell a}| = O_{\prec}(1)$.
Recall \eqref{eq: 2003e} for $\partial_{bk}\tilde{e}$.
Since $\partial_{w}G_{kb}(w)=(G^{2}(w))_{kb}$, we shall find the desired bound by considering
\begin{equation*}
	\frac{1}{N^{3/2}}\sum_{a,j,b,k,c,\ell} |\kappa(aj,bk,c \ell) (G^{2}(w))_{kb}|.
\end{equation*}
Viewing $\sum_{(aj)}|\kappa(aj,\ast,\ast)|$ as $N^{2}\times N^{2}$ matrix, by Assumption \ref{assump: 140} (i) and \eqref{eq: 869e}, we have
\begin{equation*}
	\bigg|\frac{1}{N^{3/2}}\sum_{a,j,b,k,c,\ell} |\kappa(aj,bk,c \ell) (G^{2}(w))_{kb}|\bigg|
	\le \lVert I\rVert_{\text{hs}} \lVert G^{2}(w)\rVert_{\text{hs}} \le \lVert G(w)\rVert \lVert G(w)\rVert_{\text{hs}}
	\prec (\Im w)^{-3/2},
\end{equation*}
where we used $\lVert G\rVert_{\text{hs}}\le \lVert G-M\rVert_{\text{hs}} + \lVert M\rVert_{\text{hs}}$ recalling \eqref{eq: 1921tg} and below for the bound of $\lVert G-M\rVert_{\text{hs}}$.
Applying of Lemma \ref{lem: 358} twice and
recalling the factor $\lambda$ in \eqref{eq: 2003e},
\begin{equation}\label{eq: 1808a}
	\bigg| \int_{\Omega_{0}}\frac{\partial\tilde{f}}{\partial\bar{z}}
	\bigg( \sum_{i,j,a,b,k,c,\ell} (M'M^{-1})_{ji} \frac{\kappa(aj,bk,c\ell)}{N^{3/2}}
	\E[(\partial_{bk}{\tilde{e}})(\partial_{c\ell}G_{ia})] \bigg)
	\mathrm{d}^{2}z \bigg| \prec |\lambda|\eta_{0}^{1/2}.
\end{equation}

~

\emph{Case 3.}
When both differentials hit the function $\tilde{e}=\tilde{e}(\lambda)$ in \eqref{eq: 1857tg}, we consider
\begin{multline*}
	\sum_{i,j,a,b,k,c,\ell}(M'M^{-1})_{ji}\frac{\kappa(aj,bk,c\ell)}{N^{3/2}}
	\E[(\partial_{bk}\partial_{c\ell}{\tilde{e}})G_{ia}] \\
	= \frac{1}{N^{3/2}}\sum_{a,j,b,k,c,\ell} \kappa(aj,bk,c \ell) \E[(\partial_{bk}\partial_{c\ell}\tilde{e}) (M'M^{-1}G)_{ja}].
\end{multline*}
Note that
\begin{multline*}
	\partial_{bk}\partial_{c\ell}\tilde{e}
	= - \frac{\lambda^{2}\tilde{e}}{\pi^{2}} \int_{\Omega_{0}'}\frac{\partial\tilde{f}}{\partial\bar{w}} \partial_{w}G_{kb}(w) \mathrm{d}^{2}w \int_{\Omega_{0}'}\frac{\partial\tilde{f}}{\partial\bar{w'}} \partial_{w'}G_{\ell c}(w') \mathrm{d}^{2}w' \\
	+ \frac{\mathrm{i}\lambda\tilde{e}}{\pi} \int_{\Omega_{0}'}\frac{\partial\tilde{f}}{\partial\bar{w}} \partial_{w}(G_{kc}(w)G_{\ell b}(w)) \mathrm{d}^{2}w.
\end{multline*}
We find that $|\partial_{bk}\partial_{c\ell}\tilde{e}| = O_{\prec}(1+|\lambda|^{2})$ by the Ward identity;
\begin{equation}\label{eq: 1555}
	\sum_{j}|G_{ij}(E+i\eta)|^{2} = \frac{\Im G_{ii}(E+i\eta)}{\eta},
\end{equation}
and Lemma \ref{lem: 358} regarding the integral over $w$.
Viewing $\sum_{(c\ell)}|\kappa(\ast,\ast,c\ell)|$ as $N^{2}\times N^{2}$ matrix, due to Assumption \ref{assump: 140} (i) and \eqref{eq: 869e},
\begin{equation*}
	\frac{1}{N^{3/2}}\sum_{a,j,b,k,c,\ell} |\kappa(aj,bk,c \ell) (M'M^{-1}G)_{ja}|
	\lesssim \lVert M'M^{-1}\rVert \lVert G\rVert_{\text{hs}}
	\prec \eta^{-1/2}.
\end{equation*}
By Lemma \ref{lem: 358}, we conclude that
\begin{equation}\label{eq: 1852a}
	\bigg| \int_{\Omega_{0}}\frac{\partial\tilde{f}}{\partial\bar{z}}
	\bigg( \sum_{i,j,a,b,k,c,\ell}(M'M^{-1})_{ji}\frac{\kappa(aj,bk,c\ell)}{N^{3/2}}
	\E[(\partial_{bk}\partial_{c\ell}{\tilde{e}})G_{ia}] \bigg)
	\mathrm{d}^{2}z \bigg| \prec (1+|\lambda|)^{2}\eta_{0}^{1/2}.
\end{equation}
Collecting \eqref{eq: 1979e}, \eqref{eq: 1808a}, \eqref{eq: 1852a} and analogous estimates,
the desired result \eqref{eq: 1645a} follows.
This finishes the proof of \eqref{eq: 468} for $k=2$.

\end{proof}

\begin{proof}[Proof of \eqref{eq: 468} for $k\ge3$]

For the case $k=3$, we consider
\begin{equation*}
	\int_{\Omega_{0}}\frac{\partial\tilde{f}}{\partial\bar{z}}
	\Tr(M'M^{-1}Q^{(3)}) \mathrm{d}^{2}z,
\end{equation*}
where we have from \eqref{eq: 709a}
\begin{equation*}
	Q^{(3)}_{ij} = \sum_{a}\sum_{\alpha,\beta,\gamma}\frac{\kappa(aj,\alpha,\beta,\gamma)}{6N^{2}}
	(\E[\partial_{\alpha}\partial_{\beta}\partial_{\gamma}(\tilde{e}G_{ia})-\tilde{e}\E[\partial_{\alpha}\partial_{\beta}\partial_{\gamma}G_{ia}]]).
\end{equation*}
Here we will not exploit a cancellation from $\partial_{\alpha}\partial_{\beta}\partial_{\gamma}(\tilde{e}G_{ia})-\tilde{e}\E[\partial_{\alpha}\partial_{\beta}\partial_{\gamma}G_{ia}]$, so estimate each term separately.
First consider a specific case and then explain how to generalize the argument.

When $\partial_{c\ell}\partial_{dm}$ hits $\tilde{e}$ and $\partial_{bk}$ hits $G_{ia}$ in the term $\E[\partial_{bk}\partial_{c\ell}\partial_{dm}(\tilde{e}G_{ia})]$, we have (omitting the constant factor including the sign and the expectation $\E[\cdot]$)
\begin{multline*}
	\frac{1}{N^{2}}\sum_{i,j,a,b,k,c,\ell,d,m}
	(M'M^{-1})_{ji} \kappa(aj,bk,c\ell,dm)(\partial_{c\ell}\partial_{dm}\tilde{e})G_{ib}G_{ka} \\
	= \frac{1}{N^{2}}\sum_{j,a,b,k,c,\ell,d,m}
	\kappa(aj,bk,c\ell,dm)(\partial_{c\ell}\partial_{dm}\tilde{e}) (M'M^{-1}G)_{jb} G_{ka}.
\end{multline*}
Note that $|(\partial_{c\ell}\partial_{dm}\tilde{e}) (M'M^{-1}G)_{jb} G_{ka}|\prec 1+|\lambda|^{2}$.
Then, due to Assumption \ref{assump: 140} (i),
\begin{equation*}
	\bigg| \frac{1}{N^{2}}\sum_{j,a,b,k,c,\ell,d,m}
	\kappa(aj,bk,c\ell,dm)(\partial_{c\ell}\partial_{dm}\tilde{e}) (M'M^{-1}G)_{jb} G_{ka} \bigg|
	\prec \frac{(1+|\lambda|^{2})}{N^{2}}\times N\times N \lesssim (1+|\lambda|^{2}),
\end{equation*}
where we used that the Euclidean norm of the all-one vector of dimension $N^{2}$ is given by $N$.
Applying Lemma \ref{lem: 358}, one can find that
\begin{equation*}
	\bigg| \int_{\Omega_{0}}\frac{\partial\tilde{f}}{\partial\bar{z}}
	\bigg( \frac{1}{N^{2}}\sum_{i,j,a,b,k,c,\ell,d,m}
	(M'M^{-1})_{ji} \kappa(aj,bk,c\ell,dm)(\partial_{c\ell}\partial_{dm}\tilde{e})G_{ib}G_{ka} \bigg)
	\mathrm{d}^{2}z \bigg| \prec (1+|\lambda|^{2}) \eta_{0}.
\end{equation*}

~

Note that all other cases can be handled in the same way
because any derivative of $\tilde{e}$ of the form $\partial_{\alpha_{1}}\cdots\partial_{\alpha_{p}}\tilde{e}$ is $O_{\prec}(1+|\lambda|^{p})$ by the Ward identity \eqref{eq: 1555} and Lemma \ref{lem: 358}, and we also have $|M'M^{-1}G|_{ij}\prec 1$ and $|G_{ij}|\prec 1$ by the isotropic local law \eqref{eq: 869e}.
The special case that all differentials hit $G_{ia}$ is also included due to the bound $|\tilde{e}|\prec 1$. 
As a result,
\begin{equation*}
	\bigg| \int_{\Omega_{0}}\frac{\partial\tilde{f}}{\partial\bar{z}}
	\Tr(M'M^{-1}Q^{(3)}) \mathrm{d}^{2}z \bigg|
	\prec (1+|\lambda|^{3})\eta_{0}.
\end{equation*}
The other cases $k\in\{4,\dots,L-1\}$ can be handled by the same reasoning, yielding
\begin{equation*}
	\bigg| \int_{\Omega_{0}}\frac{\partial\tilde{f}}{\partial\bar{z}}
	\Tr(M'M^{-1}Q^{(k)}) \mathrm{d}^{2}z \bigg|
	\prec (1+|\lambda|^{k})\eta_{0}.
\end{equation*}
Since $\lambda$ is fixed, the estimate $\eqref{eq: 468}$ holds for every $k\in\{3,\dots,L-1\}$ (as we consider $N$ large enough), and thus the proof ends.

\end{proof}

\begin{proof}[Proof of \eqref{eq: 474}]

Referring to \eqref{eq: 391},
since $|\mathcal{N}(\alpha_{0})|\le N^{\frac{1}{2}-\mu}$ by Assumption \ref{assump: 140} (iv),
we have for sufficiently large $L$,
\begin{equation}\label{eq: 1796}
	|\mathcal{E}^{(L)}_{ij}| \lesssim \frac{1}{N^{2}}.
\end{equation}
Then, using Lemma \ref{lem: 358},
\begin{equation*}
	\bigg| \int_{\Omega_{0}}\frac{\partial\tilde{f}}{\partial\bar{z}}
	\Tr(M'M^{-1}\mathcal{E}^{(L)}) \mathrm{d}^{2}z \bigg| \prec \eta_{0},
\end{equation*}
which provides \eqref{eq: 474}.

\end{proof}

This completes the proof of Lemma \ref{lem: 445}. \qed

\section{Proof of Proposition \ref{prop: 169}} \label{sec: 1776}

Let $\mathcal{K}(z,w)$ be as in \eqref{eq: 723ft}.
We want to show (recalling the restriction to the case $\beta=1$ for brevity)
\begin{equation*}
	\int_{\Omega_{0}}\int_{\Omega_{0}'}\frac{\partial\tilde{f}(z)}{\partial\bar{z}}\frac{\partial\tilde{f}(w)}{\partial\bar{w}}\mathcal{K}(z,w)
	\mathrm{d}^{2}w \mathrm{d}^{2}z
	= \frac{1}{2}\iint_{\mathbb{R}^{2}}\frac{(f(x)-f(y))^{2}}{(x-y)^{2}}\mathrm{d}x\mathrm{d}y
	+ o(1).
\end{equation*}
Set $\eta_{*}\equiv\eta_{*}(N)=N^{-100}$ (where one can replace the number $100$ with any sufficiently large positive real). Define $\Omega_{*}\coloneqq\{z\in\mathbb{C}:|\Im z|\in(\eta_{*},1)\}.$
According to \eqref{eq: 3463} in Lemma \ref{lem: 2129} (to be stated below), we have,
for $\eta=\Im z$ and $\tilde{\eta}=\Im w$,
\begin{equation}\label{eq: 1825}
	|\mathcal{K}(z,w)| \lesssim 
	\begin{cases}
		1, & \eta\tilde{\eta}>0, \\
		(|\eta|+|\tilde{\eta}|)^{-2}, & \eta\tilde{\eta}<0.
	\end{cases}
\end{equation}
Following the same proof as in \cite[Eq.\,(7.17)--(7.20)]{Vova25a}, together with \eqref{eq: 1825}, we obtain
\begin{equation}\label{eq: 2059a}
	\int_{\Omega_{0}}\int_{\Omega_{0}'}\frac{\partial\tilde{f}(z)}{\partial\bar{z}}\frac{\partial\tilde{f}(w)}{\partial\bar{w}}\mathcal{K}(z,w)\mathrm{d}^{2}w\mathrm{d}^{2}z 
	= \int_{\Omega_{*}}\int_{\Omega_{*}}\frac{\partial\tilde{f}(z)}{\partial\bar{z}}\frac{\partial\tilde{f}(w)}{\partial\bar{w}}\mathcal{K}(z,w)\mathrm{d}^{2}w\mathrm{d}^{2}z
	+ O(N^{-\eps_{0}}).
\end{equation}
Since $\mathcal{K}(z,w)=2\partial_{w}\Tr(M'M^{-1}V)$ (recalling \eqref{eq: 531} for $V$), it is enough to show that
\begin{equation}\label{eq: 2390}
    \int_{\Omega_{*}}\int_{\Omega_{*}} \frac{\partial\tilde{f}(z)}{\partial\bar{z}} \frac{\partial\tilde{f}(w)}{\partial\bar{w}} \partial_{w}\Tr(M'M^{-1}V) \mathrm{d}^{2}w\mathrm{d}^{2}z \\
	= \frac{1}{4}\iint_{\mathbb{R}^{2}}\frac{(f(x)-f(y))^{2}}{(x-y)^{2}}\mathrm{d}x\mathrm{d}y
	+ o(1).
\end{equation}

\begin{proposition}[Reduction via integration by parts]\label{prop: 114}
Let $E_{0}$ be as in Theorem \ref{thm: 96}.
There exists a sufficiently small constant $\hat{\eps}>0$ such that
\begin{multline*}
	\int_{\Omega_{*}}\int_{\Omega_{*}} \frac{\partial\tilde{f}(z)}{\partial\bar{z}} \frac{\partial\tilde{f}(w)}{\partial\bar{w}} \partial_{w}\Tr(M'M^{-1}V) \mathrm{d}^{2}w\mathrm{d}^{2}z \\
	= -\frac{1}{8}\int_{E_{0}-\hat{\eps}}^{E_{0}+\hat{\eps}}\int_{E_{0}-\hat{\eps}}^{E_{0}+\hat{\eps}} 
	(f(y)-f(x))^{2} \sum_{(\eta,\tilde{\eta})=\pm(\eta_{*},-\eta_{*})}\partial_{w}\Tr(M'(z)M^{-1}(z)V(z,w)) \mathrm{d}x\mathrm{d}y
	+ o(1),
\end{multline*}
where $z=x+\mathrm{i}\eta$ and $w=y+\mathrm{i}\tilde{\eta}$ on the right-hand side.
\end{proposition}
\begin{proof}
This statement is similar to \cite[Lemma 7.2]{Vova25a}.
From a technical point of view, the only difference in the proof occurs when it comes to getting suitable bounds for the log-determinant
\begin{equation}\label{eq: 2266ft}
	\mathcal{L}(z,w)\coloneqq -\log\det(\mathcal{B}_{z,w}),
\end{equation}
and its derivatives,
which is formulated in Lemma \ref{lem: 2129} below.
Note that $\partial_{z}\partial_{w}\mathcal{L}(z,w)=\frac{1}{2}\mathcal{K}(z,w)$.
The rest of the proof is essentially the same, so we omit the details.
\end{proof}

\begin{proposition}[Kernel estimate]\label{prop: 478}
Let $E_{0}$ and $\eps_{0}$ be as in Theorem \ref{thm: 96}.
Let $\beta_{*}$ be as in Lemma \ref{lem: 523}.
Consider a (small) constant $\hat{\eps}\in(0,\beta_{*}/3)$ such that $[E_{0}-\hat{\eps},E_{0}+\hat{\eps}]\subset\mathcal{I}(\eps_{0}/2)$.
Choose a small constant $\delta\in(0,1)$.
For $x,y\in[E_{0}-\hat{\eps},E_{0}+\hat{\eps}]$ with $|x-y|\ge N^{\delta}\eta_{*}$,
\begin{equation*}
	\partial_{w}\Tr(M'(x\pm\mathrm{i}\eta_{*})M^{-1}(x\pm\mathrm{i}\eta_{*})V(x\pm\mathrm{i}\eta_{*},y\mp\mathrm{i}\eta_{*}))
	= -\frac{1}{(x-y)^{2}} \Big( 1 + O(\hat{\eps}) \Big),
\end{equation*}
where the implicit constant in $O(\hat{\eps})$ is independent of $x$ and $y$.
\end{proposition}

\begin{proposition}[Negligible contribution near the diagonal]\label{prop: 2274a}
Recall $z=x+\mathrm{i}\eta$ and $w=y+\mathrm{i}\tilde{\eta}$.
Let $E_{0}$ be as in Theorem \ref{thm: 96}. Let $\hat{\eps}$ and $\delta$ as in Proposition \ref{prop: 478}.
We have
\begin{equation}\label{eq: 2276f}
	\int_{E_{0}-\hat{\eps}}^{E_{0}+\hat{\eps}}\int_{|x-y|<N^{\delta}\eta_{*}}
	\Big|(f(y)-f(x))^{2} \sum_{(\eta,\tilde{\eta})=\pm(\eta_{*},-\eta_{*})}\partial_{w}\Tr(M'(z)M^{-1}(z)V(z,w))\Big| \mathrm{d}x\mathrm{d}y
	= O(N^{3\delta}\eta_{*}).
\end{equation}
\end{proposition}

Hence, combining these propositions above,
we get the desired estimate \eqref{eq: 2390} because we can choose $\hat{\eps}$ arbitrarily small, $f$ is given by \eqref{eq: 192} with compactly supported $g$, and the contribution of the almost diagonal region $|x-y|<N^{\delta}\eta_{*}$ is negligible.
This completes the proof of Proposition \ref{prop: 169}.
\qed

~

For reference, we provide the following lemma, a necessary technical input for Proposition \ref{prop: 114}.
\begin{lemma}\label{lem: 2129}
Let $z$ and $w$ be as in Lemma \ref{lem: 792}.
Write $z=x+\mathrm{i}\eta$ and $w=y+\mathrm{i}\tilde{\eta}$.
Then, for $\mathcal{L}(z,w)$ defined in \eqref{eq: 2266ft} we have
\begin{equation}\label{eq: 3455}
	\mathcal{L}(z,w) = O(N^{2}), \quad
	|\partial_{z}\mathcal{L}(z,w)|+|\partial_{w}\mathcal{L}(z,w)| \lesssim 1 + \frac{1}{\max\{|x-y|,|\eta|+|\tilde{\eta}|\}}.
\end{equation}
For $|\eta|+|\tilde{\eta}|<1$, we have
\begin{equation}\label{eq: 3463}
	\big|\partial_{z}\partial_{w}\mathcal{L}(z,w)\big|
	\lesssim 
	\begin{cases}
		1, & \eta\tilde{\eta}>0, \\
		(|\eta|+|\tilde{\eta}|)^{-2}, & \eta\tilde{\eta}<0.
	\end{cases}
\end{equation}
\end{lemma}

Remaining technical ingredients, Proposition \ref{prop: 478}--\ref{prop: 2274a}, will be proven in the rest of this section, but we focus more on Proposition \ref{prop: 478} that is one of our main technical contributions.
Proposition \ref{prop: 2274a} (and Lemma \ref{lem: 2129}) will be given as a byproduct after showing Proposition \ref{prop: 478}.

\subsection{Proof of Proposition \ref{prop: 478}: kernel estimate}\label{sec: 2383}

Fix $\hat{\eps}\in(0,\beta_{*}/3)$ satisfying $[E_{0}-\hat{\eps},E_{0}+\hat{\eps}]\subset\mathcal{I}(3\eps_{0}/2)$. Let $x,y\in[E_{0}-\hat{\eps},E_{0}+\hat{\eps}]$.
Recall that we consider
\begin{equation}\label{eq: 2501}
	(z,w)=\pm(x+\mathrm{i}\eta_{*},y-\mathrm{i}\eta_{*}).
\end{equation}
Note that
\begin{equation*}
	\partial_{w}\Tr (M'M^{-1}V) 
	= \Tr_{\textnormal{op}}(\mathcal{C}_{M'M^{-1},I}\mathcal{B}_{z,w}^{-1}\mathcal{C}_{M,\widetilde{M}'}\mathcal{S}\mathcal{B}_{z,w}^{-1}),
\end{equation*}
where we recall $\mathcal{B}_{z,w}=1-\mathcal{C}_{M,\widetilde{M}}\mathcal{S}$ with $M=M(z)$ and $\widetilde{M}=M(w)$.
Let $\Pi_{z,w}$ be the spectral projection corresponding to $\beta_{z,w}$ as in \eqref{eq: 482}.
Note that
\begin{equation}\label{eq: 2317}
	\mathcal{B}_{z,w}^{-1} = \beta_{z,w}^{-1}\Pi_{z,w} + \mathcal{B}_{z,w}^{-1}(1-\Pi_{z,w}).
\end{equation}
Let $R\equiv R(z,w)$ and $L\equiv L(z,w)$ be the right and left eigenvectors of $\mathcal{B}_{z,w}$ such that
$\mathcal{B}_{z,w}[R]=\beta_{z,w}R$, $\mathcal{B}_{z,w}^{*}[L]=\bar{\beta}_{z,w}L$ and $\lVert R\rVert_{\textnormal{hs}}=1=\lVert L\rVert_{\textnormal{hs}}$.
One can write
\begin{equation}\label{eq: 2411}
	\Pi\equiv \Pi_{z,w} = \frac{\langle L, \cdot \rangle}{\langle L, R \rangle}R.
\end{equation}
Then,
\begin{align}
	\Tr_{\textnormal{op}}(\mathcal{C}_{M'M^{-1},I}\mathcal{B}_{z,w}^{-1}\mathcal{C}_{M,\widetilde{M}'}\mathcal{S}&\mathcal{B}_{z,w}^{-1})
	 = \frac{1}{(\beta_{z,w}\langle L, R \rangle)^{2}} \langle L,M'M^{-1}R\rangle \langle L, \mathcal{C}_{M,\widetilde{M}'}\mathcal{S}[R] \rangle \nn \\
	& + \frac{N}{\beta_{z,w}\langle L,R\rangle} \sum_{a,b} \langle E^{ab},M'M^{-1}R\rangle\langle L, \mathcal{C}_{M,\widetilde{M}'}\mathcal{S}\mathcal{B}_{z,w}^{-1}(1-\Pi_{z,w})[E^{ab}]\rangle \nn\\
	& + \frac{N}{\beta_{z,w}\langle L,R\rangle} \sum_{a,b} \langle E^{ab},\mathcal{C}_{M'M^{-1},I}\mathcal{B}_{z,w}^{-1}(1-\Pi_{z,w})\mathcal{C}_{M,\widetilde{M}'}\mathcal{S}[R]\rangle\langle L, E^{ab}\rangle \nn\\
	& + \Tr_{\text{op}}( \mathcal{C}_{M'M^{-1},I}\mathcal{B}^{-1}(1-\Pi_{z,w})\mathcal{C}_{M,\widetilde{M}'}\mathcal{S}\mathcal{B}^{-1}(1-\Pi_{z,w}) ). \label{eq: 2520e}
\end{align}
We claim that the first term in the right-hand side of \eqref{eq: 2520e} is the leading contribution in the mesoscopic regime $|x-y|\lesssim \eta_{0}$, which is formally presented in the following statements.

\begin{claim}\label{claim: 2339}
Let $z$ and $w$ be as in \eqref{eq: 2501}.
For $|x-y|\ge N^{\delta}\eta_{*}$, we have
\begin{equation}\label{eq: 2379}
	\frac{1}{(\beta_{z,w}\langle L, R \rangle)^{2}} \langle L,M'M^{-1}R\rangle \langle L, \mathcal{C}_{M,\widetilde{M}'}\mathcal{S}[R] \rangle
	= -\frac{1}{(x-y)^{2}} \big( 1 + O(\hat{\eps}) \big).
\end{equation}
\end{claim}

\begin{claim}\label{claim: 2343}
Under the same setting of Claim \ref{claim: 2339}, for $|x-y|\ge N^{\delta}\eta_{*}$,
\begin{align}
	\frac{N}{\beta_{z,w}\langle L,R\rangle} \sum_{a,b} \langle E^{ab},M'M^{-1}R\rangle\langle L, \mathcal{C}_{M,\widetilde{M}'}\mathcal{S}\mathcal{B}_{z,w}^{-1}(1-\Pi_{z,w})[E^{ab}]\rangle 
	&= O(|x-y|^{-1}), \label{eq: 2349}\\
	\frac{N}{\beta_{z,w}\langle L,R\rangle} \sum_{a,b} \langle E^{ab},\mathcal{C}_{M'M^{-1},I}\mathcal{B}_{z,w}^{-1}(1-\Pi_{z,w})\mathcal{C}_{M,\widetilde{M}'}\mathcal{S}[R]\rangle\langle L, E^{ab}\rangle 
	&= O(|x-y|^{-1}), \label{eq: 2353}\\
	\Tr_{\textnormal{op}}( \mathcal{C}_{M'M^{-1},I}\mathcal{B}_{z,w}^{-1}(1-\Pi_{z,w})\mathcal{C}_{M,\widetilde{M}'}\mathcal{S}\mathcal{B}_{z,w}^{-1}(1-\Pi_{z,w}) ) 
	&= O(|x-y|^{-1}). \label{eq: 2357}
\end{align}
\end{claim}

Given the above two claims, Proposition \ref{prop: 478} follows directly.
We will show these claims in the rest of this sub-section.

\subsubsection{Proof of Claim \ref{claim: 2339}}

Since $M\mathcal{S}[R]\widetilde{M} = - \{(1-\mathcal{C}_{M,\widetilde{M}}\mathcal{S})-1\}[R] = (1-\beta_{z,w})R$,
\begin{equation*}
	\langle L, \mathcal{C}_{M,\widetilde{M}'}\mathcal{S}[R] \rangle
	= (1-\beta_{z,w})\langle L, R\widetilde{M}^{-1}\widetilde{M}'\rangle.
\end{equation*}
Thus we have
\begin{equation}\label{eq: 2570e}
	\frac{1}{(\beta_{z,w}\langle L, R \rangle)^{2}} \langle L,M'M^{-1}R\rangle \langle L, \mathcal{C}_{M,\widetilde{M}'}\mathcal{S}[R] \rangle
	= \frac{(1-\beta_{z,w})}{(\beta_{z,w}\langle L, R \rangle)^{2}} \langle L,M'M^{-1}R\rangle \langle L, R\widetilde{M}^{-1}\widetilde{M}'\rangle.
\end{equation}

We restrict our proof to the case that
$z=x+\mathrm{i}\eta_{*}$ and $w=y-\mathrm{i}\eta_{*}$ where $\eta_{*}=N^{-100}$.
Recall that $x,y\in[E_{0}-\hat{\eps},E_{0}+\hat{\eps}]$ with $\hat{\eps}\in(0,\beta_{*}/3)$.
We note that $|z-w|\le \beta_{*}$ (for $N$ sufficiently large).
Applying Lemma \ref{lem: 523} with setting $(z_{0},w_{0})=(z,\bar{z})$ and $(z_{1},w_{1})=(z,w)$
(in addition, $R(z,w)\coloneqq R_{1}/\lVert R_{1}\rVert_{\text{hs}}$ and $L(z,w)\coloneqq L_{1}/\lVert L_{1}\rVert_{\text{hs}}$ with $R_{1}$ and $L_{1}$ defined in Lemma \ref{lem: 523}),
\begin{align}
	&\beta_{z,w}  \langle L(z,w), R(z,w)\rangle \nn\\
	&= \beta_{z,\bar{z}}\langle L(z,\bar{z}), R(z,\bar{z})\rangle + \langle L(z,\bar{z}),M(z)\mathcal{S}[R(z,\bar{z})](M(\bar{z})-M(w)) \rangle + O(|x-y|^{2}) \nn\\
	&= \beta_{z,\bar{z}}\langle L(z,\bar{z}), R(z,\bar{z})\rangle + \langle L(z,\bar{z}),\mathcal{C}_{M(z),M(\bar{z})}\mathcal{S}[R(z,\bar{z})] M^{-1}(\bar{z})(M(\bar{z})-M(w)) \rangle + O(|x-y|^{2}) \nn\\
	&= \beta_{z,\bar{z}}\langle L(z,\bar{z}), R(z,\bar{z})\rangle + \big(1-\beta_{z,\bar{z}}\big)\langle L(z,\bar{z}),R(z,\bar{z})M^{-1}(\bar{z})M'(\bar{z})\rangle (x-y) + O(|x-y|^{2}), \label{eq: 2931}
\end{align}
where we used \eqref{eq: 796} and $\lVert M'' \rVert\lesssim 1$ in the bulk (obtained by differentiating $M'$ via \eqref{eq: 671} and using \eqref{eq: 922d}, \eqref{eq: 949d}--\eqref{eq: 954d} and \eqref{eq: 796}).
Similarly, we have
\begin{multline}\label{eq: 2475a}
	\beta_{z,w}\langle L(z,w), R(z,w)\rangle
	= \beta_{\bar{w},w}\langle L(\bar{w},w), R(\bar{w},w)\rangle \\
	+ (1-\beta_{\bar{w},w})\langle L(\bar{w},w),M'(\bar{w})M^{-1}(\bar{w})R(\bar{w},w)\rangle (y-x) + O(|x-y|^{2}).
\end{multline}

In addition, by Lemma \ref{lem: 523} again, we observe that
\begin{multline}\label{eq: 2481a}
	(1-\beta_{z,w})\langle L(z,w), R(z,w)M^{-1}(w)M'(w)\rangle \\
	= (1-\beta_{z,\bar{z}})\langle L(z,\bar{z}),R(z,\bar{z})M^{-1}(\bar{z})M'(\bar{z})\rangle + O(|x-y|),
\end{multline}
and
\begin{equation}\label{eq: 2485a}
	\langle L(z,w),M'(z)M^{-1}(z)R(z,w)\rangle = \langle L(\bar{w},w),M'(\bar{w})M^{-1}(\bar{w})R(\bar{w},w)\rangle + O(|x-y|),
\end{equation}
where we also used \eqref{eq: 796} and $\lVert M'' \rVert\lesssim 1$ in the bulk.

By Lemma \ref{lem: 792}, we have $|\beta_{z,\bar{z}}|\sim\eta_{*}$ and $|\beta_{\bar{w},w}|\sim\eta_{*}$.
Since $|\langle L,R\rangle|\le 1$ and $|x-y|\ge N^{\delta}\eta_{*}$, the first terms in the right-hand sides of \eqref{eq: 2931}--\eqref{eq: 2475a} are negligible compared with the second terms if the coefficients of $(x-y)$ are of order $1$.
Note that the coefficients of the $(x-y)$ terms in \eqref{eq: 2931}--\eqref{eq: 2475a} are the same (modulo sign) as the leading terms of the factors $\langle L,M'M^{-1}R\rangle$ and $(1-\beta_{z,w})\langle L, R\widetilde{M}^{-1}\widetilde{M}'\rangle$ in \eqref{eq: 2570e} by \eqref{eq: 2481a}--\eqref{eq: 2485a}.
Thus, after cancelling these coefficients in \eqref{eq: 2570e}, to establish \eqref{eq: 2379}, it is enough to check the following technical proposition.

\begin{proposition}\label{prop: 2625}
Recall that we fixed $\hat{\eps}\in(0,\beta_{*}/3)$ satisfying $[E_{0}-\hat{\eps},E_{0}+\hat{\eps}]\subset\mathcal{I}(3\eps_{0}/2)$.
Consider $z=x+\mathrm{i}\eta_{*}$ and $w=y-\mathrm{i}\eta_{*}$ where $x,y\in[E_{0}-\hat{\eps},E_{0}+\hat{\eps}]$ and $\eta_{*}=N^{-100}$.
There exists a small constant $c>0$ such that
\begin{equation}\label{eq: 2506gt}
	|\langle L(z,\bar{z}),R(z,\bar{z})M^{-1}(\bar{z})M'(\bar{z})\rangle| > c,
\end{equation}
and, similarly,
\begin{equation*}
	|\langle L(\bar{w},w),M'(\bar{w})M^{-1}(\bar{w})R(\bar{w},w)\rangle| > c.
\end{equation*}
\end{proposition}

Given the above proposition, the proof of Claim \ref{claim: 2339} ends.
\qed

\begin{proof}[Proof of Proposition \ref{prop: 2625}] 

Since the reasoning is identical, we restrict our proof to \eqref{eq: 2506gt}.
Write $R\equiv R(z,w)$.
Note that
\begin{equation*}
	\partial_{w}\mathcal{B}_{z,w}[R] = -\mathcal{C}_{M(z),M'(w)}\mathcal{S}[R]
	=  (\beta_{z,w}-1)R M^{-1}(w)M'(w),
\end{equation*}
since $\mathcal{C}_{M(z),M(w)}\mathcal{S}[R] = (1-\beta_{z,w})R$.
Taking $\langle L,\cdot\rangle$,
\begin{equation*}
	\langle L,\partial_{w}\mathcal{B}_{z,w}[R]\rangle = (\beta_{z,w}-1)\langle L,R M^{-1}(w)M'(w)\rangle.
\end{equation*}
We further observe that, by taking $\partial_{w}$ and $\langle L,\cdot\rangle$ to $\mathcal{B}_{z,w}[R] = \beta_{z,w} R$,
\begin{equation*}
	\langle L,(\partial_{w}\mathcal{B}_{z,w})[R]\rangle = (\partial_{w}\beta_{z,w}) \langle L,R\rangle.
\end{equation*}
Thus,
\begin{equation*}
	\langle L,R M^{-1}(w)M'(w)\rangle
	= \frac{(\partial_{w}\beta_{z,w})\langle L,R\rangle}{\beta_{z,w}-1}.
\end{equation*}
Since $|\beta_{z,\bar{z}}-1|\sim 1$ by \eqref{eq: 937d} in Lemma \ref{lem: 792}, it is enough to consider $\partial_{w}\beta_{z,w}$ and $\langle L,R\rangle$ for the case $w=\bar{z}$, i.e.,
we need to show
\begin{equation}\label{eq: 2608a}
	\big|(\partial_{w}\beta_{z,w})|_{w=\bar{z}}\big|\gtrsim 1, 
\end{equation}
and
\begin{equation}\label{eq: 2612a}
	|\langle L(z,\bar{z}),R(z,\bar{z})\rangle|\gtrsim 1.
\end{equation}

To show \eqref{eq: 2608a}, we first notice the following relation from \eqref{eq: 197}:
\begin{equation}\label{eq: 2617}
	\mathcal{B}_{z,w}[M(z)-M(w)] = (z-w)M(z)M(w).
\end{equation}
Taking $\langle L,\cdot\rangle$ and using $\mathcal{B}^{*}[L]=\bar{\beta}_{z,w}L$,
we get
\begin{equation*}
	\beta_{z,w}\langle L,M(z)-M(w)\rangle = (z-w)\langle L,M(z)M(w)\rangle,
\end{equation*}
which implies
\begin{equation*}
	\beta_{z,w} = (z-w) \frac{\langle L,M(z)M(w)\rangle}{\langle L,M(z)-M(w)\rangle}.
\end{equation*}
Fix $z$ and define
\begin{equation}\label{eq: 3475}
	F(w) \coloneqq \frac{\langle L,M(z)M(w)\rangle}{\langle L,M(z)-M(w)\rangle} = \frac{\beta_{z,w}}{z-w}.
\end{equation}
Then,
\begin{equation*}
	\partial_{w}\beta_{z,w} = -F(w) + (z-w)F'(w).
\end{equation*}
For $w=\bar{z}$,
\begin{equation*}
	\partial_{w}\beta_{z,\bar{z}} = -F(\bar{z}) + (z-\bar{z})F'(\bar{z}) = -F(\bar{z}) + (2\mathrm{i}\eta_{*})F'(\bar{z}).
\end{equation*}
In order to show $|\partial_{w}\beta_{z,\bar{z}}|\sim 1$, we first check that $F(\bar{z})$ is of order $1$.
Using \eqref{eq: 3475} with $w=\bar{z}$, the estimate \eqref{eq: 937d} implies $|F(\bar{z})|\sim 1$.
Then, what remains is to show that $(2\mathrm{i}\eta)F'(\bar{z})$ is small, so the following claim ends the proof of \eqref{eq: 2608a}.
\begin{claim}\label{claim: 2917a}
We have $F'(\bar{z})=O(1)$.
\end{claim}

Now we will prove \eqref{eq: 2612a}.
Using \eqref{eq: 2617} with $w=\bar{z}$, we have
\begin{equation}\label{eq: 2780}
	\mathcal{B}_{z,\bar{z}}[\Im M(z)] = \eta_{*} M(z)M^{*}(z).
\end{equation}
As $\eta_{*}$ is small, inspired by \eqref{eq: 937d} and \eqref{eq: 2780}, one expects that $\Im M(z)$ is approximated by the right eigenvector $R$, which is indeed verified in the following claim.
 
\begin{claim}\label{claim: 3252}
We have
\begin{equation}\label{eq: 2995e}
	\Im M(z) = \gamma R(z,\bar{z}) + E_{\eta_{*}},
\end{equation}
where
\begin{equation*}
	\gamma = \frac{\langle L(z,\bar{z}),\Im M(z)\rangle}{\langle L(z,\bar{z}),R(z,\bar{z})\rangle},
\end{equation*}
and the error term $E_{\eta_{*}}$ satisfies $\lVert E_{\eta_{*}}\rVert_{\textnormal{hs}}=O(\eta_{*})$.
Moreover,
\begin{equation}\label{eq: 2813a}
	\langle L(z,\bar{z}), \Im M(z)\rangle \gtrsim 1
	\quad\text{and}\quad \gamma \sim 1.
\end{equation}
\end{claim}

Applying \eqref{eq: 2995e} of Claim \ref{claim: 3252}, we get
\begin{equation*}
	\langle L(z,\bar{z}),R(z,\bar{z})\rangle 
	= \gamma^{-1}\langle L(z,\bar{z}), \Im M(z)\rangle
	+ \gamma^{-1}\langle L(z,\bar{z}), E_{\eta_{*}}\rangle,
\end{equation*}
and the desired estimate \eqref{eq: 2612a} follows from \eqref{eq: 2813a}.
Thus we completed the proof of Proposition \ref{prop: 2625} assuming Claim \ref{claim: 2917a}--\ref{claim: 3252} (to be proven below).

\end{proof}

Since $\mathcal{S}$ and $\mathcal{C}_{M,M^{*}}$ are positivity-preserving, by the Krein--Rutman theorem (a generalization of the Perron--Frobenius theorem), we assume $R(z,\bar{z})$ and $L(z,\bar{z})$ are positive definite, without loss of generality, in the rest of our argument.
Before showing of Claim \ref{claim: 2917a}--\ref{claim: 3252}, we state the following lemma as a prerequisite.

\begin{lemma}\label{lem: 2809}
Set $L(z,\bar{z})$ to be positive definite without loss of generality.
There exists a small constant $c'>0$ such that
\begin{equation}\label{eq: 3601}
	L(z,\bar{z}) \ge c'I.
\end{equation}
\end{lemma}
\begin{proof}
Set $L\equiv L(z,\bar{z})$ for brevity.
Recall the flatness condition \eqref{eq: 294a}.
From $\mathcal{B}_{z,\bar{z}}^{*}[L]=\bar{\beta}_{z,\bar{z}}L$, we derive $\mathcal{S}\mathcal{C}_{M^{*},M}[L]=(1-\beta_{z,\bar{z}})L$.
Using the flatness and $L\ge 0$,
\begin{equation*}
	(1-\beta_{z,\bar{z}})L = \mathcal{S}[M^{*}LM] \le C_{\text{flat}}\langle M^{*}LM \rangle 
	\lesssim \langle L \rangle,
\end{equation*}
where we used Lemma \ref{lem: 354} for the last inequality.
After squaring, we get
\begin{equation*}
	(1-\beta_{z,\bar{z}})^{2}L^{2} \lesssim \langle L \rangle^{2} \quad\text{and hence}\quad
	(1-\beta_{z,\bar{z}})^{2}\langle L^{2} \rangle \lesssim \langle L \rangle^{2}.
\end{equation*}
Recall that $\langle L^{2} \rangle=\langle L^{*}L \rangle=\lVert L\rVert_{\text{hs}}^{2}=1$ and $1-\beta_{z,\bar{z}}= 1+o(1)$. Thus we have
\begin{equation}\label{eq: 3571}
	\langle L(z,\bar{z}) \rangle \gtrsim 1.
\end{equation}

Again, using the flatness, but in the other direction for this time,
\begin{equation*}
	(1-\beta_{z,\bar{z}})L = \mathcal{S}[M^{*}LM] \ge c_{\text{flat}}\langle M^{*}LM \rangle = c_{\text{flat}}\langle MM^{*}L \rangle
	\ge c_{\text{flat}}\lambda_{\min}(MM^{*}) \langle L\rangle.
\end{equation*}
Due to the fact $\lambda_{\min}(MM^{*})=\lVert M^{-1}\rVert^{-2}$
and Lemma \ref{lem: 354},
we have
\begin{equation*}
	(1-\beta_{z,\bar{z}})L \gtrsim \langle L\rangle.
\end{equation*}
Combining with \eqref{eq: 3571}, we obtain \eqref{eq: 3601}.
\end{proof}

We first prove Claim \ref{claim: 3252} and then Claim \ref{claim: 2917a}.

\begin{proof}[Proof of Claim \ref{claim: 3252}]

Let $\Pi\equiv \Pi_{z,\bar{z}}$ be the spectral projection as in \eqref{eq: 2411} for $w=\bar{z}$.
We set $R\equiv R(z,\bar{z})$ and $L\equiv L(z,\bar{z})$.
By defining $Q\coloneqq 1-\Pi$, we have
\begin{equation*}
	\Im M = \Pi[\Im M] + Q[\Im M] = \gamma R + E_{\eta_{*}},
\end{equation*}
where $M\equiv M(z)$ and $E_{\eta_{*}} \coloneqq Q[\Im M].$
In addition, for $\mathcal{B}\equiv \mathcal{B}_{z,\bar{z}}$,
\begin{equation*}
	\mathcal{B}[\Im M] = \mathcal{B}[\gamma R + E_{\eta_{*}}] = \gamma\beta_{z,\bar{z}} R + \mathcal{B}[E_{\eta_{*}}].
\end{equation*}
By \eqref{eq: 2780},
\begin{equation*}
	\gamma\beta_{z,\bar{z}} R + \mathcal{B}[E_{\eta_{*}}] = \eta_{*} MM^{*}.
\end{equation*}
Applying the operator $Q$,
\begin{equation*}
	Q[\gamma\beta_{z,\bar{z}} R + \mathcal{B}[E_{\eta_{*}}]] = \eta_{*} Q[MM^{*}].
\end{equation*}
Note that
\begin{equation*}
	Q[\gamma\beta_{z,\bar{z}} R + \mathcal{B}[E_{\eta_{*}}]] = Q\mathcal{B}[E_{\eta_{*}}] = \mathcal{B}Q[E_{\eta_{*}}]
	= \mathcal{B}Q^{2}[\Im M] = \mathcal{B}Q[\Im M] = \mathcal{B}[E_{\eta_{*}}].
\end{equation*}
Thus,
\begin{equation*}
	\mathcal{B}[E_{\eta_{*}}] = \eta_{*} Q[MM^{*}],
\end{equation*}
and, we deduce that
\begin{equation*}
	E_{\eta_{*}} = \eta_{*} \mathcal{B}^{-1}Q[MM^{*}].
\end{equation*}
By Corollary \ref{cor: 953} and Lemma \ref{lem: 354}, we get
\begin{equation*}
	\lVert E_{\eta_{*}}\rVert_{\textnormal{hs}} = O(\eta_{*}),
\end{equation*}
which completes the proof of \eqref{eq: 2995e}.

Next we shall prove \eqref{eq: 2813a}.
Note that we have the lower bound for $L(z,\bar{z})$ by Lemma \ref{lem: 2809} and thus the first part of \eqref{eq: 2813a} is obtained;
\begin{equation}\label{eq: 2829a}
	\langle L(z,\bar{z}), \Im M(z)\rangle \ge c'\langle \Im M(z)\rangle \gtrsim \rho(x) \gtrsim 1,
\end{equation}
where we used the assumption $x\in[E_{0}-\hat{\eps},E_{0}+\hat{\eps}]\subset\mathcal{I}(3\eps_{0}/2)$ (essentially meaning that $x$ is inside the spectral bulk).
For the second part of \eqref{eq: 2813a}, we use \eqref{eq: 2995e};
\begin{equation*}
	\lVert \Im M(z)\rVert_{\text{hs}}
	= \gamma + o(1).
\end{equation*}
Since $|\langle L, \Im M(z) \rangle|\le \lVert \Im M(z)\rVert_{\text{hs}} \le \lVert M\rVert_{\text{hs}},$
the second part of \eqref{eq: 2813a} follows from \eqref{eq: 949d} and \eqref{eq: 2829a}.

\end{proof}

\begin{proof}[Proof of Claim \ref{claim: 2917a}] 
Fix $z$ and define (locally for this proof)
\begin{equation*}
	N(w)\coloneqq \langle \Pi_{z,w}^{*}[L(z,\bar{z})], M(z)M(w)\rangle, \quad
	D(w)\coloneqq \langle \Pi_{z,w}^{*}[L(z,\bar{z})], M(z)-M(w)\rangle.
\end{equation*}
By Lemma \ref{lem: 523}, we observe that $L(z,w)=\Pi_{z,w}^{*}[L(z,\bar{z})]/\lVert \Pi_{z,w}^{*}[L(z,\bar{z})]\rVert_{\text{hs}}$ and $\lVert \Pi_{z,w}^{*}[L(z,\bar{z})]\rVert_{\text{hs}}\sim 1$ in a neighborhood of $w=\bar{z}$.
Thus we have the identity $F(w)=N(w)/D(w)$ near $w=\bar{z}$.
By the quotient rule, we get $F'(w) = \mathcal{R}_{1}(w) + \mathcal{R}_{2}(w)$ where
\begin{align}
	\mathcal{R}_{1}(w) &= \frac{\langle \partial_{\bar{w}}\Pi_{z,w}^{*}[L(z,\bar{z})], M(z)M(w) \rangle D(w) - N(w) \langle \partial_{\bar{w}}\Pi_{z,w}^{*}[L(z,\bar{z})], M(z) - M(w) \rangle}{D^{2}(w)}, \label{eq: 2870qw}\\
	\mathcal{R}_{2}(w) &= \frac{\langle \Pi_{z,w}^{*}[L(z,\bar{z})], M(z)M'(w) \rangle D(w) - N(w) \langle \Pi_{z,w}^{*}[L(z,\bar{z})], M(z) - M'(w) \rangle}{D^{2}(w)}. \label{eq: 2871qw}
\end{align}
We need to estimate $\mathcal{R}_{i}(w)$ for $w=\bar{z}$.
We have $D(\bar{z})\gtrsim 1$ because, setting $w=\bar{z}$ and recalling $\Im M=(M-M^{*})/2\mathrm{i}$, the first part of the estimate \eqref{eq: 2813a} implies
\begin{equation*}
	|\langle L(z,\bar{z}),\Pi_{z,\bar{z}}[M(z)-M(\bar{z})]\rangle| \gtrsim 1.
\end{equation*}
Moreover, using \eqref{eq: 572a}--\eqref{eq: 482}, \eqref{eq: 949d} and \eqref{eq: 796}, all numerators and denominators in \eqref{eq: 2870qw}--\eqref{eq: 2871qw} are $O(1)$ respectively.
Therefore $|F'(\bar{z})|$ is bounded as desired.
\end{proof}

\subsubsection{Proof of Claim \ref{claim: 2343}}\label{sec: 2944} 

We need the following technical lemma to obtain Claim \ref{claim: 2343}.

\begin{lemma}\label{lem: 2587}
Consider $\hat{\eps}\in(0,\beta_{*}/3)$ satisfying $[E_{0}-\hat{\eps},E_{0}+\hat{\eps}]\subset\mathcal{I}(3\eps_{0}/2)$.
Let $$(z,w)=\pm(x+\mathrm{i}\eta_{*},y-\mathrm{i}\eta_{*}),$$ with $x,y\in[E_{0}-\hat{\eps},E_{0}+\hat{\eps}]$ and $\eta_{*}=N^{-100}$.
For $|x-y|\ge N^{\delta}\eta_{*}$, we have
\begin{equation*}
	|\beta_{z,w}\langle L, R\rangle| \sim |x-y|.
\end{equation*}
\end{lemma}
\begin{proof}
This result is a consequence of the estimate \eqref{eq: 2931}, \eqref{eq: 937d} and Proposition \ref{prop: 2625}.
\end{proof}

From now on we will give the proof of Claim \ref{claim: 2343} by showing the estimates \eqref{eq: 2349}--\eqref{eq: 2357}.

\begin{proof}[Proof of \eqref{eq: 2349} and \eqref{eq: 2353}]
We only prove \eqref{eq: 2349} for concision. The other estimate \eqref{eq: 2353} follows by essentially the same argument.
Note that
\begin{equation*}
	\langle L, \mathcal{C}_{M,\widetilde{M}'}\mathcal{S}\mathcal{B}_{z,w}^{-1}(1-\Pi)[E^{ab}]\rangle
	= \langle (\mathcal{B}_{z,w}^{-1}(1-\Pi))^{*}\mathcal{S}\mathcal{C}_{M^{*},(\widetilde{M}')^{*}}[L], E^{ab}\rangle.
\end{equation*}
We have
\begin{multline*}
	\sum_{a,b} N \langle E^{ab},M'M^{-1}R\rangle\langle (\mathcal{B}_{z,w}^{-1}(1-\Pi))^{*}\mathcal{S}\mathcal{C}_{M^{*},(\widetilde{M}')^{*}}[L], E^{ab}\rangle \\
	= \langle (\mathcal{B}_{z,w}^{-1}(1-\Pi_{z,w}))^{*}\mathcal{S}\mathcal{C}_{M^{*},(\widetilde{M}')^{*}}[L],M'M^{-1}R \rangle \\
	\le \lVert (\mathcal{B}_{z,w}^{-1}(1-\Pi_{z,w}))^{*}\mathcal{S}\mathcal{C}_{M^{*},(\widetilde{M}')^{*}}[L]\rVert_{\text{hs}} \lVert M'M^{-1}R\rVert_{\text{hs}} = O(1),
\end{multline*}
where we used Corollary \ref{cor: 953} for the last step.
Applying Lemma \ref{lem: 2587}, we obtain \eqref{eq: 2349}.

\end{proof}

\begin{proof}[Proof of \eqref{eq: 2357}]

By the cyclic invariance of the trace,
\begin{multline*}
	\Tr_{\text{op}}( \mathcal{C}_{M'M^{-1},I}\mathcal{B}_{z,w}^{-1}(1-\Pi_{z,w})\mathcal{C}_{M,\widetilde{M}'}\mathcal{S}\mathcal{B}_{z,w}^{-1}(1-\Pi_{z,w}) ) \\
	= \Tr_{\text{op}}( \mathcal{B}_{z,w}^{-1}(1-\Pi_{z,w})\mathcal{C}_{M'M^{-1},I}\mathcal{B}_{z,w}^{-1}(1-\Pi_{z,w})\mathcal{C}_{M,\widetilde{M}'}\mathcal{S} ).
\end{multline*}
We will use the decomposition $\mathcal{B}_{z,w}^{-1}= 1 + \mathcal{B}_{z,w}^{-1}\mathcal{C}_{M,\widetilde{M}}\mathcal{S}.$
Then,
\begin{equation}\label{eq: 3299e}
	\mathcal{B}_{z,w}^{-1}(1-\Pi_{z,w})\mathcal{C}_{M'M^{-1},I}\mathcal{B}_{z,w}^{-1}(1-\Pi_{z,w})\mathcal{C}_{M,\widetilde{M}'}\mathcal{S}
	= \mathcal{D}_{1} + \mathcal{D}_{2} + \mathcal{D}_{3},
\end{equation}
where
\begin{align*}
	\mathcal{D}_{1} &\coloneqq (1-\Pi_{z,w})\mathcal{C}_{M'M^{-1},I}(1-\Pi_{z,w})\mathcal{C}_{M,\widetilde{M}'}\mathcal{S}, \\
	\mathcal{D}_{2} &\coloneqq \mathcal{B}_{z,w}^{-1}\mathcal{C}_{M,\widetilde{M}}\mathcal{S} (1-\Pi_{z,w})\mathcal{C}_{M'M^{-1},I}(1-\Pi_{z,w})\mathcal{C}_{M,\widetilde{M}'}\mathcal{S}, \\
	\mathcal{D}_{3} &\coloneqq \mathcal{B}_{z,w}^{-1}(1-\Pi_{z,w})\mathcal{C}_{M'M^{-1},I} \mathcal{B}_{z,w}^{-1}\mathcal{C}_{M,\widetilde{M}}\mathcal{S} (1-\Pi_{z,w})\mathcal{C}_{M,\widetilde{M}'}\mathcal{S}.
\end{align*}
We will estimate each part separately.
For the first part $\Tr_{\text{op}}( \mathcal{D}_{1} )$,
we claim that
\begin{equation}\label{eq: 2786}
	\Tr_{\text{op}}( \mathcal{D}_{1} ) = O(1).
\end{equation}
Note that $\mathcal{D}_{1} = \mathcal{D}_{1,1} + \mathcal{D}_{1,2} + \mathcal{D}_{1,3}$ where
\begin{align}
	\mathcal{D}_{1,1} \coloneqq \mathcal{C}_{M'M^{-1},I}\mathcal{C}_{M,\widetilde{M}'}\mathcal{S}, \qquad
	\mathcal{D}_{1,2} \coloneqq - \Pi_{z,w} \mathcal{C}_{M'M^{-1},I}\mathcal{C}_{M,\widetilde{M}'}\mathcal{S}, \nn\\
	\mathcal{D}_{1,3} \coloneqq - (1-\Pi_{z,w})\mathcal{C}_{M'M^{-1},I}\Pi_{z,w} \mathcal{C}_{M,\widetilde{M}'}\mathcal{S}. \qquad\qquad \label{eq: 2690}
\end{align}
We first estimate $\Tr_{\text{op}}(\mathcal{D}_{1,1})$
and then bound the remaining part $\Tr_{\text{op}}(\mathcal{D}_{1,2}+\mathcal{D}_{1,3})$.
Recalling that $(\mathcal{S}[E^{ab}])_{xy} = N^{-1}\kappa(xa,by)$, we get
\begin{equation*}
	\Tr_{\text{op}}(\mathcal{D}_{1,1})
	= \sum_{a,b} N \langle E^{ab}, \mathcal{C}_{M'M^{-1},I}\mathcal{C}_{M,\widetilde{M}'}\mathcal{S}[E^{ab}] \rangle \\
	= \frac{1}{N} \sum_{a,b}\sum_{x,y} \widetilde{M}'_{yb}M'_{ax}\kappa(xa,by).
\end{equation*}
Since we have the decomposition $\kappa=\kappa_{c}+\kappa_{d}$, it follows that
\begin{equation}\label{eq: 3113a}
	\frac{1}{N} \sum_{a,b}\sum_{x,y} \widetilde{M}'_{yb} M'_{ax}\kappa_{c}(xa,by) \\
	= \frac{1}{N} \sum_{r,s}\sum_{a,x} (\widetilde{M}')^{(r,s)}_{xa} M'_{ax} \kappa_{c}(xa,(a+s)(x+r)),
\end{equation}
where $(\widetilde{M}')^{(r,s)}$ is obtained from $\widetilde{M}'$ by setting $(\widetilde{M}')^{(r,s)}_{xa}\coloneqq\widetilde{M}'_{(x+r)(a+s)}$, i.e., shifting rows and columns with respect to $r$ and $s$ (modulo $N$ respectively) with $|r|,|s|\le N/2$.
In \eqref{eq: 3113a}, we used the notational convention for $(a+s)$ and $(x+r)$ as in \eqref{eq: 241b}--\eqref{eq: 245b}.

Define the matrix $K^{(r,s)}$ by $K^{(r,s)}(a,x)\coloneqq M'_{ax} \kappa_{c}(xa,(a+s)(x+r))$.
Then,
\begin{equation*}
	\eqref{eq: 3113a}
	= \sum_{r,s} \langle \overline{(\widetilde{M}')^{(r,s)}}, K^{(r,s)} \rangle.
\end{equation*}
Together with the assumption \eqref{eq: 241b}, we observe
\begin{equation}\label{eq: 3257a}
	|\langle \overline{(\widetilde{M}')^{(r,s)}}, K^{(r,s)} \rangle| \le \lVert \overline{(\widetilde{M}')^{(r,s)}}\rVert_{\text{hs}} \lVert K^{(r,s)}\rVert_{\text{hs}} \lesssim C_{r,s},
\end{equation}
and thus it follows that
\begin{equation*}
	\bigg| \sum_{r,s} \langle \overline{(\widetilde{M}')^{(r,s)}}, K^{(r,s)} \rangle \bigg| \lesssim 1.
\end{equation*}
Similarly, via the same argument but with the assumption \eqref{eq: 245b} instead of \eqref{eq: 241b},
\begin{equation}\label{eq: 3271a}
    \bigg| \frac{1}{N} \sum_{r,s}\sum_{a,x} (\widetilde{M}')^{(s,r)}_{ax} M'_{ax}\kappa_{d}(xa,(x+r)(a+s)) \bigg| \lesssim 1,
\end{equation}
which implies
\begin{equation}\label{eq: 2771}
	\Tr_{\text{op}}(\mathcal{D}_{1,1})
	= O(1).
\end{equation}

For the remaining parts $\Tr_{\text{op}}(\mathcal{D}_{1,2}+\mathcal{D}_{1,3})$,
\begin{multline*}
	\Tr_{\text{op}}( \Pi_{z,w} \mathcal{C}_{M'M^{-1},I}\mathcal{C}_{M,\widetilde{M}'}\mathcal{S} )
	= \frac{1}{\langle L,R\rangle} \sum_{a,b} N\langle E^{ab}, R\rangle \langle L,\mathcal{C}_{M'M^{-1},I}\mathcal{C}_{M,\widetilde{M}'}\mathcal{S}[E^{ab}]\rangle \\ 
	\lesssim \lVert \mathcal{S}\mathcal{C}_{M^{*},(\widetilde{M}')^{*}}\mathcal{C}_{(M'M^{-1})^{*},I}[L]\rVert_{\text{hs}} \lVert R\rVert_{\text{hs}},
\end{multline*}
and similarly,
\begin{multline*}
	\Tr_{\text{op}}( (1-\Pi)\mathcal{C}_{M'M^{-1},I}\Pi \mathcal{C}_{M,\widetilde{M}'}\mathcal{S} )
	= \frac{1}{\langle L,R\rangle} \sum_{a,b} N\langle E^{ab}, (1-\Pi)\mathcal{C}_{M'M^{-1},I}[R]\rangle \langle L, \mathcal{C}_{M,\widetilde{M}'}\mathcal{S}[E^{ab}]\rangle \\
	\lesssim \lVert \mathcal{S}\mathcal{C}_{M^{*},(\widetilde{M}')^{*}}[L]\rVert_{\text{hs}} \lVert (1-\Pi)\mathcal{C}_{M'M^{-1},I}[R]\rVert_{\text{hs}},
\end{multline*}
where we used $|\langle L,R\rangle| \sim 1$ (by \eqref{eq: 937d} and Lemma \ref{lem: 2587}) for both cases.
Thus,
\begin{equation}\label{eq: 2808}
	\Tr_{\text{op}}(\mathcal{D}_{1,2}+\mathcal{D}_{1,3})
	= O(1).
\end{equation}
In summary, combining \eqref{eq: 2690}, \eqref{eq: 2771} and \eqref{eq: 2808}, we indeed obtain the claim \eqref{eq: 2786}.

Next we shall show
\begin{equation}\label{eq: 3205}
    \Tr_{\text{op}}( \mathcal{D}_{2} )
	= O(|x-y|^{-1}).
\end{equation}
Introduce the decomposition $\mathcal{D}_{2}=\mathcal{D}_{2,1}+\mathcal{D}_{2,2}$ where
\begin{align*}
	\mathcal{D}_{2,1} &\coloneqq \mathcal{B}_{z,w}^{-1}\mathcal{C}_{M,\widetilde{M}}\mathcal{S} \mathcal{C}_{M'M^{-1},I}\mathcal{C}_{M,\widetilde{M}'}\mathcal{S}, \\
	\mathcal{D}_{2,2} &\coloneqq -\mathcal{B}_{z,w}^{-1}\mathcal{C}_{M,\widetilde{M}}\mathcal{S} (1-\Pi_{z,w})\mathcal{C}_{M'M^{-1},I} \Pi_{z,w}\mathcal{C}_{M,\widetilde{M}'}\mathcal{S} -\mathcal{B}_{z,w}^{-1}\mathcal{C}_{M,\widetilde{M}}\mathcal{S} \Pi_{z,w}\mathcal{C}_{M'M^{-1},I} \mathcal{C}_{M,\widetilde{M}'}\mathcal{S}.
\end{align*}
For the terms with the rank-one projection $\Pi_{z,w}$, we apply the argument that we used to derive \eqref{eq: 2808} in a similar fashion. Then,
\begin{equation}\label{eq: 2803}
    \Tr_{\text{op}}( \mathcal{D}_{2,2} )
	= O(|x-y|^{-1}),
\end{equation}
where the factor $|x-y|^{-1}$ originated from $\mathcal{B}_{z,w}^{-1}$.
We omit the detail.

We now focus on estimating
\begin{equation}\label{eq: 3455e}
    \Tr_{\text{op}}( \mathcal{D}_{2,1} )
	= \sum_{a,b} \Big(\mathcal{B}_{z,w}^{-1}\mathcal{C}_{M,\widetilde{M}}\mathcal{S} \mathcal{C}_{M'M^{-1},I}\mathcal{C}_{M,\widetilde{M}'}\mathcal{S}[E^{ab}]\Big)_{ab}.
\end{equation}
For the second $\mathcal{S}$ in \eqref{eq: 3455e}, write $\mathcal{S}=\mathcal{S}_{c}+\mathcal{S}_{d}$, where
\begin{equation*}
	\mathcal{S}_{\#}[E^{ab}]
	= \sum_{x,y}\frac{1}{N}\kappa_{\#}(xa,by)E^{xy}, \quad \# = c,d.
\end{equation*}
Then, for the $\mathcal{S}_{c}$-part, we write
\begin{equation*}
	\mathcal{S}_{c}[E^{ab}]
	= \frac{1}{N}\sum_{s}\sum_{x}\kappa_{c}(xa,(a+r)(x+s))E^{x(x+s)},
\end{equation*}
where we apply the same convention for $r$ and $s$ as in \eqref{eq: 3113a}.
For the $\mathcal{S}_{c}$-part of \eqref{eq: 3455e},
\begin{multline}\label{eq: 2892}
	\sum_{a,b} \Big(\mathcal{B}_{z,w}^{-1}\mathcal{C}_{M,\widetilde{M}}\mathcal{S} \mathcal{C}_{M'M^{-1},I}\mathcal{C}_{M,\widetilde{M}'}\mathcal{S}_{c}[E^{ab}]\Big)_{ab} \\
	= \frac{1}{N} \sum_{r,s} \sum_{a} \bigg(\mathcal{B}_{z,w}^{-1}\mathcal{C}_{M,\widetilde{M}}\mathcal{S} \mathcal{C}_{M'M^{-1},I}\mathcal{C}_{M,\widetilde{M}'}\Big[\sum_{x}\kappa_{c}(xa,(a+r)(x+s))E^{x(x+s)}\Big]\bigg)_{a(a+r)}\\
	= O(|x-y|^{-1}),
\end{multline}
where we used \eqref{eq: 937d}, \eqref{eq: 954d}, and (by recalling \eqref{eq: 241b})
\begin{equation*}
	\Big\lVert \sum_{x}\kappa_{c}(xa,(a+r)(x+s))E^{x(x+s)} \Big\rVert_{\text{hs}}
	\le \max_{x,a} |\kappa_{c}(xa,(a+r)(x+s))| \le C_{r,s}.
\end{equation*}
Analogously, for the $\mathcal{S}_{d}$-part of \eqref{eq: 3455e},
\begin{multline}\label{eq: 2911}
	\sum_{a,b} \Big(\mathcal{B}_{z,w}^{-1}\mathcal{C}_{M,\widetilde{M}}\mathcal{S} \mathcal{C}_{M'M^{-1},I}\mathcal{C}_{M,\widetilde{M}'}\mathcal{S}_{d}[E^{ab}]\Big)_{ab} \\
	= \frac{1}{N} \sum_{r,s} \sum_{a,b} \kappa_{d}((b+s)a,b(a+r))(\mathcal{B}_{z,w}^{-1}\mathcal{C}_{M,\widetilde{M}}\mathcal{S} \mathcal{C}_{M'M^{-1},I}\mathcal{C}_{M,\widetilde{M}'}[E^{(b+s)(a+r)}])_{ab},
\end{multline}
where we also apply the same convention as in \eqref{eq: 3113a}.
Note that
\begin{equation*}
	(\mathcal{C}_{M'M^{-1},I}\mathcal{C}_{M,\widetilde{M}'}[E^{(b+s)(a+r)}])_{ij} = (M'E^{(b+s)(a+r)}\widetilde{M}')_{ij} = (M')_{i(b+s)}(\widetilde{M}')_{(a+r)j}.
\end{equation*}
Then,
\begin{equation}\label{eq: 3176gt}
	\mathcal{S}\mathcal{C}_{M'M^{-1},I}\mathcal{C}_{M,\widetilde{M}'}[E^{(b+s)(a+r)}]
	= \frac{1}{N} \sum_{i,j} \sum_{x,y} (M')_{i(b+s)}(\widetilde{M}')_{(a+r)j} \kappa(xi,jy)E^{xy}.
\end{equation}
Using the decomposition $\kappa = \kappa_{c} + \kappa_{d}$ again in \eqref{eq: 3176gt}
and shifting the indices,
we write the contribution of the $\kappa_{c}$-part from \eqref{eq: 3176gt} to the right-hand side of \eqref{eq: 2911} as follows:
\begin{multline*}
    \frac{1}{N} \sum_{r,s} \sum_{a,b} \kappa_{d}((b+s)a,b(a+r)) \\
	\times \Big(\mathcal{B}_{z,w}^{-1}\mathcal{C}_{M,\widetilde{M}} \Big[\frac{1}{N} \sum_{r',s'} \sum_{i} (M')_{i(b+s)}(\widetilde{M}')_{(a+r)(i+s')} \Big(\sum_{x}\kappa_{c}(xi,(i+s')(x+r'))E^{x(x+r')}\Big) \Big] \Big)_{ab} \\
	= \frac{1}{N^{2}} \sum_{r,s} \sum_{r',s'} \sum_{a,b,i} \kappa_{d}((b+s)a,b(a+r)) (M')_{i(b+s)}(\widetilde{M}')_{(a+r)(i+s')}\\
	\times \Big(\mathcal{B}_{z,w}^{-1}\mathcal{C}_{M,\widetilde{M}}\Big[\sum_{x}\kappa_{c}(xi,(i+s')(x+r'))E^{x(x+r')}\Big]\Big)_{ab}.
\end{multline*}
Denote $K_{ab}\coloneqq \kappa_{d}((b+s)a,b(a+r))$, $u_{a}\coloneqq (\widetilde{M}')_{(a+r)(i+s')}$,
$v_{b}\coloneqq (M')_{i(b+s)}$, and $$H_{ab}\coloneqq \Big(\mathcal{B}_{z,w}^{-1}\mathcal{C}_{M,\widetilde{M}}\Big[\sum_{x}\kappa_{c}(xi,(i+s')(x+r'))E^{x(x+r')}\Big]\Big)_{ab}.$$
Since
\begin{equation}\label{eq: 2949}
	\sum_{a,b} |K_{ab} u_{a} H_{ab} v_{b}|
	\le (\sum_{a,b}|K_{ab}H_{ab}|^{2})^{1/2} (\sum_{a}|u_{a}|^{2})^{1/2} (\sum_{b}|v_{b}|^{2})^{1/2},
\end{equation}
we conclude that
\begin{multline}\label{eq: 2956}
	\frac{1}{N^{2}} \sum_{r,s} \sum_{r',s'} \sum_{a,b,i} \kappa_{d}((b+s)a,b(a+r)) (M')_{i(b+s)}(\widetilde{M}')_{(a+r)(i+s')}\\
	\times \Big(\mathcal{B}_{z,w}^{-1}\mathcal{C}_{M,\widetilde{M}}\Big[\sum_{x}\kappa_{c}(xi,(i+s')(x+r'))E^{x(x+r')}\Big]\Big)_{ab} 
	= O(|x-y|^{-1}N^{-1/2}),
\end{multline}
where we used $\sqrt{N}\lVert \sum_{x}\kappa_{c}(xi,(i+s')(x+r'))E^{x(x+r')}\rVert_{\text{hs}} = O(\sqrt{N})$ and the assumption \eqref{eq: 245b}.
We estimate the contribution of the $\kappa_{d}$-part from \eqref{eq: 3176gt} to the right-hand side of \eqref{eq: 2911} analogously.
Combining \eqref{eq: 2892} and \eqref{eq: 2956}, we indeed have
\begin{equation}\label{eq: 2981}
	\Tr_{\text{op}}( \mathcal{D}_{2,1} )
	= O(|x-y|^{-1}).
\end{equation}
The estimate \eqref{eq: 3205} is obtained by \eqref{eq: 2803} and \eqref{eq: 2981}.

The last piece that we need to estimate is $\Tr_{\text{op}}( \mathcal{D}_{3} )$.
Together with Corollary \ref{cor: 953}, the same reasoning for \eqref{eq: 2808} and \eqref{eq: 2981} gives us
\begin{equation}\label{eq: 2996}
    \Tr_{\text{op}}( \mathcal{D}_{3} )
	= O(|x-y|^{-1}).
\end{equation}
Therefore, recalling \eqref{eq: 3299e}, we established \eqref{eq: 2357} by \eqref{eq: 2786}, \eqref{eq: 3205} and \eqref{eq: 2996}. 
\end{proof}

Since the estimates \eqref{eq: 2349}--\eqref{eq: 2357} are now proven,
we completed the proof of Claim \ref{claim: 2343}.
\qed

\subsection{Proof of Proposition \ref{prop: 2274a} and Lemma \ref{lem: 2129}}

\begin{proof}[Proof of Proposition \ref{prop: 2274a}]
One can regard this proposition as a corollary of Proposition \ref{prop: 478}.
First, from the proof of Proposition \ref{prop: 478}, we deduce for $|x-y|<N^{\delta}\eta_{*}$,
\begin{equation*}
	\partial_{w}\Tr(M'(x\pm\mathrm{i}\eta_{*})M^{-1}(x\pm\mathrm{i}\eta_{*})V(x\pm\mathrm{i}\eta_{*},y\mp\mathrm{i}\eta_{*})) = O(\eta_{*}^{-2}).
\end{equation*}
Since $|f(x)-f(y)|^2=O(N^{2\delta}\eta_{*}^2)$ for $|x-y|<N^{\delta}\eta_{*}$, we obtain the desired estimate \eqref{eq: 2276f}.
\end{proof}

\begin{proof}[Proof of Lemma \ref{lem: 2129}]
Denote $\mathcal{B}\coloneqq 1-\mathcal{C}_{M,\widetilde{M}}\mathcal{S}$ with $M\equiv M(z)$ and $\widetilde{M}\equiv\widetilde{M}(w)$ for brevity.
Let $\{\lambda_{i}\}$ be the eigenvalues of $\mathcal{B}^{*}\mathcal{B}$. We have
\begin{equation*}
	\log|\det(\mathcal{B})| = \frac{1}{2}\log\det(\mathcal{B}^{*}\mathcal{B}) = \frac{1}{2}\sum_{i}\log\lambda_{i}
	\le \frac{1}{2}\Tr_{\text{op}}(\mathcal{B}^{*}\mathcal{B}-1)
	\lesssim N^2,
\end{equation*}
where we used $\log t\le t-1$ for $t>0$ in the third step, and Lemma \ref{lem: 354} to get $\lVert \mathcal{B} \rVert_{\lVert\cdot\rVert_{\text{hs}}\to\lVert\cdot\rVert_{\text{hs}}}=O(1)$ in the last step.
On the other hand,
\begin{equation*}
	\log|\det(\mathcal{B})| = \frac{1}{2}\sum_{i}\log\lambda_{i} \ge \frac{N}{2}\log\lambda_{\min}(\mathcal{B}^{*}\mathcal{B}).
\end{equation*}
By \eqref{eq: 913d} in Lemma \ref{lem: 792}, $\lambda_{\min} \gtrsim \eta_{*}^{2}.$
Recalling $\eta_{*}=N^{-100}$, $\log|\det(\mathcal{B})| \ge - C N\log N$
for some large $C>0$. Thus,
\begin{equation*}
	\log|\det(1-\mathcal{C}_{M,\widetilde{M}}\mathcal{S})| = O(N^{2}),
\end{equation*}
which provides a bound for $|\mathcal{L}(z,w)|$, the first part of \eqref{eq: 3455}.

Recall
\begin{equation*}
	\partial_{z}\mathcal{L}(z,w) = \Tr_{\text{op}}(\mathcal{C}_{M'M^{-1},I}(1-\mathcal{C}_{M,\widetilde{M}}\mathcal{S})^{-1}\mathcal{C}_{M,\widetilde{M}}\mathcal{S}),
\end{equation*}
\begin{equation*}
	\partial_{w}\mathcal{L}(z,w) = \Tr_{\text{op}}(\mathcal{C}_{I,\widetilde{M}^{-1}\widetilde{M}'}(1-\mathcal{C}_{M,\widetilde{M}}\mathcal{S})^{-1}\mathcal{C}_{M,\widetilde{M}}\mathcal{S}),
\end{equation*}
and
\begin{equation*}
	\frac{\partial^{2}}{\partial w\partial z}\mathcal{L}(z,w)
	= \Tr_{\textnormal{op}}(\mathcal{C}_{M'M^{-1},I}(1-\mathcal{C}_{M,\widetilde{M}}\mathcal{S})^{-1}\mathcal{C}_{M,\widetilde{M}'}\mathcal{S}(1-\mathcal{C}_{M,\widetilde{M}}\mathcal{S})^{-1}).
\end{equation*}
The estimate \eqref{eq: 3463} follows from the proof of Proposition \ref{prop: 478} as a byproduct.

Indeed, consider $\partial_{z}\mathcal{L}(z,w)$ and recall the decomposition \eqref{eq: 2317}.
Then,
\begin{equation*}
    \partial_{z}\mathcal{L}(z,w)
	= \beta_{z,w}^{-1}\Tr_{\text{op}}( \mathcal{C}_{M'M^{-1},I}\Pi_{z,w}\mathcal{C}_{M,\widetilde{M}}\mathcal{S} )
	+ \Tr_{\text{op}}( \mathcal{C}_{M'M^{-1},I}\mathcal{B}_{z,w}^{-1}(1-\Pi_{z,w})\mathcal{C}_{M,\widetilde{M}}\mathcal{S} ).
\end{equation*}
Repeating the argument in the proof of \eqref{eq: 2379} and \eqref{eq: 2357},
we can obtain
\begin{equation*}
	|\partial_{z}\mathcal{L}(z,w)| = O(\beta_{z,w}^{-1}).
\end{equation*}
Then the $\partial_{z}\mathcal{L}$-part of the second estimate in \eqref{eq: 3455} follows immediately from \eqref{eq: 937d} of Lemma \ref{lem: 792}
We skip the details for brevity because the reasoning is very similar.
The $\partial_{w}\mathcal{L}$-part of the second estimate in \eqref{eq: 3455} is shown analogously.
\end{proof}

\subsection{Proof of Lemma \ref{lem: 1009}}\label{sec: 3491}

We consider $\partial_{w}\Tr_{\text{op}}( \mathcal{C}_{M'M^{-1},I} \mathcal{B}_{z,w}^{-1} \mathcal{C}_{M,\widetilde{M}}\mathcal{S}_{d} )$.
Note that
\begin{equation}\label{eq: 4008ft}
    \partial_{w}(\mathcal{C}_{M'M^{-1},I} \mathcal{B}_{z,w}^{-1} \mathcal{C}_{M,\widetilde{M}}\mathcal{S}_{d}) \\
    =
	\mathcal{C}_{M'M^{-1},I} \mathcal{B}_{z,w}^{-1} \mathcal{C}_{M,\widetilde{M}'}\mathcal{S} \mathcal{B}_{z,w}^{-1} \mathcal{C}_{M,\widetilde{M}}\mathcal{S}_{d}
	+ \mathcal{C}_{M'M^{-1},I} \mathcal{B}_{z,w}^{-1} \mathcal{C}_{M,\widetilde{M}'}\mathcal{S}_{d}.
\end{equation}
First of all, consider the case that $|z-w|\le\beta_{*}$ and $\eta_{1}\eta_{2}<0$, where $\lVert \mathcal{B}_{z,w}^{-1} \rVert_{\lVert\cdot\rVert_{\textnormal{hs}}\to\lVert\cdot\rVert_{\textnormal{hs}}}$ is not well-bounded; thus this case of \eqref{eq: 4008ft} contains the singular part.
Recalling the decomposition of $\mathcal{B}_{z,w}^{-1}$ from \eqref{eq: 2317}, the most singular part is
\begin{multline*}
	\beta_{z,w}^{-2}\Tr_{\text{op}}( \mathcal{C}_{M'M^{-1},I} \Pi_{z,w} \mathcal{C}_{M,\widetilde{M}'}\mathcal{S} \Pi_{z,w} \mathcal{C}_{M,\widetilde{M}}\mathcal{S}_{d} ) \\
	= \frac{1}{\beta_{z,w}^{2}\langle L,R \rangle^{2}}
	\Big( \sum_{a,b} N\langle E^{ab}, M'M^{-1}R \rangle \langle L, \mathcal{C}_{M,\widetilde{M}}\mathcal{S}_{d}[E^{ab}]\rangle \Big)
	\times \langle L, \mathcal{C}_{M,\widetilde{M}'}\mathcal{S}[R] \rangle.
\end{multline*}
We observe that
\begin{align*}
	\sum_{a,b} N\langle E^{ab}, M'M^{-1}R \rangle &\langle L, \mathcal{C}_{M,\widetilde{M}}\mathcal{S}_{d}[E^{ab}]\rangle
	= \frac{1}{N^{2}} \sum_{x,y} \sum_{a,b} (M'M^{-1}R)^{\mathsf{T}}_{ba} (M^{^{*}}L\widetilde{M}^{*})^{*}_{yx} \kappa_{d}(xa,by) \\
	&\le \frac{1}{N} \sum_{r,s} \bigg|\frac{1}{N} \sum_{a,b} \kappa_{d}((b+r)a,b(a+s)) (M'M^{-1}R)^{\mathsf{T}}_{ba} (M^{^{*}}L\widetilde{M}^{*})^{*}_{a+s,b+r} \bigg| \\
	&= \frac{1}{N} \sum_{r,s} |\langle K^{(rs)}\circ (M^{^{*}}L\widetilde{M}^{*})^{(s,r)}, (M'M^{-1}R)^{\mathsf{T}} \rangle|.
\end{align*}
Here $r,s$ range over $|r|,|s|\leq N/2$, with the same convention as in
\eqref{eq: 3113a}, and
\begin{equation*}
        (K^{(rs)})^{*}_{ab}
        \coloneqq
        \kappa_{d}((b+r)a,b(a+s)).
\end{equation*}
Moreover, $(M^{*}L\widetilde{M}^{*})^{(s,r)}$ denotes the matrix obtained from
$M^{*}L\widetilde{M}^{*}$ by shifting the rows and columns,
and $K^{(rs)}\circ (M^{*}L\widetilde{M}^{*})^{(s,r)}$ denotes the entrywise
product.

Since shifting row/column does not affect the Hilbert--Schmidt norm, by the assumption \eqref{eq: 245b}
\begin{equation*}
	|\langle K^{(rs)}\circ (M^{^{*}}L\widetilde{M}^{*})^{(s,r)}, (M'M^{-1}R)^{\mathsf{T}} \rangle|
	\le \lVert K^{(rs)}\circ (M^{^{*}}L\widetilde{M}^{*})^{(s,r)} \rVert_{\text{hs}}
	\lVert (M'M^{-1}R)^{\mathsf{T}} \rVert_{\text{hs}} \lesssim D_{r,s}.
\end{equation*}
Together with $\sum_{r,s} D_{r,s} \lesssim 1$ (from \eqref{eq: 245b}) and \eqref{eq: 937d}, we have
\begin{equation*}
	\Big| \beta_{z,w}^{-2}\Tr_{\text{op}}( \mathcal{C}_{M'M^{-1},I} \Pi_{z,w} \mathcal{C}_{M,\widetilde{M}'}\mathcal{S} \Pi_{z,w} \mathcal{C}_{M,\widetilde{M}}\mathcal{S}_{d} ) \Big|
	\lesssim N^{-1} (1+|\eta|^{-1}) (1+|\tilde{\eta}|^{-1}).
\end{equation*}
By Lemma \ref{lem: 358}, we get
\begin{equation*}
	\int_{\Omega_{0}}\int_{\Omega_{0}'}\frac{\partial\tilde{f}(z)}{\partial\bar{z}}\frac{\partial\tilde{f}(w)}{\partial\bar{w}}
	\beta_{z,w}^{-2}\Tr_{\text{op}}( \mathcal{C}_{M'M^{-1},I} \Pi_{z,w} \mathcal{C}_{M,\widetilde{M}'}\mathcal{S} \Pi_{z,w} \mathcal{C}_{M,\widetilde{M}}\mathcal{S}_{d} ) 
	\mathrm{d}^{2}w \mathrm{d}^{2}z
	= o(1).
\end{equation*}
The remaining (less singular) terms from the decomposition of $\mathcal{B}_{z,w}^{-1}$ are estimated similarly and thus we have
\begin{equation}\label{eq: 4057ft}
	\Big| \partial_{w}\Tr_{\text{op}}( \mathcal{C}_{M'M^{-1},I} \mathcal{B}_{z,w}^{-1} \mathcal{C}_{M,\widetilde{M}}\mathcal{S}_{d} ) \Big|
	\lesssim N^{-1} (1+|\eta|^{-1}) (1+|\tilde{\eta}|^{-1}) + 1 + |\eta|^{-1} + |\tilde{\eta}|^{-1}.
\end{equation}
We omit the details because these steps are essentially identical with those of Section \ref{sec: 2944}.

For the other case that $\eta_{1}\eta_{2}>0$ or $|z-w|>\beta_{*}$ (no singular part domain), following the argument of \eqref{eq: 2771} and \eqref{eq: 2981} analogously, we find that
\begin{equation}\label{eq: 4065ft}
	\partial_{w}\Tr_{\text{op}}( \mathcal{C}_{M'M^{-1},I} \mathcal{B}_{z,w}^{-1} \mathcal{C}_{M,\widetilde{M}}\mathcal{S}_{d} ) = O(1),
\end{equation}
where we used $\mathcal{B}_{z,w}^{-1}= 1 + \mathcal{B}_{z,w}^{-1}\mathcal{C}_{M,\widetilde{M}}\mathcal{S}$
and $\lVert \mathcal{B}_{z,w}^{-1} \rVert_{\lVert\cdot\rVert_{\textnormal{hs}}\to\lVert\cdot\rVert_{\textnormal{hs}}} = O(1)$ from \eqref{eq: 913d}.
Combining \eqref{eq: 4057ft} and \eqref{eq: 4065ft}, the proof is now completed by Lemma \ref{lem: 358}.
\qed

\bibliographystyle{abbrv}

\end{document}